\documentclass[12pt]{article}

\usepackage[colorlinks, linkcolor=blue, citecolor=blue]{hyperref}           
\usepackage{color}
\usepackage{graphicx,subfigure,amsmath,amssymb,amsfonts,bm,epsfig,epsf,url,dsfont}
\usepackage{amsthm}

\newtheorem{thm}{Theorem}[section]
\newtheorem{lem}[thm]{Lemma}
\newtheorem{assum}[thm]{Assumption}
\newtheorem{proposition}[thm]{Proposition}
\newtheorem{remark}[thm]{Remark}
\newtheorem{definition}[thm]{Definition}

\newtheorem{model}[thm]{Model}
\newtheorem{example}[thm]{Example}

\usepackage{tikz}
\usepackage{bbm}
\usepackage[small,bf]{caption}
\usepackage{natbib}
\usepackage[top=1in,bottom=1in,left=1in,right=1in]{geometry}
\usepackage{fancybox}
\usepackage{hyperref}
\usepackage{algorithm}
\usepackage{verbatim}

\usepackage{threeparttable}
\usepackage[titletoc,toc]{appendix}

\usepackage{mathtools}

\usepackage{scalerel,stackengine}
\stackMath
\newcommand\reallywidehat[1]{%
\savestack{\tmpbox}{\stretchto{%
  \scaleto{%
    \scalereL_[\widthof{\ensuremath{#1}}]{\kern-.6pt\bigwedge\kern-.6pt}%
    {\rule[-\textheight/2]{1ex}{\textheight}}
  }{\textheight}%
}{0.5ex}}%
\stackon[1pt]{#1}{\tmpbox}%
}

\newcommand*{\rom}[1]{\expandafter\@slowromancap\romannumeral #1@}

\DeclareMathOperator*{\argmax}{arg max}
\DeclareMathOperator*{\argmin}{arg min}

\newcommand{\norm}[1]{\left\|#1\right\|}

\usepackage{xspace}

\numberwithin{equation}{section}

\allowdisplaybreaks
\usepackage[utf8]{inputenc}
\usepackage{authblk}
\usepackage{url}
\usepackage[left=1in,right=1in,top=1in,bottom=1in]{geometry}

\usepackage{titlesec}
\titleformat{\part}[display]
  {\centering\bfseries\LARGE}
  {}{0pt}{}
\titlespacing*{\part}{0pt}{1.0em}{1.0em}

\usepackage{booktabs}
\usepackage{array}
\usepackage{enumitem}
\newcolumntype{L}[1]{>{\raggedright\arraybackslash}p{#1}}
\setlist[itemize]{leftmargin=*,nosep,topsep=0pt,parsep=0pt,partopsep=0pt}

\makeatletter
\renewcommand\paragraph{\@startsection{paragraph}{4}{\z@}%
  {1ex}
  {-1em}
  {\normalfont\normalsize\bfseries}}
\makeatother

\begin{document}

\title{Sample Complexity of Weak-Signal Latent-Variable Models}

\author[1]{Amnon Balanov\thanks{Corresponding author: \url{amnonba15@gmail.com}}}
\author[1]{Tamir Bendory}
\author[2]{Dan Edidin}

\affil[1]{\normalsize Department of Electrical and Computer Engineering,
Technion---Israel Institute of Technology, Haifa 3200003, Israel}

\affil[2]{Department of Mathematics, University of Missouri, Columbia, MO 65211, USA}

\maketitle

\begin{abstract}
We study parametric estimation in latent-variable models in the weak-signal regime, where a small signal-strength parameter $\beta$ controls the dependence of the observations on the unknown parameter. At $\beta=0$, this dependence disappears, so all parameter values induce the same observed distribution, which we call the null law. We develop a unified framework for analyzing sample complexity in the weak-signal regime by expanding the likelihood ratio around this shared null law in powers of $\beta$. We prove that if the unknown parameter is first globally identifiable at order $r$ in this expansion, then estimation to fixed accuracy requires at least on the order of $\beta^{-2r}$ samples. We further identify a likelihood-separation stability condition that is necessary for any estimator to attain this rate and, under mild regularity conditions, sufficient for both maximum likelihood and method-of-moments estimators to attain it. We also study the regime in which the target accuracy vanishes together with the signal strength, and show that global localization and local refinement may be governed by different likelihood orders, leading to distinct statistical scales. Finally, we apply our framework to broad classes of algebraically structured latent-variable models, including additive-noise channels and exponential-family models.
\end{abstract}

\newpage
\setcounter{tocdepth}{2}
\tableofcontents

\newpage

\section{Introduction}

Latent-variable models provide a common framework for statistical inference from noisy observations whose distribution depends on the unknown signal through an unobserved label, transformation, or nuisance variable. Examples include finite mixture models, where the latent variable selects an unobserved component label~\cite{mclachlan2019finite,lindsay1989moment}; factor analysis models, where continuous latent factors explain dependencies among observed variables~\cite{bartholomew2011latent}; orbit recovery problems, where it encodes an unknown shift, rotation, or more general group action~\cite{bendory2017bispectrum,abbe2018multireference,bandeira2023estimation,bandeira2020optimal}; and imaging problems such as cryo-electron microscopy, where it represents an unknown viewing direction and possibly additional nuisance transformations~\cite{bendory2020single}.

A central difficulty arises in weak-signal regimes. As the signal strength tends to zero, the observed law approaches a null law that no longer depends on the unknown parameter. In this regime, the latent variable is typically unrecoverable from an individual observation, even though the underlying parameter may still be estimated from sufficiently many samples. This phenomenon is familiar in high-noise alignment and orbit-recovery problems, where single-sample recovery of the latent transformation may be impossible while aggregate information still permits recovery of the underlying signal~\cite{aguerrebere2016fundamental,bandeira2023estimation,serov2026interplay}. This raises a fundamental question that guides this study:
\begin{quote}
How many samples are  required for reliable recovery in weak-signal latent-variable models?
\end{quote}

In several Gaussian weak-signal models, it is now well understood that sample complexity is governed by the lowest moment order that identifies, up to intrinsic symmetries, the unknown parameter. This principle has been established in settings such as Gaussian mixture models and more general orbit-recovery problems in the presence of Gaussian noise~\cite{abbe2018multireference,bandeira2023estimation,katsevich2023likelihood}. A central contribution of this work is to extend this analysis beyond Gaussian observations, developing a general likelihood-based framework applicable to broad classes of latent-variable models, including non-Gaussian additive-noise channels and exponential-family models. 

The key to this generalization is that the underlying sample-complexity mechanism can be characterized directly through a likelihood-ratio expansion around a common null law, without relying on Gaussian-specific properties. To describe this mechanism, let $\beta>0$ denote a signal-strength parameter tending to zero in the weak-signal regime, with $\beta=0$ corresponding to a null model in which the observed distribution no longer depends on the unknown parameter. If two parameter values first differ at order $r$ in the likelihood expansion, then their likelihood-ratio difference is of order $\beta^r$, while the corresponding statistical divergences are of order $\beta^{2r}$. Distinguishing such a pair therefore requires on the order of $\beta^{-2r}$ samples. When order $r$ is the lowest order required for global identification, this mechanism suggests the sample-complexity scale $\beta^{-2r}$ for fixed-accuracy recovery.

However, two distinct issues arise when translating this principle into a general sample-complexity characterization. First, the fact that the parameter becomes identifiable from the likelihood expansion through a given order does not by itself determine how well the corresponding observed distributions are separated at finite $\beta$. Different orders of the expansion are combined in the observed law and may partially cancel when comparing different parameter values. Consequently, two parameters that are distinguishable from the expansion through order $r$ may nevertheless induce distributions whose separation is smaller than order $\beta^r$. Thus, the identifying order provides an information-theoretic lower bound on the sample complexity, but this bound need not be tight without additional conditions.

Second, fixed-accuracy recovery does not distinguish global localization from fine local resolution. The order of the likelihood expansion required to localize the parameter over the full parameter space may be strictly larger than the order required to distinguish infinitesimal perturbations once the parameter has been localized to a fixed neighborhood. When the target accuracy is allowed to shrink together with the signal strength, these two notions can lead to different statistical scales.

This work develops a likelihood-based framework that separates these phenomena. We expand the observed likelihood ratio around the common zero-signal law and use its Taylor coefficients, which we call the \emph{likelihood moments}, as intrinsic statistical coordinates. For fixed target accuracy $\varepsilon$, we define the \emph{global identifiability order} $r_{\mathrm{glob}}$ as the lowest order of the expansion that globally distinguishes observationally distinct parameters. We show that this order determines the information-theoretic scale $\beta^{-2r_{\mathrm{glob}}}$. Attaining this scale by any estimator additionally requires stable separation of the finite-$\beta$ likelihoods: well-separated parameters must remain statistically distinguishable at the scale predicted by $r_{\mathrm{glob}}$, without cancellations between different orders of the expansion. This stability condition is information-theoretically necessary and, under suitable regularity conditions, sufficient for attaining the predicted sample-complexity rate using two classical estimation paradigms for latent-variable models: maximum likelihood and the method of moments.
We develop these results in Sections~\ref{sec:general_framework}-\ref{sec:rate-optimal-recovery}.

We then study the joint regime in which both the signal strength $\beta$ and the target accuracy $\varepsilon$ tend to zero. For algebraically structured likelihood-moment maps, we define the \emph{local identifiability order} $r_{\mathrm{loc}}$ as the smallest order for which the differential of the likelihood moments up to that order is injective. Thus, $r_{\mathrm{loc}}$ is the first order that detects every infinitesimal perturbation of the parameter. Generically, $r_{\mathrm{loc}}\leq r_{\mathrm{glob}}$, with strict inequality possible. The resulting sample complexity has two components: global localization costs order $\beta^{-2r_{\mathrm{glob}}}$ samples, whereas refinement to a small accuracy $\varepsilon$ costs order $\beta^{-2r_{\mathrm{loc}}}\varepsilon^{-2}$ samples. Hence, the dominant statistical difficulty changes with the target resolution, producing a global-to-local transition. We develop this notion and its statistical consequences in Section~\ref{sec:global_to_local_algebraic}.

Finally, in Section~\ref{sec:applications}, we give verifiable conditions for this likelihood geometry in a broad class of algebraic linear perturbation models, including additive-noise channels and exponential-family models.
This extends weak-signal moment-order sample-complexity laws beyond the Gaussian observation models in which they have primarily been studied.

\subsection{The weak-signal latent-variable model}
\label{subsec:weak_signal_model_intro}

We now formalize the weak-signal setting used throughout the paper. Let $\theta^\star\in\Theta$ be the unknown deterministic parameter, where $\Theta\subseteq\mathbb{R}^m$ is a compact set. We observe $n$ i.i.d.\ samples $Y_1,\ldots,Y_n$, with $Y_i\in\mathcal{Y}\subseteq\mathbb{R}^{d}$. Each observation also depends on an unobserved latent nuisance variable $Z_i\in\mathcal{Z}$. The latent space $\mathcal{Z}$ may be discrete, continuous, or hybrid.

For a candidate parameter $\theta\in\Theta$ and signal-strength parameter $\beta\ge 0$, the conditional observation model is
\begin{align}
    \label{eq:conditional_observation_density_intro}
    Y_i\vert Z_i=z \sim p_{\theta,\beta}(\cdot\vert z),
\end{align}
where $p_{\theta,\beta}(\cdot\vert z)$ is the conditional density. Since $Z_i$ is unobserved, inference is based on the marginal observed density
\begin{align}
    \label{eq:observed_density_intro}
    p_{\theta,\beta}(y) = \int_{\mathcal{Z}} p_{\theta,\beta}(y\vert z)\,\mu(dz),
\end{align}
where $\mu$ is the law of $Z$. We denote by $P_{\theta,\beta}(Y)$ the induced observed law of $Y$.

Our goal is to estimate $\theta^\star$ from $Y_1,\ldots,Y_n$. In general, however, $\theta^\star$ is identifiable only up to a model-intrinsic equivalence symmetry, so the proper inferential target is the equivalence class $[\theta^\star]$. For example, in a $K$-component Gaussian mixture model, the distribution is invariant under permutation of the component labels, while in orbit-recovery models it is invariant under the corresponding group action \cite{bandeira2023estimation,fan2023likelihood}. Throughout, we therefore formulate estimation on the quotient space induced by this equivalence relation and measure error by the quotient distance $d_{\mathrm{eq}}(\theta,\theta')$, defined as the minimum distance between representatives of the two equivalence classes.

\paragraph{Weak-signal limit and null law.}
The weak-signal assumption is that, at zero signal strength, the conditional observation law loses all dependence on both the unknown parameter and the latent variable:
\begin{align}
    \label{eq:null_conditional_model}
    p_{\theta,0}(y\vert z)=p_0(y),
    \qquad
    \theta\in\Theta,\ z\in\mathcal{Z}.
\end{align}
Consequently, after marginalizing over $Z$, $p_{\theta,0}(y)=p_0(y)$ for every $\theta\in\Theta$. We denote by $P_0$ the null law of this common observed law. 
At $\beta=0$, recovery is impossible because the observed law is independent of $\theta$. For small positive $\beta$, the law $P_{\theta,\beta}$ may depend on $\theta$, but only through weak deviations from $P_0$. The statistical problem studied in this paper is to determine how many samples are required, as $\beta\to 0$, to recover the identifiable information about $\theta^\star$ from samples drawn from $P_{\theta^\star,\beta}$.

\paragraph{The likelihood-ratio expansion.}
The central object in our analysis is the likelihood ratio relative to the zero-signal law $P_0$. For the weak-signal models considered in this paper, the likelihood ratio admits an expansion of the form
\begin{align}
    L_{\theta, \beta}(Y) \triangleq \frac{dP_{\theta,\beta}}{dP_0}(Y) = 1+\beta\phi_1(Y;\theta)+\beta^2\phi_2(Y;\theta)+\cdots .    \label{eq:informal_likelihood_ratio_expansion}
\end{align}
Equivalently, when the map $\beta \mapsto L_{\theta,\beta}(Y)$ is real analytic in a neighborhood of $\beta=0$, it admits a Taylor expansion whose coefficients are given by
\begin{align}
    \phi_\ell(\cdot;\theta) = \frac{1}{\ell!} \left. \frac{\partial^\ell}{\partial\beta^\ell} L_{\theta, \beta}(\cdot) \right|_{\beta=0}.    \label{eqn:likelihood_moments_def}
\end{align}
We call these coefficients the \emph{likelihood moments}, and define the associated maps $\Phi_\ell:\Theta\to L^2(P_0)$,
\begin{align}
    \int \left| \phi_\ell(y;\theta) \right|^2 \; dP_0(y) < \infty, \qquad \Phi_\ell(\theta) \triangleq \phi_\ell(\cdot;\theta) \in L^2(P_0).    \label{eq:likelihood_moment_definition}
\end{align}
These coefficients describe the successive parameter-dependent deviations of the observed law from the null law.

\subsection{Main results}

We measure statistical difficulty through the number of samples required to recover the unknown parameter, up to observational equivalence, to a prescribed accuracy. Given $n$ i.i.d. observations $Y_1,\ldots,Y_n\sim P_{\theta,\beta}$, let $\widehat{\theta}=\widehat{\theta}(Y_1,\ldots,Y_n)\in\Theta$ denote an estimator of $\theta$. For an accuracy level $\varepsilon>0$ and confidence level $\delta\in(0,1)$, let
\begin{align}
    n^\star(\varepsilon,\delta;\beta) \triangleq \inf\left\{ n: \inf_{\widehat{\theta}} \sup_{\theta} \, P_{\theta,\beta}^{\otimes n} \left( d_{\mathrm{eq}}(\widehat{\theta},\theta)>\varepsilon \right) \leq\delta \right\},    \label{eq:intro_sample_complexity}
\end{align}
where the infimum is over all measurable estimators based on $Y_1,\ldots,Y_n$, and the supremum is taken over the parameter class under consideration. Thus, $n^\star(\varepsilon,\delta;\beta)$ is the minimax sample complexity for estimation to quotient error at most $\varepsilon$ with probability at least $1-\delta$.

We study two related weak-signal regimes. First, we fix $\varepsilon>0$ and $\delta\in(0,1)$ and ask how $n^\star(\varepsilon,\delta;\beta)$ grows as $\beta\to 0$. Second, we consider the finer joint regime in which both $\beta\to 0$ and $\varepsilon\to 0$, and ask how the difficulty of global identification interacts with the difficulty of increasingly accurate local estimation.

The likelihood-ratio expansion provides a natural ordering of the information about the parameter in the weak-signal limit. For $r\geq 1$, define the \emph{truncated likelihood-moment map}
\begin{align}
    \Phi_{\leq r}(\theta) \triangleq \bigl(\Phi_1(\theta),\ldots,\Phi_r(\theta)\bigr).
    \label{eq:truncated_likelihood_moment_map}
\end{align}
Two identifiability orders play distinct roles. The \emph{global identifiability order} $r_{\mathrm{glob}}$ is the smallest order for which $\Phi_{\leq r}$ identifies a generic parameter up to observational equivalence over the full parameter space $\Theta$. The \emph{local identifiability order} $r_{\mathrm{loc}}$ is the smallest order for which the differential of $\Phi_{\leq r}$ detects every identifiable infinitesimal perturbation around the ground-truth target. Thus, $r_{\mathrm{glob}}$ governs global localization, whereas $r_{\mathrm{loc}}$ governs fine resolution once the target has been localized. In the algebraic setting considered below, generically $r_{\mathrm{loc}}\leq r_{\mathrm{glob}}$, and the inequality can be strict.

\begin{thm}[Informal main theorem: fixed accuracy]
\label{thm:informal_main}
Fix $\varepsilon>0$ and $\delta\in(0,1)$, and suppose that $r_{\mathrm{glob}}<\infty$. Then:

\begin{enumerate}

    \item \emph{(Information-theoretic lower bound; see Theorem~\ref{thm:lecam_lower_bound}.)}
    If the parameter class contains two values separated by more than $2\varepsilon$ whose likelihood moments agree through order $r_{\mathrm{glob}}-1$, then estimation to accuracy $\varepsilon$ with confidence $1-\delta$ requires
    \begin{align}
        n^\star(\varepsilon,\delta;\beta) \gtrsim \beta^{-2r_{\mathrm{glob}}}\log(1/\delta),
    \end{align}
    up to constants depending on the model and $\varepsilon$.
    
    \item \emph{(Necessary condition for minimax sharpness; see
    Theorem~\ref{thm:cross_order_stability_necessity}.)}
    If the rate $\beta^{-2r_{\mathrm{glob}}}$ is attainable by any estimator, then there exists a constant $c_\varepsilon>0$, independent of $\beta$, such that, for all sufficiently small $\beta$,
    \begin{align}
        \inf_{\substack{\theta,\theta'\in\Theta:\\
        d_{\mathrm{eq}}(\theta,\theta')>2\varepsilon}} \left\|L_{\theta,\beta}-L_{\theta',\beta}\right\|_{L^2(P_0)} \geq c_\varepsilon\beta^{r_{\mathrm{glob}}}.
        \label{eqn:necessary_sharpness}
    \end{align}
    
    \item \emph{(Matching upper bounds under stable likelihood geometry.)}
    Conversely, if~\eqref{eqn:necessary_sharpness} holds, and under suitable regularity conditions, both maximum likelihood (Theorem~\ref{thm:mle_upper_bound}) and the method-of-moments based on geometric median-of-means estimation (Theorem~\ref{thm:mom_upper_bound}) attain
    \begin{align}
        n^\star(\varepsilon,\delta;\beta) \lesssim \beta^{-2r_{\mathrm{glob}}}\log(1/\delta).
    \end{align}
\end{enumerate}
\end{thm}

When the target accuracy $\varepsilon$ also tends to zero, the local order $r_{\mathrm{loc}}$ becomes relevant and can lead to a different statistical scale.

\begin{thm}[Informal main theorem: shrinking accuracy; see Theorems~\ref{thm:global_to_local_upper_bound} and~\ref{thm:global_to_local_lower_bound}]
\label{thm:informal_joint_main}
Assume that the likelihood-moment maps have the regular polynomial structure developed in Section~\ref{sec:global_to_local_algebraic}, so that order $r_{\mathrm{glob}}$ provides global separation while order $r_{\mathrm{loc}}$ provides stable local inversion along the identifiable directions. Then, as $\beta,\varepsilon\to 0$, 
\begin{align}
    n^\star(\varepsilon,\delta;\beta) \asymp \max\left\{\beta^{-2r_{\mathrm{glob}}}, \beta^{-2r_{\mathrm{loc}}}\varepsilon^{-2} \right\} \log(1/\delta),    \label{eq:intro_global_local_sample_complexity}
\end{align}
up to model-dependent constants.
\end{thm}

The two terms in~\eqref{eq:intro_global_local_sample_complexity} correspond to different statistical tasks. The first is the cost of \emph{global localization}: identifying the correct equivalence class up to a fixed neighborhood, which requires order $\beta^{-2r_{\mathrm{glob}}}$ samples. The second is the cost of \emph{local refinement}: nearby parameters at quotient distance of order $\varepsilon$ are separated in the first locally informative likelihood moments by order $\varepsilon$, giving likelihood separation of order $\beta^{r_{\mathrm{loc}}}\varepsilon$ and hence the scale $\beta^{-2r_{\mathrm{loc}}}\varepsilon^{-2}$.

The two scales in~\eqref{eq:intro_global_local_sample_complexity} coincide when $\varepsilon\asymp\beta^{\,r_{\mathrm{glob}}-r_{\mathrm{loc}}}$. Above this scale, the dominant difficulty is global localization; below it, the dominant difficulty is fine local resolution. In particular, when $r_{\mathrm{loc}}=r_{\mathrm{glob}}=r$, the shrinking-accuracy law reduces to
\begin{align}
    n^\star(\varepsilon,\delta;\beta) \asymp \beta^{-2r}\varepsilon^{-2}\log(1/\delta).
\end{align}

\subsection{Motivating examples}
\label{subsec:algebraic_linear_perturbation_intro}

\paragraph{Existing theory: Gaussian latent-variable models.}
Much of the existing weak-signal theory has been developed for Gaussian latent-variable models of the form~\cite{balanov2026generalized}
\begin{align}
    Y=\beta A(Z)\theta^\star+\sigma\xi,
    \qquad
    \xi\sim\mathcal{N}(0,I),
    \label{eq:gaussian_latent_variable_example_intro}
\end{align}
which provide a canonical class of examples. Conditional on $Z=z$, the observation is Gaussian with mean $\beta A(z)\theta^\star$ and covariance $\sigma^2 I$, while at $\beta=0$ the model reduces to the parameter-independent null law $\mathcal{N}(0,\sigma^2 I)$. A key feature of this setting is that the weak-signal likelihood expansion separates into known observation-side coordinates, represented by Hermite polynomials associated with the Gaussian null law, and signal-side coefficients determined by the corresponding moments of the latent signal $A(Z)\theta$~\cite{balanov2026generalized,bandeira2023estimation}. This separation makes the order-by-order statistical information in the observed distribution directly comparable to the order at which the signal moments identify the unknown parameter. The random operator $A(Z)$ may represent, for example, a component selector, shift, group action, projection, mask, or latent structural state. This framework includes Gaussian mixtures, multi-reference alignment, orbit recovery, heterogeneous and projected latent-variable models, and random masking or subsampling models~\cite{hsu2013learning,anandkumar2014tensor,chaganty2013spectral,sedghi2016provable,bendory2017bispectrum,abbe2018multireference,abbe2018estimation,bandeira2023estimation,katsevich2023likelihood,bendory2022super}.

\paragraph{Algebraic weak-signal perturbation models.}
A broader class with the same basic structure is obtained from a known parametric observation family $\{P_m:m\in\mathcal{M}\}$, where $m$ is the parameter indexing the observation distribution. Conditional on $Z=z$, we assume that this parameter follows the weak-signal perturbation
\begin{align}
    m_{\theta,\beta}(z) =  m_0+\beta s(z;\theta),    \label{eq:linear_perturbation_intro}
\end{align}
where $m_0$ is a parameter-independent null value and $s(z;\theta)$ is the latent signal perturbation. Accordingly, $Y\vert Z=z\sim P_{m_{\theta,\beta}(z)}$. Thus, the dependence of the conditional observation law on $(\theta,\beta,z)$ is mediated entirely through the observation-family parameter $m_{\theta,\beta}(z)$. The perturbation in~\eqref{eq:linear_perturbation_intro} is linear in the signal-strength parameter $\beta$. When $\theta\mapsto s(z;\theta)$ is polynomial, we call the resulting model an \emph{algebraic weak-signal perturbation model}; the important special case $s(z;\theta)=A(z)\theta$ is also linear in $\theta$.

We develop this class formally in Section~\ref{subsec:linear_weak_signal_perturbations}.  In particular, two broad classes of observation models fit naturally into this framework, which we describe next.

\paragraph{Additive-noise models.}
A particularly important specialization of the preceding framework is obtained when the observation family is generated by additive noise. In this case,
\begin{align}
    Y=\beta s(Z;\theta^\star)+\xi,
    \label{eq:additive_noise_example_intro}
\end{align}
where $\xi$ has a known distribution and is independent of $Z$. Equivalently, conditional on $Z=z$, the observation law is a translation of the noise distribution by the weak signal $\beta s(z;\theta^\star)$. The Gaussian model~\eqref{eq:gaussian_latent_variable_example_intro} is the canonical example. More generally, for a broad class of additive-noise distributions, suitable Appell polynomials play the role of the Hermite polynomials and provide observation-side coordinates that separate the likelihood expansion from the latent signal moments~\cite{ta2015probabilistic}. We study this construction in Section~\ref{subsec:additive_noise_channels}.

\paragraph{Exponential-family models.}
The algebraic weak-signal perturbation framework also includes exponential-family observation models, in which the parameter of the conditional observation law is perturbed around a fixed null value according to~\eqref{eq:linear_perturbation_intro}. A canonical example is provided by count-valued observations,
\begin{align}
    Y\vert Z=z \sim \operatorname{Poisson}\!\left( \lambda_0+\beta s(z;\theta^\star) \right),
    \qquad
    \lambda_0>0, \label{eq:poisson_latent_variable_example_intro}
\end{align}
where $\lambda_0$ is a known background intensity. At $\beta=0$, the model reduces to the parameter-independent null law $\operatorname{Poisson}(\lambda_0)$. Such models arise naturally in Poisson inverse problems and photon-limited imaging, where the observed counts constitute weak perturbations of a background intensity~\cite{antoniadis2006poisson,nowak2000statistical,timmermann2002multiscale}, and they also include finite Poisson mixtures and related latent-variable models for count data~\cite{karlis2007finite}.
More generally, the intensity may be vector-valued, with $s(Z;\theta^\star)$ describing, for example, a latent shift, a component-dependent intensity profile, a projection, or another structured perturbation of the background intensity.
Section~\ref{subsec:mean_scaling_exponential_family} develops the corresponding construction for exponential-family models.

\subsection{Related work}
\label{subsec:related_work}

This work connects two lines of research: moment methods for latent-variable models and likelihood-based analyses of weak-signal asymptotics. Moment and tensor methods are classical tools for finite mixtures, hidden-variable models, hidden Markov models, and mixtures of regressions~\cite{lindsay1989moment,hsu2013learning,anandkumar2014tensor,chaganty2013spectral,sedghi2016provable}. In much of this literature, moments are introduced as model-specific statistics that yield identifiability results and computationally tractable algorithms. In our setting, the relevant moments instead arise canonically from the likelihood-ratio expansion around the common zero-signal law: its Taylor coefficients define the likelihood moments that organize the weak-signal statistical information. In structured models, these likelihood moments reduce to familiar null-adapted features, such as Hermite moments in the Gaussian case~\cite{balanov2026generalized}.

Orbit-recovery problems, including those motivated by structural biology, provide an important example of this connection. At low signal-to-noise ratio (SNR), individual latent transformations are typically unrecoverable, while the underlying signal may still be identified from aggregate invariant information. In Gaussian orbit-recovery models, the relevant likelihood moments correspond to familiar invariant moments of the group action, including the power spectrum and bispectrum, and more generally to invariant polynomials~\cite{bendory2017bispectrum,perry2019sample,bandeira2023estimation}. This viewpoint separates two aspects of the problem: estimating the likelihood moments from noisy observations and determining whether those moments identify the underlying parameter. Our framework formulates this separation directly through the likelihood expansion and extends it beyond Gaussian observation models.

A parallel line of work studies the likelihood geometry of Gaussian latent-variable models in the low-SNR regime. These analyses show that, for Gaussian mixtures and orbit-recovery models, the likelihood and Fisher information are organized by the lowest moment order at which identifiability occurs~\cite{katsevich2023likelihood,fan2023likelihood,fan2024maximum,balanov2026generalized}. The present paper builds on this viewpoint by developing a general weak-signal likelihood framework, distinguishing identifiability of the likelihood moments from finite-$\beta$ likelihood separation, and separating the orders governing global localization and fine local resolution. 

\subsection{Notation}
Throughout, $\xrightarrow[]{\mathcal{D}}$, $\xrightarrow[]{\mathbb{P}}$, $\xrightarrow[]{\mathrm{a.s.}}$, and $\xrightarrow[]{L^p}$ denote convergence in distribution, in probability, almost surely, and in $L^p$, respectively.
We write $\|\cdot\|$ for the Euclidean norm on vectors and $\|\cdot\|_F$ for the Frobenius norm on matrices. 

We denote by $P_{\theta,\beta}$  the observed law of $Y$. At $\beta=0$, all parameters induce the same null law $P_0$. We write $L^2(P_0)$ for the Hilbert space of square-integrable functions of the observed variable $Y$ under the null law, equipped with the norm
\begin{align}
    \|f\|_{L^2(P_0)}^2 = \mathbb{E}_{Y\sim P_0}\left[f(Y)^2\right].
\end{align}

Recall the definition of the likelihood moments $\Phi_\ell(\theta)$ from~\eqref{eq:likelihood_moment_definition}. Throughout, equality of likelihood moments is understood componentwise in $L^2(P_0)$.
Thus, $\Phi_{\le k}(\theta)=\Phi_{\le k}(\theta')$ means that  $\left\|\Phi_\ell(\theta)-\Phi_\ell(\theta')\right\|_{L^2(P_0)}=0$ for $\ell=1,\ldots,k.$

For $v\in\mathbb{R}^d$, its order-$k$ tensor power $v^{\otimes k}\in(\mathbb{R}^d)^{\otimes k}$ has entries $(v^{\otimes k})_{i_1,\ldots,i_k} = v_{i_1}\cdots v_{i_k}$. For a random vector $Y$ with finite moments of order $k$, its order-$k$ tensor moment is $\mathbb{E}\!\left[Y^{\otimes k}\right]$, whose $(i_1,\ldots,i_k)$ entry is $\mathbb{E}[Y_{i_1}\cdots Y_{i_k}]$. For $n\geq 1$, we denote by $P^{\otimes n}$ the law of $n$ i.i.d.\ samples from $P$. 

For nonnegative sequences ${a_n}$ and ${b_n}$, we write $a_n=O(b_n)$ if there exists a constant $C<\infty$ and $n_0$ such that $a_n\le C b_n$ for all $n\ge n_0$, and $a_n=o(b_n)$ if $a_n/b_n\to 0$. We write $a_n=\Theta(b_n)$ if both $a_n=O(b_n)$ and $b_n=O(a_n)$ hold.  We also use the notation $a_n \gtrsim b_n$ to denote $b_n = O(a_n)$. Finally, we write $a_n \asymp b_n$ as a synonym for $a_n = \Theta(b_n)$.

\section{Information-theoretic characterization of weak-signal sample complexity}
\label{sec:general_framework}

\subsection{Observational equivalence and quotient geometry}
\label{subsec:general_equivalence_classes}

Throughout, we assume that the parameter space is a compact set $\Theta \subseteq \mathbb{R}^m$. To quantify estimation error and define the sample complexity, we equip $\Theta$ with a metric $d_\Theta$ that is compatible with its topology. The most natural choice is the metric induced by the  Euclidean norm, $d_\Theta(\theta_0,\theta_1) = \| \theta_0-\theta_1 \|_2$. However, other choices may be appropriate. For example, if $\Theta$ is a smooth manifold, one may use the geodesic distance induced by a Riemannian metric on $\Theta$. Throughout, we assume that $d_\Theta$ is locally equivalent to the metric induced by the Euclidean norm near the ground-truth parameter $\theta^\star$, i.e., there exists a neighborhood $U_\star \subseteq \Theta$ of $\theta^\star$ and constants $0<c_0\le C_0<\infty$ such that
\begin{align}
    c_0 \| \theta_0-\theta_1 \|_2 \le d_\Theta(\theta_0,\theta_1) \le C_0 \| \theta_0-\theta_1 \|_2,
\end{align}
for every $\theta_0,\theta_1 \in U_\star$.
Under this assumption, the specific choice of metric may affect the constants in the resulting bounds, but not the weak-signal sample-complexity exponent.

\paragraph{Observational equivalence.}
Since the parameter is identifiable only up to intrinsic symmetries of the observed distribution, we define equivalence directly at the level of the observed law~\cite{balanov2026generalized}.
This formulation covers the usual  group-orbit ambiguities that arise in latent-variable models; representative examples are summarized in Table~\ref{tab:linear_mean_scaling_models}.

\begin{definition}[Observational equivalence and equivalence class]
\label{def:general_observational_equivalence}
For $\theta,\theta' \in \Theta$, we say that $\theta$ and $\theta'$ are \emph{observationally equivalent}, and write $\theta \sim \theta'$, if $P_{\theta,\beta} = P_{\theta',\beta}$ for all sufficiently small $\beta > 0$. The \emph{equivalence class} of $\theta$ is $[\theta] \triangleq \left\{ \theta' \in \Theta : \theta' \sim \theta \right\}$.
The collection of all equivalence classes defines the \emph{quotient space} $\Theta/{\sim} \triangleq \left\{ [\theta] : \theta \in \Theta \right\}$. The identifiable target is the equivalence class $[\theta^\star]$.
\end{definition}
For the exposition in the main text, and to cover the principal examples in a concrete way, we specialize to the case where the observational equivalence relation is induced by a compact Lie group action. More general observational equivalence relations, equipped with an appropriate smooth quotient structure, are discussed in Appendix~\ref{sec:smooth_quotient_models}.

\begin{assum}[Group-induced equivalence]
\label{ass:group_induced_equivalence}
There exists a compact Lie group $G$ acting  on $\Theta$ such that
\begin{align}
    \theta \sim \theta' \quad\Longleftrightarrow\quad \theta' = g \cdot \theta
    \quad \text{for some } g \in G. \label{eq:group_induced_equivalence}
\end{align}
In addition, the metric $d_\Theta$ is invariant under this action, namely, for all $\theta_0,\theta_1 \in \Theta$ and $g \in G$, $d_\Theta(g\cdot\theta_0,g\cdot\theta_1) = d_\Theta(\theta_0,\theta_1)$.
\end{assum}

Assumption~\ref{ass:group_induced_equivalence} holds in the main examples considered in this paper; see Appendix~\ref{subsec:linear_mean_scaling_latent_variable_models}. In a $K$-component mixture model, $G$ is the permutation group acting on the component labels. In multi-reference alignment, $G$ is the cyclic-shift group. In orbit-recovery models, $G$ is the corresponding transformation group acting on the signal. The assumption also covers reflections, rotations, and other compact group-orbit models. 

\paragraph{Quotient loss.}
Since the identifiable object is an equivalence class rather than a particular representative, the distance metric should be defined on the quotient space $\Theta/{\sim}$. Under Assumption~\ref{ass:group_induced_equivalence}, we define the quotient loss by 
\begin{align} 
    d_{\mathrm{eq}}(\theta_0,\theta_1) = d_{\Theta/{\sim}}\bigl([\theta_0],[\theta_1]\bigr) = \min_{g\in G} d_\Theta(\theta_0,g\cdot\theta_1), \label{eq:orbit_distance_special_case} 
\end{align}
where the minimum is attained because $G$ is compact.
Since $\Theta$ is also compact, the orbit space $\Theta/G$ is a compact metric space, 
and the canonical projection $\pi:\Theta\to\Theta/G$ is continuous.

\subsection{Likelihood-ratio expansion}
\label{subsec:likelihood_expansion}

Recall that $P_{\theta,\beta}$ denotes the observed law of $Y$, obtained after marginalizing over the latent variable $Z$, and let $P_0$ denote the parameter-independent null law at $\beta=0$.
We write
\begin{align}
    L_{\theta,\beta}(Y) \triangleq \frac{dP_{\theta,\beta}}{dP_0}(Y) = \frac{p_{\theta,\beta}(Y)}{p_0(Y)}, 
\end{align}
for the observed likelihood ratio relative to the null law. We impose the following regularity conditions on its expansion around $\beta=0$.

\begin{assum}[Regular observed likelihood-ratio expansion]
\label{ass:regular_observed_likelihood_ratio_expansion}
For every finite $p\geq 1$, the map $\beta \mapsto L_{\theta,\beta}$ is real analytic as an $L^p(P_0)$-valued map in a neighborhood of $\beta=0$, uniformly over $\theta\in\Theta$.
\end{assum}

Assumption~\ref{ass:regular_observed_likelihood_ratio_expansion} implies that for every fixed $s\ge 1$, the likelihood ratio admits the Taylor expansion
\begin{align}
    L_{\theta,\beta}(Y) = 1+\sum_{\ell=1}^{s}\beta^\ell\phi_\ell(Y;\theta) + R_{s+1}(Y;\theta,\beta), \label{eq:observed_likelihood_ratio_expansion}
\end{align}
where $\phi_\ell(\cdot;\theta)\in L^p(P_0)$ are the likelihood moments as defined in~\eqref{eqn:likelihood_moments_def} and the remainder satisfies $\sup_{\theta\in\Theta} \left\|R_{s+1}(\cdot;\theta,\beta)\right\|_{L^p(P_0)} = O(\beta^{s+1})$, as $\beta\to 0$.  Since $L_{\theta,\beta}=dP_{\theta,\beta}/dP_0$, we have $\mathbb{E}_{Y \sim P_0} [L_{\theta,\beta}(Y)]=1$ for every $\theta$ and $\beta$. Taking $P_0$-expectations in the Taylor expansion therefore gives
\begin{align}
    \mathbb{E}_{Y\sim P_0}\left[\phi_\ell(Y;\theta)\right]=0,
    \qquad \ell\ge 1. \label{eq:observed_likelihood_ratio_centering}
\end{align}

Throughout the remainder of the paper, unless stated otherwise, we assume that Assumptions~\ref{ass:group_induced_equivalence} and~\ref{ass:regular_observed_likelihood_ratio_expansion} hold. We therefore omit them from subsequent theorem and proposition statements, and state only the additional assumptions required for each result.

\subsection{Statistical distinguishability}

We next relate the likelihood-moment expansion to the statistical distinguishability of two parameter values. For $\theta,\theta'\in\Theta$, define their \emph{first mismatching 
likelihood-moment order} by
\begin{align}
    r(\theta,\theta') \triangleq \min\left\{ \ell\geq 1: \|\Phi_\ell(\theta)-\Phi_\ell(\theta')\|_{L^2(P_0)}>0 \right\},   \label{eq:first_mismatching_likelihood_moment_order}
\end{align}
with the convention $r(\theta,\theta')=\infty$ if all likelihood moments agree.

If $r(\theta,\theta')=r<\infty$, the likelihood-ratio expansion implies that the first nonzero difference between the two observed laws, relative to the common null law, appears at amplitude $\beta^r$. To translate this $P_0$-based expansion into divergences between $P_{\theta,\beta}$ and $P_{\theta',\beta}$, we also need mild control of the reciprocal likelihood ratio, since quantities such as the $\chi^2$ divergence involve division by one of the two densities.

\begin{assum}[Reciprocal likelihood regularity]
\label{ass:reciprocal_likelihood_regularity}
There exists a fixed $p_0>2$ such that, for all sufficiently small $\beta$,
\begin{align}
    \sup_{\theta\in\Theta} \norm{ L_{\theta,\beta}^{-1}-1 }_{L^{p_0}(P_0)} = O(\beta). \label{eq:observed_likelihood_ratio_reciprocal_assumption}
\end{align}
\end{assum}

Assumption~\ref{ass:reciprocal_likelihood_regularity} ensures that passing from likelihood ratios relative to the null law $P_0$ to likelihood ratios between two nearby weak-signal laws does not change their leading order in $\beta$. It is used only in the information-theoretic lower-bound analysis. The condition holds in the weak-signal models of primary interest here; see, for example, the exponential-family models in Section~\ref{subsec:mean_scaling_exponential_family}. It mainly excludes irregular situations in which an arbitrarily small perturbation can drive the likelihood ratio close to zero on rare events. The following proposition makes the resulting divergence scale precise, where $\chi^2(P,Q)$ denotes the chi-square divergence between probability distributions $P$ and $Q$, and $D_{\mathrm{KL}}(P \| Q)$ denotes the Kullback-Leibler divergence between $P$ and $Q$. Its proof is given in Appendix~\ref{sec:proof_prop_divergence_first_likelihood_moment}.

\begin{proposition}[Weak-signal divergence from the first mismatching likelihood moment]
\label{prop:divergence_first_likelihood_moment}
Suppose Assumption~\ref{ass:reciprocal_likelihood_regularity} holds. Let $\theta_0,\theta_1\in\Theta$, and suppose that
$r(\theta_0,\theta_1)=r<\infty$.
Then, as $\beta\to 0$,
\begin{align}
    \chi^2(P_{\theta_0,\beta},P_{\theta_1,\beta}) &= \beta^{2r}\|\Phi_r(\theta_0)-\Phi_r(\theta_1)\|_{L^2(P_0)}^2 + O(\beta^{2r+1}),    \label{eq:chi_square_expansion_general}
\end{align}
and
\begin{align}
    D_{\mathrm{KL}}(P_{\theta_0,\beta}\|P_{\theta_1,\beta})
    &= \frac{\beta^{2r}}{2} \|\Phi_r(\theta_0)-\Phi_r(\theta_1)\|_{L^2(P_0)}^2 + O(\beta^{2r+1}).
    \label{eq:kl_expansion_general}
\end{align}
\end{proposition}

Thus, if two parameters agree in all likelihood moments below order $r$, their first statistical separation occurs at amplitude $\beta^r$, and the corresponding $\chi^2$ and KL divergences are of order $\beta^{2r}$. This pairwise divergence scale is the basis of the sample-complexity lower bound developed next.

\subsection{Global identifiability order}
\label{sec:global_informative_order}

We next formalize the likelihood-moment order required for global identification. This is a population identifiability notion: it asks at which likelihood-moment order the truncated map determines the parameter up to observational equivalence. In many latent variable models, such identifiability holds generically rather than uniformly over the entire parameter space. We say that a property holds for a \emph{generic signal} in $\Theta$ if it holds outside an exceptional set $\Theta_{\mathrm{exc}}\subseteq\Theta$ of measure zero.

\begin{definition}[Generic global identifiability likelihood-moment order]
\label{def:generic_rglob_likelihood_main}
We say that the model is \emph{globally likelihood-moment identifiable of order $k$ for a generic signal} if there exists an exceptional set $\Theta_{\mathrm{exc}}\subseteq\Theta$ of measure zero such that, for every $\theta^\star\in\Theta\setminus\Theta_{\mathrm{exc}}$,
\begin{align}
    \label{eq:generic_global_ident_likelihood_main}
    \Phi_{\le k}(\theta)=\Phi_{\le k}(\theta^\star) \quad\Longrightarrow\quad \theta\sim\theta^\star,
    \qquad    \forall\,\theta\in\Theta .
\end{align}
The corresponding generic global identifiability likelihood-moment order is
\begin{align}
    \label{eq:generic_rglob_likelihood}
    r_{\mathrm{glob}} \triangleq \min\left\{ k\ge 1: \eqref{eq:generic_global_ident_likelihood_main} \text{ holds for a generic signal} \right\},
\end{align}
with the convention $r_{\mathrm{glob}}=\infty$ if no finite order has this property.
\end{definition}


When $r_{\mathrm{glob}}<\infty$, we define the corresponding generic parameter set
\begin{align}
    \Theta_{\mathrm{gen}} \triangleq \left\{ \theta^\star\in\Theta: \Phi_{\leq r_{\mathrm{glob}}}(\theta) = \Phi_{\leq r_{\mathrm{glob}}}(\theta^\star) \Longrightarrow \theta\sim\theta^\star \ \text{for every }\theta\in\Theta \right\}.
    \label{eq:generic_parameter_set}
\end{align}
By definition, $\Theta\setminus\Theta_{\mathrm{gen}}$ has measure zero. Since the constants appearing in uniform recovery guarantees may deteriorate as the target approaches this exceptional set, we formulate the minimax sample complexity results on arbitrary compact subsets $K\subseteq\Theta_{\mathrm{gen}}$.

\subsection{Information-theoretic sample-complexity lower bound}
\label{subsec:minimax_general}

First, we study the weak-signal limit $\beta\to 0$ with the accuracy level $\varepsilon>0$ and confidence level $\delta\in(0,1)$ held fixed. Thus, our principal interest is the dependence of the sample complexity on the signal strength $\beta$, while constants may depend on the prescribed accuracy and confidence levels.

\begin{definition}[Minimax sample complexity]
\label{def:sample_complexity_minimax}
Fix a compact set $K\subseteq\Theta_{\mathrm{gen}}$, an accuracy level $\varepsilon>0$, and a confidence level $1-\delta$, with $\delta\in(0,1)$.
For $n$ i.i.d.\ observations drawn from $P_{\theta,\beta}$, define the minimax error probability over $K$ by
\begin{align}
    \mathfrak{R}_{n}(\varepsilon;\beta) \triangleq \inf_{\widehat\theta:\mathcal{Y}^n\to\Theta} \sup_{\theta\in K}   \, P_{\theta,\beta}^{\otimes n} \left( d_{\mathrm{eq}} \bigl( \widehat\theta(Y_1,\ldots,Y_n),\theta \bigr) > \varepsilon \right),
    \label{eq:minimax_error_prob}
\end{align}
where $d_{\mathrm{eq}}$ is the quotient loss induced by observational
equivalence, as defined in~\eqref{eq:orbit_distance_special_case}.
The corresponding information-theoretic sample complexity is
\begin{align}
    n^\star(\varepsilon,\delta;\beta) \triangleq \inf \left\{ n\in\mathbb N: \mathfrak{R}_{n}(\varepsilon;\beta)\leq\delta \right\}.
    \label{eq:sample_complexity_def}
\end{align}
\end{definition}

The next result, proved in Appendix~\ref{app:proof_lecam_lower_bound}, gives the basic information-theoretic lower bound. The argument is a two-point testing reduction. If two parameter values in $K$ are separated by more than the target accuracy and agree in all likelihood moments below order $r_{\mathrm{glob}}$, then their divergence is of order $\beta^{2r_{\mathrm{glob}}}$. Consequently, distinguishing their product laws with high probability requires on the order of $\beta^{-2r_{\mathrm{glob}}}$ observations. 

\begin{thm}[Minimax lower bound from a hard pair]
\label{thm:lecam_lower_bound}
Suppose
Assumption~\ref{ass:reciprocal_likelihood_regularity} holds. Fix a compact set $K\subseteq\Theta_{\mathrm{gen}}$, $\varepsilon>0$, and $\delta\in(0,1/4)$. Assume that there exist $\theta_0,\theta_1\in K$ such that $d_{\mathrm{eq}}(\theta_0,\theta_1) > 2\varepsilon$, and $r(\theta_0,\theta_1)=r=r_{\mathrm{glob}}<\infty$, where $r(\theta_0,\theta_1)$ is the first mismatching likelihood-moment order defined in \eqref{eq:first_mismatching_likelihood_moment_order}. Then, there exists a constant $c_{\theta_0,\theta_1}>0$, independent of $\beta$ and $\delta$, such that for all sufficiently small $\beta$,
\begin{align}
    n^\star(\varepsilon,\delta;\beta) \geq c_{\theta_0,\theta_1} \beta^{-2r} \log\left(1/\delta\right). \label{eq:minimax_hard_pair_lower_bound}
\end{align}
\end{thm}

\subsection{Conditions for minimax sharpness}
\label{sec:likelihood_moment_achievability}

The preceding lower bound identifies $\beta^{-2r_{\mathrm{glob}}}$ as the weak-signal sample-complexity scale suggested by likelihood-moment identifiability. However, identifiability of the individual likelihood coefficients does not by itself guarantee that this rate is attainable uniformly over the parameter space. At finite $\beta$, differences arising at distinct likelihood orders are recombined in the observed likelihood and may cancel. We next show that, for any fixed accuracy, ruling out excessive cancellation between parameter values that must be distinguished is an information-theoretic requirement for attaining the predicted minimax rate. The proof is provided in Appendix~\ref{subsec:proof_thm_cross_order_stability_necessity}.

\begin{thm}[Necessary likelihood separation for minimax sharpness]
\label{thm:cross_order_stability_necessity}
Let $r=r_{\mathrm{glob}}<\infty$. Fix a compact set $K\subseteq\Theta_{\mathrm{gen}}$, an accuracy level $\varepsilon>0$, and $0<\delta_0<1/2$. Suppose that there exists $C_{\varepsilon}<\infty$ such that, for all sufficiently small $\beta$,
\begin{align}
    n^\star(\varepsilon,\delta_0;\beta) \leq C_{\varepsilon}\beta^{-2r}. \label{eq:minimax_nominal_upper_bound}
\end{align}
Then there exist constants
$c_{\varepsilon}>0$ and $\beta_0>0$ such that, for every
$0<\beta\leq\beta_0$,
\begin{align}
    \inf_{\substack{\theta,\theta'\in K:\\ d_{\mathrm{eq}}(\theta,\theta')>2\varepsilon}} \left\| L_{\theta,\beta}-L_{\theta',\beta} \right\|_{L^2(P_0)} \geq c_{\varepsilon}\beta^r. \label{eq:necessary_likelihood_separation}
\end{align}
\end{thm}

Thus, if the rate $\beta^{-2r_{\mathrm{glob}}}$ is attainable uniformly over a compact parameter class $K$, then every pair of parameters in $K$ that must be distinguished at accuracy $\varepsilon$ must induce likelihood ratios separated at order $\beta^{r_{\mathrm{glob}}}$. The cancellations that can violate this condition require a special alignment between different orders of the expansion and are typically nongeneric in algebraically structured models. They are nevertheless important conceptually, because they show that coefficient-wise identifiability alone does not determine the statistical order. The following simple example illustrates such cancellation.

\begin{example}[Why $r_{\mathrm{glob}}$ alone is insufficient]
\label{ex:rglob_insufficient}
Let $P_0$ be the uniform distribution on $\{-1,1\}^2$, write $Y=(Y_1,Y_2)$, and let $\theta\in[-1,1]$. For sufficiently small $\beta$, consider the likelihood ratio
\begin{align}
    L_{\theta,\beta}(Y) = 1+\left(\beta\theta^2+\beta^2\theta\right)Y_1 + \beta^3\theta Y_2.    \label{eq:cancellation_counterexample_likelihood}
\end{align}
The first two likelihood coefficients are $\Phi_1(\theta)=\theta^2Y_1$ and $\Phi_2(\theta)=\theta Y_1$.
The first coefficient does not identify $\theta$, since $\Phi_1(\theta)=\Phi_1(-\theta)$, whereas the first two coefficients together do. Hence, $r_{\mathrm{glob}}=2$.

Fix any $\theta_0\in[-1,1)\setminus \{0\}$ and, for sufficiently small $\beta>0$, define $\theta_\beta=-\theta_0-\beta$. Then $\theta_\beta\in[-1,1]$ for all sufficiently small $\beta$, and $|\theta_\beta-\theta_0| = |2\theta_0+\beta| \to 2|\theta_0|>0$.
Thus, the pair remains separated by a fixed positive distance as $\beta\to 0$. On the other hand,
\begin{align}
    \nonumber L_{\theta_\beta,\beta}-L_{\theta_0,\beta}
    &=
    \left[ \beta(\theta_\beta^2-\theta_0^2) +\beta^2(\theta_\beta-\theta_0) \right]Y_1 + \beta^3(\theta_\beta-\theta_0)Y_2
    \\ \nonumber 
    &=
    \beta(\theta_\beta-\theta_0) \bigl(\theta_\beta+\theta_0+\beta\bigr)Y_1 +\beta^3(\theta_\beta-\theta_0)Y_2
    \\
    &= -\beta^3(2\theta_0+\beta)Y_2.
\end{align}
Hence, the order-one and order-two contributions cancel exactly, and
\begin{align}
    \norm{L_{\theta_\beta,\beta}-L_{\theta_0,\beta}}_{L^2(P_0)} = \beta^3|2\theta_0+\beta| \asymp \beta^3.
\end{align}
The corresponding divergence is therefore of order $\beta^6$, so distinguishing this fixed-accuracy pair requires order $\beta^{-6}$ samples, rather than the order $\beta^{-4}$ suggested by $r_{\mathrm{glob}}=2$. The failure is not confined to an exceptional set of signals: for every fixed $\theta_0\in[-1,1)\setminus{0}$, one can construct a $\beta$-dependent, fixed-distance alternative whose lower-order likelihood contributions cancel.
\end{example}

\section{Rate-optimal recovery via maximum likelihood and method-of-moments}
\label{sec:rate-optimal-recovery}

The preceding section showed that attaining the weak-signal scale $\beta^{-2r_{\mathrm{glob}}}$ requires the finite-$\beta$ likelihoods of well-separated parameters to remain separated at order $\beta^{r_{\mathrm{glob}}}$. We now give verifiable conditions ensuring this separation and establish matching upper bounds for maximum likelihood and the method-of-moments.

\subsection{Finite-dimensional likelihood geometry}
\label{subsec:finite_dimensional_likelihood_geometry}

We first impose an additional condition that each likelihood-moment order lies in a fixed finite-dimensional subspace of $L^2(P_0)$. This structure is standard in many of the applications considered in the literature and yields a decomposition into parameter-dependent coefficients and fixed observation-side likelihood coordinates. Although finite-dimensionality is not intrinsically required for the maximum likelihood estimator upper bound, it provides a convenient common framework for the present analysis and is essential for the method-of-moments estimator developed below. A more general maximum likelihood formulation, which does not require finite-dimensional likelihood moments, is given in Appendix~\ref{app:general_mle}.

\begin{assum}[Finite-dimensional likelihood geometry]
\label{ass:finite_dimensional_likelihood_geometry}
For every $\ell\geq 1$, the linear subspace $\mathcal{H}_\ell \triangleq \operatorname{span} \left\{ \Phi_\ell(\theta): \theta\in\Theta \right\} \subset L^2(P_0)$ is finite dimensional, and $\theta\mapsto\Phi_\ell(\theta)$ is continuous as a map from $\Theta$ into $L^2(P_0)$.
\end{assum}

Under Assumption~\ref{ass:finite_dimensional_likelihood_geometry}, let $D_\ell \triangleq \dim(\mathcal{H}_\ell) <\infty$, and choose a basis $q_{\ell,1},\ldots,q_{\ell,D_\ell}$ of $\mathcal{H}_\ell$. Then, for every $\theta\in\Theta$, there exists a unique coordinate vector $\alpha_\ell(\theta) \triangleq \bigl(\alpha_{\ell,1}(\theta), \ldots,  \alpha_{\ell,D_\ell}(\theta) \bigr) \in\mathbb{R}^{D_\ell}$ such that
\begin{align}
    \phi_\ell(y;\theta) = \sum_{j=1}^{D_\ell} \alpha_{\ell,j}(\theta)q_{\ell,j}(y).    \label{eqn:finite_dimentional_repr}
\end{align}
We define the corresponding  likelihood-coordinate vector $q_\ell(y) \triangleq \bigl( q_{\ell,1}(y), \ldots, q_{\ell,D_\ell}(y) \bigr) \in\mathbb{R}^{D_\ell}$. Then~\eqref{eqn:finite_dimentional_repr} becomes $\phi_\ell(y;\theta) = \left\langle \alpha_\ell(\theta), q_\ell(y) \right\rangle$, and the likelihood-ratio expansion can be written as
\begin{align}
    L_{\theta,\beta}(y) = 1+ \sum_{\ell\geq 1} \beta^\ell \left\langle \alpha_\ell(\theta), q_\ell(y) \right\rangle. \label{eqn:finite_dimensional_likelihood_ratio_expansion}
\end{align}

This representation separates two roles that will be important below.
The coordinate vector $\alpha_\ell(\theta)$ carries the dependence on the unknown parameter, whereas the basis functions $q_\ell$ describe the observation-side geometry of the order-$\ell$ likelihood coefficient.
The choice of basis for $\mathcal{H}_\ell$ is not intrinsic, but once a basis is fixed the coordinate vector $\alpha_\ell(\theta)$ is uniquely determined.

\begin{lem}[Coordinate faithfulness]
\label{lem:coordinate_faithfulness}
Suppose
Assumption~\ref{ass:finite_dimensional_likelihood_geometry} holds.
For every fixed $k\geq 1$, write $\alpha_{\leq k}(\theta) \triangleq \bigl(\alpha_1(\theta),\ldots,\alpha_k(\theta) \bigr) \in\mathbb{R}^{D_1+\cdots+D_k}$. Then, for every $\theta,\theta'\in\Theta$,
\begin{align}
    \Phi_{\leq k}(\theta)=\Phi_{\leq k}(\theta')
    \quad\Longleftrightarrow\quad
    \alpha_{\leq k}(\theta)=\alpha_{\leq k}(\theta').
    \label{eq:coordinate_faithfulness}
\end{align}
\end{lem}

Thus, under
Assumption~\ref{ass:finite_dimensional_likelihood_geometry}, identifiability through the likelihood moments $\Phi_1,\ldots,\Phi_k$ is equivalent to identifiability through the finite-dimensional coordinate map $\theta\mapsto\alpha_{\leq k}(\theta)$.
The regularity and stability of this parameter-side map are logically separate from the observation-side geometry of the coordinates $q_1,\ldots,q_k$, which will be studied below through their Gram structure. 

\subsection{Cross-order Gram geometry}
\label{subsec:cross_order_likelihood_geometry}

Assumption~\ref{ass:finite_dimensional_likelihood_geometry} provides a finite-dimensional coordinate representation for each likelihood order separately. However, the weak-signal likelihood combines these coordinates across different orders, and linear dependence between the corresponding coordinate spaces may allow contributions from different powers of $\beta$ to cancel. We therefore record a condition on the joint geometry of the likelihood coordinates.

Fix a truncation order $r\geq 1$, and stack the likelihood-coordinate vectors as
\begin{align}
    q_{\leq r}(Y) \triangleq \bigl( q_1(Y),\ldots,q_r(Y) \bigr) \in\mathbb{R}^{D_{\leq r}},
    \qquad
    D_{\leq r} \triangleq \sum_{\ell=1}^r D_\ell .
    \label{eq:stacked_likelihood_coordinates}
\end{align}
We define the corresponding Gram matrix under the null law $P_0$ by
\begin{align}
    \mathbf{G}_r \triangleq \mathbb{E}_{Y \sim P_0} \left[ q_{\leq r}(Y) q_{\leq r}(Y)^\top \right].    \label{eq:cross_order_likelihood_Gram}
\end{align}
Equivalently, $\mathbf{G}_r$ is the block matrix with $(\ell,k)$ block $\mathbf{G}_{\ell k} \triangleq \mathbb{E}_{ Y \sim P_0} \left[ q_\ell(Y)q_k(Y)^\top \right]$.

\begin{assum}[Nondegenerate cross-order likelihood geometry]
\label{ass:nondegenerate_cross_order_likelihood_geometry}
For the truncation order $r = r_{\mathrm{glob}}$, $\lambda_{\min}(\mathbf{G}_r)>0$.
\end{assum}

Assumption~\ref{ass:nondegenerate_cross_order_likelihood_geometry} requires the likelihood coordinates through order $r_{\mathrm{glob}}$ to be jointly linearly independent in $L^2(P_0)$. Importantly, this is an observation-side condition: it depends only on the null law $P_0$ and the likelihood-coordinate functions, and is separate from the algebraic identifiability of the parameter through the maps $\alpha_1,\ldots,\alpha_{r_\mathrm{glob}}$.

The next lemma, which is proved in Appendix~\ref{subsec:lem_fixed_accuracy_likelihood_separation}, shows that finite-dimensional likelihood geometry together with the cross-order nondegeneracy condition provides a concrete sufficient mechanism for the separation established in Theorem~\ref{thm:cross_order_stability_necessity}.

\begin{lem}[Uniform fixed-accuracy likelihood separation on generic compact sets]
\label{lem:fixed_accuracy_likelihood_separation}
Let $r=r_{\mathrm{glob}}<\infty$, and let $K\subseteq \Theta_{\mathrm{gen}}$ be compact.
Suppose Assumptions~\ref{ass:finite_dimensional_likelihood_geometry} and \ref{ass:nondegenerate_cross_order_likelihood_geometry} hold. Then, \eqref{eq:necessary_likelihood_separation} holds, that is, for every fixed $\varepsilon>0$, there exist constants $c_{\varepsilon}>0$ and $\beta_{\varepsilon}>0$ such that, for every $0<\beta\leq\beta_{\varepsilon}$,
\begin{align}
    \inf_{\substack{ \theta^\star\in K,\;\theta\in\Theta:\\ d_{\mathrm{eq}}(\theta,\theta^\star)\geq\varepsilon }} \left\| L_{\theta,\beta} - L_{\theta^\star,\beta} \right\|_{L^2(P_0)} \geq c_{\varepsilon}\beta^r.    \label{eq:fixed_accuracy_likelihood_separation}
\end{align}
\end{lem}

\subsection{Rate-optimality of maximum likelihood}

Maximum likelihood provides a natural benchmark for statistical efficiency~\cite{dempster1977maximum,cramer1999mathematical}. Given independent observations $Y_1,\ldots,Y_n\sim P_{\theta^\star,\beta}$, the \emph{maximum-likelihood estimator} (MLE) is any measurable maximizer
\begin{align}
    \widehat{\theta}_{\mathrm{MLE}} \in \argmax_{\theta\in\Theta} \sum_{i=1}^{n} \log p_{\theta,\beta}(Y_i).
    \label{eq:mle_noncancellation_estimator}
\end{align}

Lemma~\ref{lem:fixed_accuracy_likelihood_separation} shows that the finite-dimensional likelihood geometry provides the deterministic separation needed for fixed-accuracy recovery given in Theorem~\ref{thm:cross_order_stability_necessity}. To turn this separation into a uniform probability bound for the MLE, we must also control the stochastic fluctuations of the empirical likelihood simultaneously over $\theta\in\Theta$. We therefore impose the following uniform envelope condition on the Taylor remainder.

\begin{assum}[Uniform likelihood-remainder envelope]
\label{ass:uniform_likelihood_remainder_envelope}
Let $R_{2r+1}(\cdot;\theta,\beta)$ denote the remainder in the order-$2r$ likelihood-ratio expansion.  For all sufficiently small $\beta$,
\begin{align}
    \left\| \sup_{\theta\in\Theta} \left| R_{2r+1}(\cdot;\theta,\beta) \right| \right\|_{L^{4r+2}(P_0)} = O(\beta^{2r+1}).    \label{eq:uniform_likelihood_remainder_envelope_noncancellation}
\end{align}
\end{assum}

Combining the fixed-accuracy likelihood separation of Lemma~\ref{lem:fixed_accuracy_likelihood_separation} with the uniform likelihood-process control provided by Assumption~\ref{ass:uniform_likelihood_remainder_envelope} yields an exponential error bound for maximum likelihood at the scale $n\beta^{2r_{\mathrm{glob}}}$. The result is stated in the next theorem and proved in Appendix~\ref{subsec:thm_mle_upper_bound}.

\begin{thm}[Rate-optimality of maximum likelihood]
\label{thm:mle_upper_bound}
Let $r=r_{\mathrm{glob}}<\infty$, and let $K\subseteq\Theta_{\mathrm{gen}}$ be compact. Suppose Assumptions~\ref{ass:finite_dimensional_likelihood_geometry}, \ref{ass:nondegenerate_cross_order_likelihood_geometry}, and~\ref{ass:uniform_likelihood_remainder_envelope} hold. Then, for every fixed $\varepsilon>0$, there exist constants $c_\varepsilon,C_\varepsilon,\beta_0(\varepsilon)>0$ such that, for all $0<\beta\leq\beta_0(\varepsilon)$ and all $n$ satisfying $n\beta^{2r}\geq C_\varepsilon$, the MLE $\widehat\theta_{\mathrm{MLE}}$ defined in~\eqref{eq:mle_noncancellation_estimator}
satisfies
\begin{align}
    \sup_{\theta^\star\in K}  P_{\theta^\star,\beta}^{\otimes n} \left( d_{\mathrm{eq}} (\widehat\theta_{\mathrm{MLE}},\theta^\star) > \varepsilon \right) \leq C_\varepsilon \exp\left\{ -c_\varepsilon n\beta^{2r} \right\}.    \label{eq:mle_noncancellation_exponential_bound}
\end{align}
Consequently, $n \gtrsim  \beta^{-2r} \log\left(1/\delta\right)$ observations are sufficient for quotient error at most $\varepsilon$ with probability at least $1-\delta$.
\end{thm}

\subsection{Rate-optimality of the method-of-moments}
\label{subsec:likelihood_coordinate_moment_estimator}

We next construct a method-of-moments estimator using the finite-dimensional likelihood coordinates from Assumption~\ref{ass:finite_dimensional_likelihood_geometry}. For $r=r_{\mathrm{glob}}$, the observable vector $q_{\leq r}(Y)$ spans the likelihood expansion through the globally identifying order, so it is enough to estimate its expectation from the data. To obtain the required concentration, we employ a geometric median-of-means procedure to estimate these population moments~\cite{minsker2015geometric,lugosi2019mean}. Under the same cross-order nondegeneracy that guarantees likelihood separation, these coordinates retain the relevant finite-order likelihood geometry and yield the fixed-accuracy weak-signal exponent $\beta^{-2r_{\mathrm{glob}}}$.

Recall the stacked coordinate vector $q_{\leq r}$ defined in~\eqref{eq:stacked_likelihood_coordinates}, and define the corresponding population moment map by
\begin{align}
    \Gamma_{r,\beta}(\theta) \triangleq \mathbb{E}_{Y\sim P_{\theta,\beta}} \left[q_{\leq r}(Y)\right].    \label{eq:likelihood_coordinate_population_map}
\end{align}
Given a confidence level $\delta\in(0,1/2)$, let $\widehat{\Gamma}_{r,\delta}$ be a \emph{geometric median-of-means estimator} of $\Gamma_{r,\beta}(\theta^\star)$ constructed from $q_{\leq r}(Y_1),\ldots,q_{\leq r}(Y_n)$ using $O(\log(1/\delta))$ blocks~\cite{minsker2015geometric,lugosi2019mean}. We define
\begin{align}
    \widehat{\theta}_{\mathrm{MoM}} \in \argmin_{\theta\in\Theta} \, \left\| \widehat{\Gamma}_{r,\delta} - \Gamma_{r,\beta}(\theta)
    \right\|_2^2.    \label{eq:likelihood_coordinate_moment_estimator}
\end{align}
Thus, unlike the MLE, the method-of-moments estimator uses the likelihood coordinates explicitly as observable summary statistics.

\begin{remark}[General weighting matrices]
\label{rem:general_mom_weighting}
The estimator in~\eqref{eq:likelihood_coordinate_moment_estimator} uses the identity weighting matrix.
More generally, one may consider
\begin{align}
    \widehat{\theta}_{W} \in \argmin_{\theta\in\Theta} \left( \widehat{\Gamma}_{r,\delta} - \Gamma_{r,\beta}(\theta) \right)^\top W \left( \widehat{\Gamma}_{r,\delta} - \Gamma_{r,\beta}(\theta) \right),
\end{align}
where $W$ is any positive definite weighting matrix, including inverse-covariance choices~\cite{hansen1982large}. We use identity weighting because the choice of weighting matrix does not affect the weak-signal sample-complexity exponent considered here.
\end{remark}

Since the likelihood-coordinate dimension is fixed, robust mean estimation requires only uniformly bounded second moments.

\begin{lem}[Uniform second moments of likelihood coordinates]
\label{lem:bounded_second_moment_likelihood_coordinates}
Suppose Assumptions~\ref{ass:regular_observed_likelihood_ratio_expansion} and~\ref{ass:finite_dimensional_likelihood_geometry} hold.
Then, for every fixed truncation order $k\geq1$, there exist constants $V_k<\infty$ and $\beta_0>0$ such that
\begin{align}
    \sup_{\theta\in\Theta} \sup_{0<\beta\leq\beta_0} \mathbb{E}_{Y\sim P_{\theta,\beta}} \left[ \left\| q_{\leq k}(Y) - \Gamma_{k,\beta}(\theta) \right\|_2^2 \right] \leq V_k.    \label{eq:bounded_second_moment_likelihood_coordinates}
\end{align}
\end{lem}

By Lemma~\ref{lem:bounded_second_moment_likelihood_coordinates} (proved in Appendix~\ref{app:proof_of_bounded_second_moment_likelihood_coordinates}), standard geometric median-of-means concentration gives a constant $C_k<\infty$, depending only on the fixed coordinate dimension, such that
\begin{align}
    \sup_{\theta\in\Theta} P_{\theta,\beta}^{\otimes n} \left( \left\| \widehat{\Gamma}_{k,\delta} - \Gamma_{k,\beta}(\theta) \right\|_2 >  C_k \sqrt{\frac{V_k\log(1/\delta)}{n} } \right) \leq \delta,    \label{eq:robust_likelihood_coordinate_concentration}
\end{align}
for every $\delta\in(0,1/2)$ and all sufficiently small $\beta$.

The next theorem, proved in Appendix~\ref{subsec:proof_thm_mom_upper_bound}, shows that this second-moment condition is sufficient for the likelihood-coordinate method-of-moments estimator to attain the same fixed-accuracy weak-signal sample-complexity rate as maximum likelihood.

\begin{thm}[Rate-optimality of the method-of-moments estimator]
\label{thm:mom_upper_bound}
Let $r=r_{\mathrm{glob}}<\infty$, and let $K\subseteq\Theta_{\mathrm{gen}}$ be compact.
Suppose Assumptions~\ref{ass:finite_dimensional_likelihood_geometry}, \ref{ass:nondegenerate_cross_order_likelihood_geometry}, and the conditions of Lemma~\ref{lem:bounded_second_moment_likelihood_coordinates} hold. Then, for every fixed $\varepsilon>0$, there exist constants $C_{\varepsilon},\beta_{\varepsilon}>0$ such that, for every $0<\beta\leq\beta_{\varepsilon}$ and every $\delta\in(0,1/2)$, $n \geq C_{\varepsilon} \beta^{-2r} \log (1/\delta)$ implies
\begin{align}
    \sup_{\theta^\star\in K} P_{\theta^\star,\beta}^{\otimes n} \left( d_{\mathrm{eq}} \bigl( \widehat{\theta}_{\mathrm{MoM}}, \theta^\star \bigr) > \varepsilon \right) \leq \delta.    \label{eq:likelihood_coordinate_moment_exponential_bound}
\end{align}
Consequently, the method-of-moments estimator attains the fixed-accuracy sample-complexity scale $n \lesssim \beta^{-2r_{\mathrm{glob}}} \log(1/\delta)$.
\end{thm}

\section{Global-to-local sample complexity under algebraic likelihood geometry}
\label{sec:global_to_local_algebraic}

Sections~\ref{sec:general_framework}-\ref{sec:rate-optimal-recovery} characterize weak-signal recovery when the target accuracy $\varepsilon>0$ is fixed. We now study the joint regime $\beta\to 0$ and $\varepsilon\to 0$. This refinement exposes a distinction that is invisible at fixed accuracy: the order required to \emph{locate the correct equivalence class globally} need not equal the order required to \emph{resolve increasingly small perturbations once the target has been localized}. The former is the global identifiability order $r_{\mathrm{glob}}$, whereas the latter is the local identifiability order $r_{\mathrm{loc}}$.

The purpose of this section is to make this distinction precise and to show how it leads to two statistical scales. Global localization is governed by $\beta^{-2r_{\mathrm{glob}}}$, while local refinement to accuracy $\varepsilon$ is governed by $\beta^{-2r_{\mathrm{loc}}}\varepsilon^{-2}$. We first introduce the additional algebraic and local regularity needed for the shrinking-accuracy analysis, then establish the corresponding population stability, and finally state matching upper and lower bounds.


\subsection{Algebraic structure}
\label{subsec:algebraic_regular_parameter_set}

For the global-to-local analysis in this section, we impose the additional assumption that the parameter space is a \emph{compact, semi-algebraic\footnote{A set is semi-algebraic if it can be described by finitely many polynomial equalities and inequalities.}} set $\Theta\subseteq\mathbb{R}^m$. This additional algebraic structure is natural because the finite-dimensional likelihood-moment maps considered below are assumed to depend polynomially on the parameter. It places the parameter space and the likelihood-moment maps within a common finite-dimensional algebraic framework, allowing generic identifiability, local regularity, and stability to be analyzed using standard tools from semi-algebraic geometry.

The semi-algebraic setting still covers a broad range of structured parameter spaces arising in applications. Examples include compact subsets of linear or affine subspaces, including subspaces obtained by principal component analysis; sparse and structured-sparse models, which are naturally described by finite unions of low-dimensional subspaces~\cite{donoho2006compressed,elad2010sparse}; compact algebraic manifolds and more general stratified algebraic sets; and parameter sets represented by piecewise-polynomial or piecewise-linear generative maps. In particular, ranges of ReLU networks are semi-algebraic and provide an important class of such structured models~\cite{bendory2025phase,bendory2025transversality}.

\begin{assum}[Algebraic parameter structure]
\label{ass:algebraic_parameter_structure}
Assume that $\Theta\subseteq\mathbb{R}^m$ is compact and semi-algebraic and that, for every $\ell\geq 1$, the coordinate map $\alpha_\ell:\Theta\to\mathbb{R}^{D_\ell}$ defined in~\eqref{eqn:finite_dimentional_repr} is the restriction to $\Theta$ of a polynomial map in the ambient parameter coordinates.
\end{assum}

The preceding assumption concerns only the finite-dimensional parameter-side geometry. For the shrinking-accuracy analysis, we also require a mild smoothness condition ensuring that the weak-signal likelihood expansion remains uniform after one derivative with respect to the parameter.

\begin{assum}[Local likelihood regularity]
\label{ass:local_likelihood_regularity}
For every finite $p\geq 1$, there exists an open neighborhood $U\supset\Theta$ such that, for all sufficiently small $\beta$, the map $\theta \mapsto L_{\theta,\beta}$ admits a continuously differentiable extension from $\Theta$ to $U$ as an $L^p(P_0)$-valued map. Moreover, the likelihood-ratio expansion in Assumption~\ref{ass:regular_observed_likelihood_ratio_expansion} holds uniformly over $\Theta$ together with its first derivative with respect to $\theta$.
\end{assum}

In particular, for every fixed $k\geq 1$,
\begin{align}
    D_\theta L_{\theta,\beta} = \sum_{\ell=1}^{k}\beta^\ell D\Phi_\ell(\theta) + R^{(1)}_{k+1}(\theta,\beta), \label{eq:differentiated_likelihood_expansion}
\end{align}
where, uniformly over $\theta\in\Theta$ as $\beta\to 0$,
\begin{align}
    \sup_{\theta\in\Theta} \norm{R^{(1)}_{k+1}(\theta,\beta)}_{L^p(P_0)} = O(\beta^{k+1}).
\end{align}

\subsection{Local identifiability likelihood-moment order}
\label{subsec:local_informative_order}

We next introduce the local analogue of the global informative likelihood-moment order. Let $W_{\theta^\star}$ denote the identifiable tangent space at the target, namely, the tangent directions transverse to the observational-equivalence class; see Definition~\ref{def:tangent_normal_quotient} for a formal definition. Whereas the global order is determined by the values of the likelihood-moment maps over the entire parameter space, the local order records the first order at which their differentials detect every identifiable infinitesimal perturbation.

\begin{definition}[Local identifiability likelihood-moment order]
\label{def:rloc_likelihood}
Let $\theta^\star\in\Theta$ be the ground-truth parameter. For each
$h\in W_{\theta^\star} \setminus\{0\}$, define
\begin{align}
    r_{\mathrm{loc}}(h;\theta^\star) \triangleq \inf \left\{ j\ge 1: D\Phi_j(\theta^\star)[h] \neq 0 \text{ in }L^2(P_0) \right\},   \label{eq:def_directional_local_likelihood_order}
\end{align}
with the convention that $r_{\mathrm{loc}}(h;\theta^\star) = \infty$ if no such order exists. The \emph{local identifiability likelihood-moment order} at $\theta^\star$ is
\begin{align}
    r_{\mathrm{loc}}(\theta^\star) \triangleq \sup_{ h\in W\setminus\{0\} } r_{\mathrm{loc}}(h;\theta^\star) \in \{1,2,\ldots\}\cup\{\infty\}.
    \label{eq:rloc_likelihood_def}
\end{align}
\end{definition}

Equivalently, $r_{\mathrm{loc}}(\theta^\star)$ is the smallest order $k$
such that the differential of the truncated likelihood-moment map $D\Phi_{\le k}(\theta^\star) : W_{\theta^\star} \to L^2(P_0)^k$ is injective. Indeed,
\begin{align}
    r_{\mathrm{loc}}(\theta^\star) = \inf \left\{k\ge 1: \ker D\Phi_{\le k}(\theta^\star) \cap W_{\theta^\star} = \{0\} \right\}.    \label{eq:rloc_equivalent_differential_injectivity}
\end{align}
Thus, $r_{\mathrm{loc}}(\theta^\star)$ is the first likelihood-moment order at which every nonzero identifiable perturbation is detected to first order.

\begin{proposition}[Generic relation between the local and global identifiability orders]
\label{prop:generic_rloc_rglob}
Suppose that Assumption~\ref{ass:algebraic_parameter_structure} holds, and let $r_{\mathrm{glob}}<\infty$. Then, there exists an exceptional set of measure zero such that, for every target $\theta^\star$ outside this exceptional set,
\begin{align}
    r_{\mathrm{loc}}(\theta^\star) \leq r_{\mathrm{glob}}. \label{eq:generic_rloc_le_rglob}
\end{align}
\end{proposition}

Proposition~\ref{prop:generic_rloc_rglob}, which is proved in Appendix~\ref{subsec:proof_prop_generic_rloc_rglob}, does not assert equality between the two orders. The local order may be strictly smaller than the global order: lower-order likelihood moments may distinguish all infinitesimal directions near a generic target while still failing to separate distant equivalence classes; see~\cite[Example 3.10]{balanov2026generalized}. Conversely, the inequality may fail at exceptional targets, such as signals with additional symmetries or other differential degeneracies.

We denote by $\Theta_{\mathrm{reg}}\subseteq\Theta_{\mathrm{gen}}$ a generic subset on which the local quotient structure is regular, $r_{\mathrm{loc}}(\theta^\star) \leq r_{\mathrm{glob}}$, and the ranks of the truncated coordinate maps $\alpha_{\leq k}$, $1\leq k\leq r_{\mathrm{glob}}$, are locally constant. By Proposition~\ref{prop:generic_rloc_rglob} and the generic quotient regularity established in Appendix~\ref{sec:smooth_quotient_models}, the complement of $\Theta_{\mathrm{reg}}$ is contained in a lower-dimensional exceptional set. Henceforth, we fix a compact set $K\subseteq\Theta_{\mathrm{reg}}$ on which $r_{\mathrm{loc}}(\theta^\star)$ is constant, and denote this common value by $r_{\mathrm{loc}}$.

\subsection{Population stability across local and global scales}
\label{subsec:population_stability_local_global}

The next proposition shows that at order $r_{\mathrm{loc}}$, the likelihood coordinates locally control the quotient distance to the target, whereas order $r_{\mathrm{glob}}$ is sufficient to separate the target from all parameters outside a fixed neighborhood. The proof is deferred to Appendix~\ref{subsec:proof_prop_population_stability_local_global}.

\begin{proposition}[Population stability across local and global scales]
\label{prop:population_stability_local_global}
Suppose Assumption~\ref{ass:algebraic_parameter_structure} holds.
Let $K \subseteq \Theta_{\mathrm{reg}}$ be the compact regular parameter class fixed above, on which $r_{\mathrm{loc}}$ is constant. Then, there exist constants $\eta>0$, $C_{\mathrm{inv}}<\infty$, and $\gamma_\eta>0$ such that, uniformly over $\theta^\star\in K$, the following hold.
\begin{enumerate}
    \item \emph{(Local stable invertibility.)}
    For every $\theta\in\Theta$ satisfying
    $d_{\mathrm{eq}}(\theta,\theta^\star)\leq\eta$,
    \begin{align}
        d_{\mathrm{eq}}(\theta,\theta^\star) \leq C_{\mathrm{inv}} \left\|\alpha_{\leq r_{\mathrm{loc}}}(\theta) - \alpha_{\leq r_{\mathrm{loc}}}(\theta^\star) \right\|_2.        \label{eq:population_local_stability_alpha}
    \end{align}

    \item \emph{(Global separation.)}
    For every $\theta\in\Theta$ satisfying
    $d_{\mathrm{eq}}(\theta,\theta^\star)\geq\eta$,
    \begin{align}
        \left\| \alpha_{\leq r_{\mathrm{glob}}}(\theta) -\alpha_{\leq r_{\mathrm{glob}}}(\theta^\star) \right\|_2 \geq \gamma_\eta.        \label{eq:population_global_separation_alpha}
    \end{align}
\end{enumerate}
\end{proposition}

For the statistical analysis, Proposition~\ref{prop:population_stability_local_global} must be transferred from the parameter-side coordinates $\alpha_\ell$ to the observed likelihood. The cross-order nondegeneracy condition in Assumption~\ref{ass:nondegenerate_cross_order_likelihood_geometry} prevents cancellation among the likelihood coordinates, while Assumption~\ref{ass:local_likelihood_regularity} makes the higher-order remainder proportional to the local displacement. Together they yield the following two-scale likelihood separation, which is proved in Appendix~\ref{subsec:proof_prop_likelihood_separation_local_global}.

\begin{proposition}[Likelihood separation across local and global scales]
\label{prop:likelihood_separation_local_global}
Suppose Assumptions~\ref{ass:nondegenerate_cross_order_likelihood_geometry}, \ref{ass:algebraic_parameter_structure}, and \ref{ass:local_likelihood_regularity} hold. Then, there exist constants $c_{\mathrm{loc}},c_{\mathrm{glob}}>0$ and $\beta_0>0$ such that, uniformly over $\theta^\star\in K$ and $0<\beta\leq\beta_0$,
\begin{align}
    \norm{L_{\theta,\beta}-L_{\theta^\star,\beta}}_{L^2(P_0)}
    &\geq
    c_{\mathrm{loc}}\, \beta^{r_{\mathrm{loc}}} d_{\mathrm{eq}}(\theta,\theta^\star), && d_{\mathrm{eq}}(\theta,\theta^\star)\leq\eta,   \label{eq:local_likelihood_separation_r_loc}
    \\
    \norm{L_{\theta,\beta}-L_{\theta^\star,\beta}}_{L^2(P_0)}
    &\geq
    c_{\mathrm{glob}}\, \beta^{r_{\mathrm{glob}}},
    && d_{\mathrm{eq}}(\theta,\theta^\star)\geq\eta.   \label{eq:global_likelihood_separation_r_glob}
\end{align}
\end{proposition}

The first inequality is the key new ingredient for shrinking accuracy: two nearby parameters at quotient distance $\varepsilon$ are statistically separated at scale $\beta^{r_{\mathrm{loc}}}\varepsilon$. The second is a fixed-radius statement, so the corresponding separation is of order $\beta^{r_{\mathrm{glob}}}$ with no additional factor of $\varepsilon$.

\subsection{Upper bound: global localization followed by local refinement}
\label{subsec:global_to_local_upper_bound}

Because the likelihood moments through order $r_{\mathrm{loc}}$ need not identify the target globally, an estimator based only on these moments cannot in general be applied over the full parameter space. This motivates a two-stage procedure. The first stage uses the fixed-accuracy theory of Section~\ref{sec:rate-optimal-recovery} to localize the target to a fixed neighborhood at the global order $r_{\mathrm{glob}}$. The second stage restricts estimation to this neighborhood and refines the estimate using likelihood coordinates only through order $r_{\mathrm{loc}}$.

Concretely, we split the $n$ observations into two independent subsamples of sizes $n_1$ and $n_2$, with $n_1+n_2=n$. Using the first subsample, we construct a preliminary estimator $\widetilde{\theta}$ satisfying $d_{\mathrm{eq}}(\widetilde{\theta},\theta^\star)\leq\eta/4$ with probability at least $1-\delta/2$, using the fixed-accuracy robust method-of-moments estimator from Theorem~\ref{thm:mom_upper_bound}. 
Using the second subsample, $Y_1^{(2)},\ldots,Y_{n_2}^{(2)}$, let $\widehat{\Gamma}_{r_{\mathrm{loc}},\delta/2}$ denote the geometric median-of-means estimator of $\Gamma_{r_{\mathrm{loc}},\beta}(\theta^\star)$ (as defined in~\eqref{eq:likelihood_coordinate_population_map}), constructed from $q_{\leq r_{\mathrm{loc}}}(Y_i^{(2)})$, for $1\leq i\leq n_2$, following the procedure introduced in Section~\ref{subsec:likelihood_coordinate_moment_estimator}. We then restrict the search to $\Theta_{\mathrm{loc}}(\widetilde{\theta}) \triangleq \{\theta\in\Theta: d_{\mathrm{eq}}(\theta,\widetilde{\theta}) \leq \eta/2 \}$, and define
\begin{align}
    \widehat{\theta}_{\mathrm{loc}} \in \argmin_{\theta\in\Theta_{\mathrm{loc}}(\widetilde{\theta})} \left\| \widehat{\Gamma}_{r_{\mathrm{loc}},\delta/2} - \Gamma_{r_{\mathrm{loc}},\beta}(\theta) \right\|_2^2.
\label{eq:local_likelihood_coordinate_moment_estimator}
\end{align}
The resulting two-stage sample-complexity bound is stated below, with the proof deferred to Appendix~\ref{subsec:two_stage_upper_bound}.

\begin{thm}[Global-to-local upper bound]
\label{thm:global_to_local_upper_bound}
Let $K\subseteq\Theta_{\mathrm{reg}}$ be the compact regular generic parameter class fixed above.
Suppose the assumptions of the fixed-accuracy method-of-moments result in Theorem~\ref{thm:mom_upper_bound} hold, together with Assumptions~\ref{ass:algebraic_parameter_structure} and~\ref{ass:local_likelihood_regularity}.
Then there exist constants $C_0<\infty$, $\varepsilon_0>0$, and $\beta_0>0$ such that, for every $0<\varepsilon\leq\varepsilon_0$, $0<\beta\leq\beta_0$, and $0<\delta<1/2$, the two-stage robust method-of-moments estimator satisfies
\begin{align}
    \sup_{\theta^\star\in K}
    P_{\theta^\star,\beta}^{\otimes n} \left( d_{\mathrm{eq}} \bigl( \widehat{\theta}_{\mathrm{loc}}, \theta^\star \bigr) > \varepsilon \right) \leq \delta
\end{align}
whenever
\begin{align}
    n \geq C_0 \max\left\{\beta^{-2r_{\mathrm{glob}}}, \beta^{-2r_{\mathrm{loc}}}\varepsilon^{-2} \right\} \log \! \left(1/{\delta}\right).    \label{eq:global_to_local_upper_sample_complexity}
\end{align}
\end{thm}

\subsection{Lower bound: global and local hard alternatives}
\label{subsec:global_to_local_lower_bound}

The lower bound reflects the same two scales as the upper bound: globally separated alternatives yield the term $\beta^{-2r_{\mathrm{glob}}}$, whereas nearby alternatives at a regular target yield the local term $\beta^{-2r_{\mathrm{loc}}}\varepsilon^{-2}$. Here, by hard alternatives we mean pairs of parameter values that are separated in the parameter space but statistically difficult to distinguish, requiring a large number of observations. The next lemma constructs such local hard alternatives, whose likelihood moments agree through order $r_{\mathrm{loc}}-1$; its proof is given in Appendix~\ref{subsec:proof_lem_local_hard_alternatives}.

\begin{lem}[Local hard alternatives at order $r_{\mathrm{loc}}$]
\label{lem:local_hard_alternatives}
Suppose Assumption~\ref{ass:algebraic_parameter_structure} holds. Let $K\subseteq \Theta_{\mathrm{reg}}$ be the compact parameter class fixed above, and suppose it contains a relative neighborhood of some $\theta_0\in\Theta_{\mathrm{reg}}$.
Then, there exist constants $c_0,C_0,t_0>0$ and a continuously differentiable curve $t\mapsto\theta(t)\in K$, $|t|\leq t_0$, with $\theta(0)=\theta_0$, such that
\begin{align}
    c_0|t| \leq  d_{\mathrm{eq}}(\theta(t),\theta_0) \leq C_0|t|,  
    \qquad 
    \Phi_{\leq r_{\mathrm{loc}}-1}(\theta(t)) = \Phi_{\leq r_{\mathrm{loc}}-1}(\theta_0), \label{eq:local_hard_curve_distance}
\end{align}
and
\begin{align}
    c_0|t| \leq \norm{ \Phi_{r_{\mathrm{loc}}}(\theta(t)) - \Phi_{r_{\mathrm{loc}}}(\theta_0)}_{L^2(P_0)} \leq C_0|t|.    \label{eq:local_hard_curve_informative_order}
\end{align}
\end{lem}

Choosing $|t|\asymp\varepsilon$ yields a pair of parameters at quotient distance of order $\varepsilon$ whose likelihood moments agree through order $r_{\mathrm{loc}}-1$ and first differ by order $\varepsilon$ at order $r_{\mathrm{loc}}$. Together with Assumption~\ref{ass:local_likelihood_regularity}, this gives a divergence of order at most $\beta^{2r_{\mathrm{loc}}}\varepsilon^2$, and hence the local lower-bound scale $\beta^{-2r_{\mathrm{loc}}}\varepsilon^{-2}$. Combining this with the fixed-accuracy global lower bound yields the following theorem, proved in Appendix~\ref{subsec:proof_thm_global_to_local_lower_bound}.

\begin{thm}[Global-to-local minimax lower bound]
\label{thm:global_to_local_lower_bound}
Let $K\subseteq \Theta_{\mathrm{reg}}$ be the compact parameter class fixed above. Suppose that $K$ contains a relative neighborhood of at least one point of $\Theta_{\mathrm{reg}}$, so that Lemma~\ref{lem:local_hard_alternatives} applies.
Suppose further that $K$ contains a fixed pair $\theta_0,\theta_1$ satisfying $d_{\mathrm{eq}}(\theta_0,\theta_1)>0$, and $r(\theta_0,\theta_1)=r_{\mathrm{glob}}$. Assume, in addition, that Assumptions~\ref{ass:reciprocal_likelihood_regularity} and~\ref{ass:local_likelihood_regularity} hold. Then there exist constants $c_0,\varepsilon_0,\beta_0>0$ such that, for every $0<\varepsilon\leq\varepsilon_0$, $0<\beta\leq\beta_0$, and $0<\delta<1/4$,
\begin{align}
    n^\star(\varepsilon,\delta;\beta) \geq c_0 \max\left\{\beta^{-2r_{\mathrm{glob}}}, \beta^{-2r_{\mathrm{loc}}}\varepsilon^{-2}\right\} \log\left(1/\delta\right).    \label{eq:global_to_local_lower_rate}
\end{align}
\end{thm}

The two terms in~\eqref{eq:global_to_local_lower_rate} require different conditions. The global hard-pair assumption is needed only for the $\beta^{-2r_{\mathrm{glob}}}$ term, whereas the local term follows from the regular local geometry around a single target. Consequently, restricting the parameter class may reduce the global localization difficulty if the restricted class no longer contains a pair realizing $r_{\mathrm{glob}}$, while leaving the local refinement scale $\beta^{-2r_{\mathrm{loc}}}\varepsilon^{-2}$ unchanged.
    
\paragraph{The global-to-local transition.}
Combining Theorems~\ref{thm:global_to_local_upper_bound} and~\ref{thm:global_to_local_lower_bound} yields the two-scale sample-complexity law~\eqref{eq:intro_global_local_sample_complexity}. The two terms are equal when $\varepsilon \asymp \beta^{r_{\mathrm{glob}}-r_{\mathrm{loc}}}$. This crossover separates two regimes. When $\varepsilon\gg\beta^{r_{\mathrm{glob}}-r_{\mathrm{loc}}}$, the dominant difficulty is identifying the correct equivalence class globally, and the sample complexity is governed by $r_{\mathrm{glob}}$. When $\varepsilon\ll\beta^{r_{\mathrm{glob}}-r_{\mathrm{loc}}}$, the dominant difficulty is fine local estimation, and the sample complexity is governed by $r_{\mathrm{loc}}$ together with the usual $\varepsilon^{-2}$ refinement factor.

In the important special case $r_{\mathrm{loc}}=r_{\mathrm{glob}}=r$, the global and local identifiability orders coincide and the shrinking-accuracy law reduces to $n^\star(\varepsilon,\delta;\beta) \asymp \beta^{-2r}\varepsilon^{-2} \log\left(1/\delta\right)$. Thus, the direct $\beta^{-2r}\varepsilon^{-2}$ law is a special case of the more general global-to-local theory, rather than an assumption imposed a priori.

\section{Applications to algebraic weak-signal perturbation models}
\label{sec:applications}

The preceding sections characterize weak-signal sample complexity in terms of finite-dimensional likelihood geometry and the global and local identifiability likelihood-moment orders. We now specialize this framework to a broad class of algebraic linear perturbation models, for which the abstract likelihood coordinates admit an explicit representation in terms of moments of the latent signal and derivatives of the observation model.

We proceed in three steps.
First, we introduce a general class of \emph{algebraic weak-signal perturbation models}, in which the parameter of a known observation family is perturbed linearly in the weak-signal parameter $\beta$, while the latent signal depends polynomially on $\theta$. We then specialize this framework to two important families of observation models: additive-noise channels and exponential-family models. These specializations provide concrete likelihood coordinates and verifiable conditions under which the general sample-complexity results of Sections~\ref{sec:rate-optimal-recovery}-\ref{sec:global_to_local_algebraic} apply.

\subsection{Algebraic weak-signal perturbations}
\label{subsec:linear_weak_signal_perturbations}

We first consider a parametric observation family whose parameter is perturbed \emph{linearly} in the weak-signal parameter $\beta$, with a perturbation determined by a latent signal depending polynomially on $\theta$. 

\begin{model}[Linear weak-signal perturbation model]
\label{model:linear_weak_input}
Let $\theta\in\Theta$, let $Z\sim\mu$, and let $\{P_m:m\in\mathcal{M}\}$ be a known parametric family on $\mathcal{Y}$ with densities $p_m$, where $m$ ranges over a neighborhood of $m_0\in\mathcal{M}\subset\mathbb{R}^d$. For a known map $s:\mathcal Z\times\Theta\to\mathbb{R}^d$ that is polynomial in $\theta$, assume that $p_{\theta,\beta}(y\vert z) = p_{m_0+\beta s(z;\theta)}(y)$; that is, conditional on $Z=z$, the parameter of the observation distribution is $m_{\theta,\beta}(z)=m_0+\beta s(z;\theta)$, namely, a linear perturbation in the signal strength $\beta$ around the null value $m_0$.
\end{model}

We next impose standard regularity on the observation family and the latent signal map. The first condition ensures a uniform likelihood expansion around the null model, while the second controls the weak-signal perturbation and its variation with respect to the parameter.

\begin{assum}[Regular linear perturbation model]
\label{ass:regular_linear_perturbation_model}
For every finite $p\geq 1$, the map $m\mapsto p_m(y)/p_{m_0}(y)$ is real analytic as an $L^p(P_0)$-valued map in a neighborhood of $m_0$. Moreover, $s(z;\cdot)$ is continuously differentiable on $\Theta$ for every $z\in\mathcal Z$, and
\begin{align}
    \sup_{\theta\in\Theta}\sup_{z\in\mathcal Z} \left( \norm{s(z;\theta)} + \norm{D_\theta s(z;\theta)} \right) <\infty. \label{eq:uniform_signal_and_derivative_bound}
\end{align}
\end{assum}

A particularly important special case is $s(z;\theta)=A(z)\theta$, where $A(z)$ is a linear operator, so that the latent signal depends linearly on $\theta$. This is a canonical and extensively studied setting in the weak-signal latent-variable literature; see Appendix~\ref{subsec:linear_mean_scaling_latent_variable_models} for representative models and applications captured by this framework.

For $k\geq 1$, define the order-$k$ signal moment
\begin{align}
    M_k(\theta) \triangleq \mathbb{E}_{Z\sim\mu} \left[ s(Z;\theta)^{\otimes k} \right], \label{eq:linear_weak_input_signal_moment}
\end{align}
and the order-$k$ observation side coordinates,
\begin{align}
    Q_k(y) \triangleq \frac{1}{k!} \left. \nabla_m^{\otimes k} \left( \frac{p_m(y)}{p_{m_0}(y)} \right) \right|_{m=m_0}.    \label{eq:linear_weak_input_Qk}
\end{align}

\begin{lem}[Likelihood moments for the linear weak-signal perturbation model]
\label{lem:linear_weak_input_likelihood_moments}
Under Model~\ref{model:linear_weak_input} and Assumption~\ref{ass:regular_linear_perturbation_model},
\begin{align}
    \frac{p_{\theta,\beta}(y\vert z)}{p_0(y)} = 1+ \sum_{k\geq 1} \beta^k \left\langle  Q_k(y), s(z;\theta)^{\otimes k} \right\rangle,    \label{eq:linear_weak_input_conditional_expansion}
\end{align}
and consequently
\begin{align}
    \phi_k(y;\theta) = \left\langle Q_k(y), M_k(\theta) \right\rangle, \qquad k\geq 1.    \label{eq:linear_weak_input_likelihood_moment}
\end{align}
\end{lem}

\begin{proof}
Taylor expansion of $p_m(y)/p_{m_0}(y)$ at $m_0$ along the perturbation $m=m_0+\beta s(z;\theta)$ gives \eqref{eq:linear_weak_input_conditional_expansion}.
Averaging over $Z$ and using $M_k(\theta)=\mathbb{E}_{Z \sim \mu}[s(Z;\theta)^{\otimes k}]$ yields~\eqref{eq:linear_weak_input_likelihood_moment}.
\end{proof}

The preceding likelihood-moment representation separates the model into two components: the signal moments $M_k(\theta)$, which carry the parameter dependence, and the observation-side coordinates $Q_k$, which are determined entirely by the channel. Consequently, the general recovery theory of Sections~\ref{sec:rate-optimal-recovery}-\ref{sec:global_to_local_algebraic} applies once the corresponding parameter-side algebraic structure, observation-side nondegeneracy, and concentration conditions are verified.

\paragraph{The method-of-moments for algebraic weak-signal perturbations.}
For $r\geq1$, let $Q_{\leq r}(y) \triangleq (Q_1(y),\ldots,Q_r(y))$, where each tensor is represented in any fixed finite-dimensional coordinate system. Define the corresponding population moment map by
\begin{align}
    \Gamma_{r,\beta}(\theta) \triangleq \mathbb{E}_{Y\sim P_{\theta,\beta}} \left[ Q_{\leq r}(Y)\right].    \label{eq:channel_coordinate_population_map}
\end{align}
Given a confidence level $\delta\in(0,1/2)$, let $\widehat{\Gamma}_{r,\delta}$ denote a geometric median-of-means estimator of $\Gamma_{r,\beta}(\theta^\star)$ constructed from $Q_{\leq r}(Y_1),\ldots,Q_{\leq r}(Y_n)$, following the procedure introduced in Section~\ref{subsec:likelihood_coordinate_moment_estimator}. We define the method-of-moments estimator by
\begin{align}
    \widehat{\theta}_{\mathrm{MoM}} \in \argmin_{\theta\in\Theta} \| \widehat{\Gamma}_{r,\delta} - \Gamma_{r,\beta}(\theta) \|_2^2.
\label{eq:channel_coordinate_mom_estimator}
\end{align}

\begin{proposition}[Recovery for algebraic linear perturbations]
\label{prop:linear_perturbation_verification}
Consider Model~\ref{model:linear_weak_input}, let $r=r_{\mathrm{glob}}<\infty$, and suppose that $\Theta$ is compact and semi-algebraic.
Assume that:
\begin{enumerate}[label=(\roman*)]
    \item the observational equivalence is induced by the compact group action in Assumption~\ref{ass:group_induced_equivalence};
    
    \item Assumption~\ref{ass:regular_linear_perturbation_model} holds;
    
    \item the likelihood coordinates induced by $Q_{\leq r}(Y)$ satisfy the analogue of the cross-order nondegeneracy condition in Assumption~\ref{ass:nondegenerate_cross_order_likelihood_geometry}.

\end{enumerate}
Then the method-of-moments estimator $\widehat{\theta}_{\mathrm{MoM}}$ in~\eqref{eq:channel_coordinate_mom_estimator} attains the fixed-accuracy upper bound of Theorem~\ref{thm:mom_upper_bound}. Moreover, the corresponding two-stage method-of-moments estimator, based on geometric median-of-means estimation in both stages, attains the global-to-local upper bound of Theorem~\ref{thm:global_to_local_upper_bound}.
\end{proposition}

Among the conditions in Proposition~\ref{prop:linear_perturbation_verification} (proved in Appendix~\ref{sec:proof_linear_perturbation_verification}), the main structural requirement is the cross-order nondegeneracy condition in Assumption~\ref{ass:nondegenerate_cross_order_likelihood_geometry}, which rules out cancellations between distinct likelihood orders. The second-moment condition of Lemma~\ref{lem:bounded_second_moment_likelihood_coordinates} is used only to control estimation of the finite-dimensional population moment map. In fact, Assumption~\ref{ass:regular_linear_perturbation_model} already implies that, for every fixed $r$, the coordinates $Q_{\leq r}(Y)$ have uniformly bounded second moments under $P_{\theta,\beta}$ for all sufficiently small $\beta$.

\paragraph{Direct extraction of signal moments.}
The representation in~\eqref{eq:linear_weak_input_likelihood_moment} separates the likelihood moments into the observation-side coordinates $Q_k$ and the signal-only moments $M_k(\theta)$.
For each $k\geq 1$, let $\mathcal{V}_k \triangleq \operatorname{span} \{ M_k(\theta):\theta\in\Theta\}$ denote the finite-dimensional signal-moment space of order $k$. In many important observation models, there exist observable tensor-valued features $\psi_k:\mathcal{Y}\to\mathcal{V}_k$ that are dual to the likelihood coordinates in the sense of biorthogonality under the null law $P_0$: their pairing with the likelihood coordinates reproduces the corresponding signal tensor at the same order and vanishes across different orders. Specifically, for every $j,k\geq 1$ and every $M\in\mathcal{V}_j$, consider the relations
\begin{align}
    \mathbb{E}_{Y\sim P_0}[\psi_k(Y)]
    &=0,
    \label{eq:dual_coordinate_centering}
    \\
    \mathbb{E}_{Y\sim P_0} \left[ \psi_k(Y) \left\langle Q_j(Y),M\right\rangle \right]
    &=
    \begin{cases}
        M, & j=k,\\
        0, & j\neq k.
    \end{cases}
\label{eq:dual_coordinate_biorthogonality}
\end{align}
The following lemma shows that these biorthogonality relations are equivalent to exact extraction of the individual signal moments from population expectations.

\begin{lem}[Dual coordinates and exact signal-moment recovery]
\label{lem:dual_coordinates_signal_moments}
The following statements are equivalent:
\begin{enumerate}
    \item The observable features $\{\psi_k\}_{k\geq 1}$ satisfy     \eqref{eq:dual_coordinate_centering}--\eqref{eq:dual_coordinate_biorthogonality}.
    
    \item For every $k\geq 1$, every $\theta\in\Theta$, and all sufficiently small $\beta$,
    \begin{align}
        \mathbb{E}_{Y\sim P_{\theta,\beta}}[\psi_k(Y)] = \beta^k M_k(\theta).   \label{eq:direct_signal_moment_extraction}
    \end{align}    
\end{enumerate}
Whenever these equivalent conditions hold, for every finite $r\geq 1$, the likelihood coordinates $Q_1,\ldots,Q_r$ are jointly linearly independent on the corresponding signal-moment spaces. Consequently, the corresponding cross-order Gram matrix is positive definite, and Assumption~\ref{ass:nondegenerate_cross_order_likelihood_geometry} holds.
\end{lem}

Lemma~\ref{lem:dual_coordinates_signal_moments}, which is proved in Appendix~\ref{app:proof_dual_coordinates_signal_moments}, gives an equivalent characterization of exact signal-moment extraction: observable features recover each $M_k(\theta)$ exactly after the natural rescaling by $\beta^k$ if and only if they are biorthogonal, order by order, to the likelihood coordinates.
Either characterization also implies cross-order nondegeneracy for every finite truncation.
We next show how this structure arises in two important classes of observation models: additive-noise channels, through Appell polynomials, and exponential families, through null-adapted polynomial coordinates.

\subsection{Additive-noise observation channels}
\label{subsec:additive_noise_channels}

We now specialize the algebraic linear perturbation framework to additive-noise channels. This is obtained from Model~\ref{model:linear_weak_input} by taking the observation family to be the translation family $p_m(y) = f_0(y-m)$, with $m_0=0$. Accordingly, the parameter-side quantities remain the signal moments $M_k(\theta)$ from~\eqref{eq:linear_weak_input_signal_moment}, whereas the observation-side likelihood coordinates are determined entirely by the translation geometry of the noise law.
Our main goal in this subsection is to show that these coordinates satisfy the cross-order nondegeneracy condition in Assumption~\ref{ass:nondegenerate_cross_order_likelihood_geometry}. We also identify explicit observable statistics that recover the signal moments exactly.

\begin{model}[Additive-noise latent-variable model]
\label{model:additive_noise_channel}
Let $Z\sim\mu$, let $\theta\in\Theta$ be the unknown parameter, and let $s(z;\theta)\in\mathbb{R}^d$ denote the latent signal contribution. We observe
\begin{align}
    Y = \beta s(Z;\theta)+\xi,
    \label{eq:additive_noise_model}
\end{align}
where $\xi\in\mathbb{R}^d$ is independent of $Z$ and has a known density $f_0$. At $\beta=0$, the observed law is the parameter-independent null law $P_0$ with density $f_0$.
\end{model}

Throughout, we use the signal moments $M_k$ and observation-side likelihood coordinates $Q_k$ defined in~\eqref{eq:linear_weak_input_signal_moment} and~\eqref{eq:linear_weak_input_Qk}, respectively.

\paragraph{Exact debiasing by Appell polynomials.}
To construct observable statistics that recover the signal moments directly, we seek functions whose expectation over the noise, evaluated at $x+\xi$, equals the corresponding tensor monomial $x^{\otimes k}$, where $x\triangleq\beta s(Z;\theta)$. Appell polynomials associated with the noise distribution provide precisely such a construction: their generating function normalizes the exponential generating function of the observation by that of the noise, thereby canceling the noise contribution at the population level \cite{ta2015probabilistic}.

Assume that the moment-generating function $\mathcal{M}_\xi(t)\triangleq \mathbb{E}_{\xi \sim f_0}\left[e^{\langle t,\xi\rangle}\right]$ is finite in a neighborhood of the origin.
Define the tensor-valued Appell polynomials $A_k$ through
\begin{align}
    \frac{e^{\langle t,y\rangle}}{\mathcal{M}_\xi(t)} = \sum_{k=0}^\infty \frac{1}{k!} \left\langle A_k(y),t^{\otimes k}\right\rangle . \label{eq:additive_appell_generating_function}
\end{align}

\begin{lem}[Exact debiasing for additive noise]
\label{lem:additive_appell_debiasing}
Suppose that $\mathcal{M}_\xi(t)$ is finite in a neighborhood of the origin. Then, for every $k\geq 1$ and every deterministic $x\in\mathbb{R}^d$,
\begin{align}
    \mathbb{E}_{\xi\sim f_0}
    \left[A_k(x+\xi)\right]
    =
    x^{\otimes k}.
    \label{eq:additive_appell_translation_identity}
\end{align}
Consequently, under Model~\ref{model:additive_noise_channel},
\begin{align}
    \mathbb{E}_{Y\sim P_{\theta,\beta}} \left[A_k(Y)\right] = \beta^kM_k(\theta).
    \label{eq:additive_exact_debiasing}
\end{align}
\end{lem}

\begin{proof}
Taking expectation over $\xi$ in
\eqref{eq:additive_appell_generating_function} with $y=x+\xi$ gives
\begin{align}
    \mathbb{E}_{\xi \sim f_0} \left[ \frac{e^{\langle t,x+\xi\rangle}} {\mathcal{M}_\xi(t)} \right] = e^{\langle t,x\rangle}.
\end{align}
Expanding both sides around $t=0$ and identifying the order-$k$ homogeneous tensor coefficients gives \eqref{eq:additive_appell_translation_identity}. Substituting $x=\beta s(Z;\theta)$ and averaging over $Z$ yields \eqref{eq:additive_exact_debiasing}.
\end{proof}

The exact debiasing identity also verifies the cross-order Gram condition. Indeed, applying the reverse implication of Lemma~\ref{lem:dual_coordinates_signal_moments} to the translation family $x+\xi$ shows that the Appell polynomials are biorthogonal to the likelihood coordinates:
\begin{align}
    \mathbb{E}_{Y\sim P_0} \left[ A_k(Y)\left\langle Q_j(Y),M\right\rangle \right] =
    \begin{cases}
        M, & j=k,\\
        0, & j\neq k,
    \end{cases}    \label{eq:additive_appell_likelihood_biorthogonality}
\end{align}
for every $j,k\geq1$ and every symmetric order-$j$ tensor $M$, since the pure tensors $x^{\otimes j}$ span the symmetric tensor space. Thus, the Appell polynomials provide the dual observable coordinates $\psi_k=A_k$ of Lemma~\ref{lem:dual_coordinates_signal_moments}. Consequently, the cross-order nondegeneracy condition in Assumption~\ref{ass:nondegenerate_cross_order_likelihood_geometry} holds, as does the corresponding condition in Proposition~\ref{prop:linear_perturbation_verification}.

\subsection{Exponential-family latent-variable models}
\label{subsec:mean_scaling_exponential_family}

Next, we consider weak-signal exponential-family models in which the latent signal perturbs the mean parameter.
These models fall directly within the algebraic weak-signal perturbation framework of Section~\ref{subsec:linear_weak_signal_perturbations}; the additional exponential-family structure provides natural null-adapted polynomial coordinates that isolate the signal-only moments and verify the required observation-side nondegeneracy.

\begin{model}[Mean-scaling exponential-family latent-variable model]
\label{model:mean_scaling_exponential_family}
Let $Z\sim\mu$ be a latent variable, let $\theta\in\Theta\subseteq\mathbb{R}^m$ be the unknown parameter, and let $s(z;\theta)\in\mathbb{R}^d$ denote the latent signal contribution.
Conditional on $Z=z$, the observation has density
\begin{align}
    p_{\theta,\beta}(y\vert z)
    =
    h(y)\exp\left\{
        \left\langle \eta\bigl(m_{\theta,\beta}(z)\bigr),T(y)\right\rangle
        -A\bigl(\eta(m_{\theta,\beta}(z))\bigr)
    \right\},
    \label{eq:mean_scaling_exp_family_density}
\end{align}
where $T(y)\in\mathbb{R}^d$ is the sufficient statistic, $A$ is the log-partition function, and $m_{\theta,\beta}(z)=\mathbb{E}[T(Y)\vert Z=z]$ is the conditional mean parameter, with corresponding natural parameter $\eta(m_{\theta,\beta}(z))$.
We assume that
\begin{align}
    m_{\theta,\beta}(z)=m_0+\beta s(z;\theta),
    \label{eq:mean_scaling_conditional_mean}
\end{align}
where $m_0$ is a parameter-independent null mean.
At $\beta=0$, the conditional law is independent of both $\theta$ and $z$, with null density
\begin{align}
    p_0(y) = h(y)\exp\left\{ \left\langle \eta(m_0),T(y)\right\rangle -A\bigl(\eta(m_0)\bigr) \right\}.
    \label{eq:mean_scaling_null_density}
\end{align}
\end{model}

This framework includes many standard exponential-family observation models. Examples include Gaussian observations with perturbed mean, Poisson observations with perturbed intensity, binomial observations with perturbed success probability, and Gamma and exponential observations with perturbed mean. It also covers families such as the log-normal and Weibull distributions after choosing the appropriate sufficient statistic $T(Y)$. Representative examples, together with their natural and mean parameterizations and the corresponding null-adapted polynomial coordinates, are summarized in Table~\ref{tab:observation_families}. 

We next impose a natural condition ensuring that the weak-signal perturbation stays inside the regular part of the exponential family.

\begin{table*}[t]
\centering
\begingroup
\setlength{\tabcolsep}{2.2pt}
\renewcommand{\arraystretch}{1.12}
\scriptsize

\caption{\textbf{Exponential-family observation models.}
Each row specifies the sufficient statistic $T(y)$, the mean parameter $m=\mathbb{E}[T(Y)]$, the natural parameter $\eta(m)$, the log-partition function $A(\eta)$, and the corresponding likelihood coordinate $Q_k(T(y))$ defined in~\eqref{eq:linear_weak_input_Qk}. The likelihood coordinates belong to the classical orthogonal polynomial family associated with the null law. Here, $H_k$ denotes the probabilists' Hermite polynomial, $C_k$ a Charlier polynomial~\cite{zacharovas2010charlier}, $L_k^{(\alpha-1)}$ a generalized Laguerre polynomial~\cite{szego1939orthogonal}, and $K_k$ a Krawtchouk polynomial~\cite{nikiforov1991classical}. For vector-valued observations with i.i.d.  coordinates, the corresponding multivariate likelihood coordinates are obtained by tensorizing these scalar polynomial systems.}
\label{tab:observation_families}

\begin{tabular}{
|L{2.55cm}||
L{1.70cm}|
L{2.25cm}|
L{2.25cm}|
L{2.35cm}|
L{4.30cm}|
}
\hline
\textbf{Observation family}
&
\textbf{Sufficient statistic $T(y)$}
&
\textbf{Mean parameter $m=\mathbb{E}[T(Y)]$}
&
\textbf{Natural parameter $\eta(m)$}
&
\textbf{Log-partition $A(\eta)$}
&
\textbf{Likelihood coordinate $Q_k(T(y))$}

\\
\hline\hline

\textbf{Gaussian distribution}
&
$y$
&
$\mathbb{E}[Y]=m$
&
$\displaystyle \frac{m}{\sigma^2}$
&
$\displaystyle \frac{\sigma^2\eta^2}{2}$
&
$\displaystyle
\frac{1}{k!\sigma^k}
H_k\!\left(\frac{y-m_0}{\sigma}\right)$
\\
\hline

\textbf{Poisson distribution}
&
$y$
&
$\mathbb{E}[Y]=m$
&
$\log m$
&
$e^\eta$
&
$\displaystyle
\frac{1}{k!m_0^k}
C_k(y;m_0)$
\\
\hline

\textbf{Gamma distribution}
&
$y$
&
$\mathbb{E}[Y]=m$
&
$\displaystyle -\frac{\alpha}{m}$
&
$\displaystyle -\alpha\log(-\eta)$
&
$\displaystyle
\frac{(-1)^k}{m_0^k}
L_k^{(\alpha-1)}
\!\left(\frac{\alpha y}{m_0}\right)$
\\
\hline

\textbf{Exponential distribution}
&
$y$
&
$\mathbb{E}[Y]=m$
&
$\displaystyle -\frac{1}{m}$
&
$\displaystyle -\log(-\eta)$
&
$\displaystyle
\frac{(-1)^k}{m_0^k}
L_k\!\left(\frac{y}{m_0}\right)$
\\
\hline

\textbf{Binomial distribution}
&
$y$
&
$\mathbb{E}[Y]=m=Np$
&
$\displaystyle
\log\!\left(\frac{m}{N-m}\right)$
&
$\displaystyle
N\log(1+e^\eta)$
&
$\displaystyle
\frac{(-1)^k}{N^k(1-p_0)^k}
\binom{N}{k}
K_k(y;p_0,N)$
\\
\hline

\textbf{Log-normal distribution}
&
$\log y$
&
$\mathbb{E}[\log Y]=m$
&
$\displaystyle \frac{m}{\sigma^2}$
&
$\displaystyle \frac{\sigma^2\eta^2}{2}$
&
$\displaystyle
\frac{1}{k!\sigma^k}
H_k\!\left(\frac{\log y-m_0}{\sigma}\right)$
\\
\hline

\textbf{Weibull distribution}
&
$y^\kappa$
&
$\mathbb{E}[Y^\kappa]=m=\lambda^\kappa$
&
$\displaystyle -\frac{1}{m}$
&
$\displaystyle -\log(-\eta)$
&
$\displaystyle
\frac{(-1)^k}{m_0^k}
L_k\!\left(\frac{y^\kappa}{m_0}\right)$
\\
\hline

\end{tabular}

\endgroup
\end{table*}

\begin{assum}[Regularity around the null]
\label{assum:regular_exponential_family}
Let $\Omega \triangleq \{\eta:A(\eta)<\infty\}$ denote the natural-parameter space of the exponential family, and let
\begin{align}
    \eta_0 \triangleq \eta(m_0),
    \qquad
    \eta_{\theta,\beta}(z) \triangleq \eta\bigl(m_0+\beta s(z;\theta)\bigr).
\end{align}
We say that the mean-scaling exponential-family model is \emph{regular around the null} if the following conditions hold.

\begin{enumerate}
\item \emph{(Interior null parameter).}
The null natural parameter lies in the interior of the natural-parameter space: $\eta_0\in\operatorname{int}(\Omega)$.

\item \emph{(Uniform interior perturbations).}
There exist $\beta_0>0$ and a compact set $\mathcal{K}\subset\operatorname{int}(\Omega)$ such that $\eta_{\theta,\beta}(z)\in \mathcal{K}$ for all $\theta\in\Theta$, all $z\in\mathcal{Z}$, and all $0\le\beta\le\beta_0$.

\end{enumerate}
\end{assum}

The next lemma, which is proved in Appendix~\ref{app:proof_mean_scaling_regular_likelihood_ratio_expansion}, verifies that this regularity condition implies the  likelihood-ratio regularity assumed in the general theory (Assumption~\ref{ass:regular_observed_likelihood_ratio_expansion}).
Indeed, compactness inside $\operatorname{int}(\Omega)$ gives uniform control of exponential-family likelihood ratios, while the mean-scaling assumption gives analyticity in $\beta$ around the null.

\begin{lem}[Regular observed likelihood-ratio expansion for exponential families]
\label{lem:mean_scaling_regular_likelihood_ratio_expansion}
Consider Model~\ref{model:mean_scaling_exponential_family}. If Assumption~\ref{assum:regular_exponential_family} holds, then Assumptions~\ref{ass:regular_observed_likelihood_ratio_expansion} and~\ref{ass:reciprocal_likelihood_regularity} hold.
\end{lem}

\paragraph{Dual polynomial coordinates.}
Throughout, we use the signal moments $M_k$ and observation-side likelihood coordinates $Q_k$ defined in~\eqref{eq:linear_weak_input_signal_moment} and~\eqref{eq:linear_weak_input_Qk}, respectively. For the standard families listed in Table~\ref{tab:observation_families}, the likelihood coordinates are generated by the classical orthogonal polynomial system associated with the null law. These families are Hermite for Gaussian observations, Charlier for Poisson observations~\cite{zacharovas2010charlier}, and Laguerre-type for Gamma and exponential observations~\cite{szego1939orthogonal}. 

These polynomials form a complete orthogonal system in $L^2(P_0)$, where their tensor products yield the orthogonal decomposition $L^2(P_0) = \bigoplus_{k\geq0}\mathcal{H}_k$, where $\mathcal{H}_k$ is the subspace spanned by the order-$k$ likelihood coordinates.
In particular,
\begin{align}
    \mathbb{E}_{Y\sim P_0} \left[ \left\langle Q_j(Y),M\right\rangle \left\langle Q_k(Y),N\right\rangle \right] =0,
    \qquad j\neq k.    \label{eq:exp_family_cross_order_orthogonality}
\end{align}

To normalize the order-$k$ coordinate, we define its within-order Gram operator by
$\mathbf{G}_k M \triangleq \mathbb{E}_{Y\sim P_0} \left[ Q_k(Y)\left\langle Q_k(Y),M\right\rangle \right]$.
The corresponding dual polynomial coordinate is the Gram-normalized likelihood coordinate $P_k(T(Y)) \triangleq \mathbf{G}_k^{-1}Q_k(Y)$, which satisfies the biorthogonality relations
\begin{align}
    \mathbb{E}_{Y\sim P_0} \left[ P_k(T(Y) \left\langle Q_j(Y),M\right\rangle \right] =
    \begin{cases}
        M, & j=k,\\
        0, & j\neq k,
    \end{cases}
    \label{eq:exp_family_dual_coordinates}
\end{align}
so that $P_k(T(Y))$ is precisely the dual observable coordinate $\psi_k(Y)$ of Lemma~\ref{lem:dual_coordinates_signal_moments}. Consequently, $\mathbb{E}_{Y\sim P_{\theta,\beta}} \left[P_k(T(Y))\right] = \beta^k M_k(\theta)$, and Assumption~\ref{ass:nondegenerate_cross_order_likelihood_geometry} follows.

\section{Discussion and outlook}
\label{sec:discussion}

\paragraph{A general recipe for weak-signal analysis.}
While moment-order arguments have appeared previously in several important Gaussian model-specific settings, the contribution here is to formulate the underlying mechanism at the level of the likelihood expansion and to separate the ingredients that must be verified in a new model. A useful analysis can be organized as follows. First, identify the common null law $P_0$ at $\beta=0$ and expand $dP_{\theta,\beta}/dP_0$ in powers of $\beta$. Second, determine the first orders at which the likelihood moments identify the parameter globally and locally, yielding $r_{\mathrm{glob}}$ and $r_{\mathrm{loc}}$. In structured models, this step often reduces to an algebraic or invariant-theoretic problem. Third, verify that the relevant finite-dimensional likelihood coordinates have nondegenerate joint Gram geometry, so that information from different orders cannot cancel at finite signal strength. Finally, construct observable features with sufficiently controlled fluctuations. When these ingredients hold, the likelihood expansion determines both the fixed-accuracy exponent and the global-to-local sample-complexity transition.

\paragraph{Future research directions.}
Several directions remain open. First, the present theory treats the signal dimension as fixed. An important extension is the high-dimensional regime $d\to\infty$, which is of broad interest in many applications and may introduce additional statistical scales through the dimension and conditioning of the likelihood-moment spaces~\cite{romanov2021multi}. A related direction is to characterize weak-signal performance under normalized mean-squared error (NMSE), rather than fixed-accuracy minimax recovery, as studied for example in~\cite{romanov2021multi}. Such an analysis is closely connected to the shrinking-accuracy regime $\varepsilon\to 0$ developed here, since NMSE requires resolving the parameter increasingly accurately rather than only to a fixed error tolerance.

Second, the framework is primarily statistical and does not resolve the computational difficulty of the resulting inverse problems. Even when the likelihood moments required for identifiability are known accurately, their inversion may be nonconvex or poorly conditioned. Developing computationally efficient procedures and determining when computational constraints create a gap from the information-theoretic sample complexity remain important directions for future work.

Third, it would be valuable to study statistical efficiency within the weak-signal models considered here. The present results characterize the sample size required for recovery, but do not compare the asymptotic covariance of different rate-optimal estimators. In particular, one may ask whether maximum likelihood estimation and optimally weighted generalized method of moments are asymptotically equivalent in the weak-signal regime, and more generally when the likelihood moments retain all of the leading Fisher information~\cite{balanov2026generalized}. Such results would refine the present sample-complexity theory by distinguishing estimators that attain the same weak-signal exponent but differ in their asymptotic efficiency.

\bibliographystyle{plain}

\begin{appendices}

{\centering{\section*{Appendix}}}

\paragraph{Appendix organization.}
Appendix~\ref{app:information_theoretic_preliminaries} collects the information-theoretic preliminaries used throughout the paper. Appendix~\ref{sec:two_point_testing} reviews the two-point testing arguments underlying the lower bounds and necessary separation conditions. Appendices~\ref{sec:information_theoretic_sample_complexity}--\ref{sec:rate_optimal_recovery_MoM} contain the proofs of the fixed-accuracy results: the information-theoretic bounds, the maximum-likelihood upper bound, and the method-of-moments upper bound, respectively. Appendix~\ref{sec:smooth_quotient_models} develops the local quotient geometry and analytic remainder estimates needed for the shrinking-accuracy analysis, while Appendix~\ref{sec:sample_complexity_algebric_likelihood_geometry} proves the global-to-local results of Section~\ref{sec:global_to_local_algebraic}. Finally, Appendix~\ref{sec:applications_algebraic_perturbations} provides the proofs and representative model specializations associated with the applications in Section~\ref{sec:applications}.

\section{Information-theoretic preliminaries}
\label{app:information_theoretic_preliminaries}

Throughout, let $P$ and $Q$ be probability measures on a common sample space $\mathcal{Y}$, dominated by a measure $\nu$, with densities $p=dP/d\nu$ and $q=dQ/d\nu$, respectively. We first introduce the information-theoretic quantities used throughout the paper, and then collect standard relations among them that will be repeatedly invoked in the proofs. The results below are standard; see, for example, \cite{tsybakov2009introduction, le2012asymptotics}.

\paragraph{Information-theoretic quantities and basic comparisons.}
We use the standard total variation and Hellinger distances, Kullback--Leibler divergence, and chi-square divergence:
\begin{align}
    \operatorname{TV}(P,Q) &\triangleq \frac{1}{2}\int_{\mathcal{Y}}|p(y)-q(y)| \, d\nu(y),
    \label{eq:info_tv_definition}
    \\
    h^2(P,Q) &\triangleq \int_{\mathcal{Y}} \left(\sqrt{p(y)}-\sqrt{q(y)}\right)^2 \, d\nu(y) \, \label{eq:info_hellinger_definition}
    \\
    D_{\mathrm{KL}}(P \| Q) &\triangleq
    \int_{\mathcal{Y}} p(y)\log \! \left(\frac{p(y)}{q(y)}\right) \,d\nu(y),
    \label{eq:info_kl_definition}
    \\
    \chi^2(P,Q) &\triangleq \int_{\mathcal{Y}} \frac{(p(y)-q(y))^2}{q(y)} \, d\nu(y).
    \label{eq:info_chisq_definition}
\end{align}
The following standard comparisons will be used repeatedly:
\begin{align}
    \frac{1}{2} h^2(P,Q) &\leq \operatorname{TV}(P,Q) \leq h(P,Q),
    \label{eq:info_tv_hellinger}\\
    \operatorname{TV}^2(P,Q) &\leq \frac{1}{2} D_{\mathrm{KL}}(P\|Q),
    \label{eq:info_pinsker}\\
    h^2(P,Q) &\leq D_{\mathrm{KL}}(P\|Q),
    \label{eq:info_hellinger_kl}\\
    D_{\mathrm{KL}}(P\|Q) &\leq \log\!\left(1+\chi^2(P,Q)\right) \leq \chi^2(P,Q).
    \label{eq:info_kl_chisq}
\end{align}

\paragraph{Likelihood-ratio representations.}
The following identities express the information-theoretic quantities above directly in terms of likelihood ratios with respect to a common reference measure, which is the form used throughout the weak-signal likelihood analysis.

\begin{lem}[Likelihood-ratio representations]
\label{lem:info_likelihood_ratio_representations}
Define $L_0 \triangleq dP/dP_0$, $L_1 \triangleq dQ/dP_0$. Then:
\begin{align}
    \operatorname{TV}(P,Q)
    &=
    \frac{1}{2} \|L_0-L_1\|_{L^1(P_0)},
    \label{eq:info_tv_likelihood_ratio}\\
    h^2(P,Q)
    &=
    \mathbb{E}_{Y \sim P_0} \left[ \left(\sqrt{L_0(Y)}-\sqrt{L_1(Y)}\right)^2 \right], \label{eq:info_hellinger_likelihood_ratio}\\
    D_{\mathrm{KL}}(P\|Q)
    &=
    \mathbb{E}_{Y \sim P_0} \left[ L_0(Y) \log\!\left(\frac{L_0(Y)}{L_1(Y)}\right) \right],
    \label{eq:info_kl_likelihood_ratio}\\
    \chi^2(P,Q)
    &=
    \mathbb{E}_{Y \sim P_0} \left[ \frac{(L_0(Y)-L_1(Y))^2}{L_1(Y)} \right].
    \label{eq:info_chisq_likelihood_ratio}
\end{align}
In particular,
\begin{align}
    \chi^2(P,P_0) = \|L_0-1\|_{L^2(P_0)}^2.
    \label{eq:info_chisq_null_identity}
\end{align}
\end{lem}

\begin{proof}
The first two identities follow immediately from~\eqref{eq:info_tv_definition}--\eqref{eq:info_hellinger_definition} after changing the measure to $P_0$. The KL and chi-square identities follow similarly from $dP=L_0\,dP_0$ and $dQ=L_1\,dP_0$. The final identity is the special case $Q=P_0$.
\end{proof}

\section{Two-point testing}
\label{sec:two_point_testing}

\subsection{Binary hypothesis-testing problem}
A standard approach to deriving minimax lower bounds is to exhibit two well-separated parameter values, or more generally two separated subsets of the parameter space, that nevertheless induce nearly indistinguishable distributions of the observations. A recurring device is therefore to reduce the estimation problem to a binary hypothesis-testing problem. Given two probability measures $P$ and $Q$, a test is a measurable function $\psi:\mathcal{Y}\to\{0,1\}$, where $\psi=0$ corresponds to choosing $P$ and $\psi=1$ corresponds to choosing $Q$. Its two error probabilities are therefore
\[
    P(\psi=1)
    \qquad\text{and}\qquad
    Q(\psi=0).
\]
The first is the probability of choosing $Q$ when the data are generated from $P$, while the second is the probability of choosing $P$ when the data are generated from $Q$.

\begin{lem}[Binary testing inequalities]
\label{lem:info_binary_testing}
Let $P$ and $Q$ be probability measures on $\mathcal{Y}$, and let $\psi:\mathcal{Y}\to\{0,1\}$ be any test, where $\psi=0$ corresponds to choosing $P$ and $\psi=1$ to choosing $Q$. Then
\begin{align}
    P(\psi=1)+Q(\psi=0)
    &\geq 1-\mathrm{TV}(P,Q),  \label{eq:info_testing_tv}
    \\
    P(\psi=1)+Q(\psi=0) &\geq \frac{1}{2} \exp\left\{- D_{\mathrm{KL}}(P\|Q)\right\}.  \label{eq:info_bretagnolle_huber}
\end{align}
Consequently, if both testing errors are at most $\delta$, then
\begin{align}
    \mathrm{TV}(P,Q)  & \geq 1-2\delta,    \label{eq:info_testing_tv_lower}
    \\
     D_{\mathrm{KL}}(P\|Q) &\geq \log\left(\frac{1}{4\delta}\right).    \label{eq:info_testing_kl_lower}
\end{align}
Moreover, for $0<\delta<1/4$, there exists a universal constant $c>0$ such that
\begin{align}
     D_{\mathrm{KL}}(P\|Q) \geq c\log(1/\delta).    \label{eq:info_testing_log_lower}
\end{align}
\end{lem}

\begin{proof}
The first inequality is the standard total-variation characterization of binary testing, while the second is the Bretagnolle--Huber inequality~\cite{tsybakov2009introduction}. If both testing errors are at most $\delta$, their sum is at most $2\delta$, which gives \eqref{eq:info_testing_tv_lower} and \eqref{eq:info_testing_kl_lower}. Finally, Pinsker's inequality together with \eqref{eq:info_testing_tv_lower} yields $ D_{\mathrm{KL}}(P\|Q)\geq 2(1-2\delta)^2$. Combining this constant lower bound with \eqref{eq:info_testing_kl_lower} gives \eqref{eq:info_testing_log_lower} uniformly over $0<\delta<1/4$.
\end{proof}

\paragraph{Tensorization.} 
The following tensorization identities describe how the relevant divergences scale under independent product measures and will be used to pass between single-sample and $n$-sample bounds.

\begin{lem}[Tensorization for i.i.d. product measures]
\label{lem:info_tensorization}
Assume $D_{\mathrm{KL}}(P\|Q) < \infty$. Then, for every $n\geq 1$,
\begin{align}
    D_{\mathrm{KL}} \left( P^{\otimes n}\|Q^{\otimes n} \right) = n D_{\mathrm{KL}}(P\|Q).
    \label{eq:info_kl_tensorization}
\end{align}
Moreover,
\begin{align}
    h^2 \left( P^{\otimes n},Q^{\otimes n} \right) &= 2\left[ 1- \left( 1-\frac{h^2(P,Q)}{2} \right)^n \right] \leq nh^2(P,Q).
    \label{eq:info_hellinger_tensorization}
\end{align}
\end{lem}

\begin{proof}
The KL identity follows from additivity of the log-likelihood ratio under
independent products.
For Hellinger distance, the affinity
$\int\sqrt{pq}\,d\nu=1-h^2(P,Q)/2$ tensorizes, giving the equality in
\eqref{eq:info_hellinger_tensorization}; the inequality follows from
$1-(1-x)^n\leq nx$ for $x\in[0,1]$.
\end{proof}

\subsection{Le Cam's two-point method}

We now employ Lemma~\ref{lem:info_binary_testing} to derive a minimax lower bound via Le Cam's two-point method. The basic idea is to identify two well-separated parameter values that nevertheless induce nearly indistinguishable distributions of the observations. In our setting, suppose that $\theta_0,\theta_1\in\Theta$ satisfy $d_{\mathrm{eq}}(\theta_0,\theta_1)>2\varepsilon$, while the product laws $P_{\theta_0,\beta}^{\otimes n}$ and $P_{\theta_1,\beta}^{\otimes n}$ remain statistically close. Any estimator achieving quotient error at most $\varepsilon$ under both parameters would then induce a binary test that reliably distinguishes between these two product laws. Lemma~\ref{lem:info_binary_testing} shows that such a test cannot exist when the corresponding divergence is small. Consequently, no estimator can achieve uniformly small quotient error over the full parameter space. By the i.i.d.\ product structure, the relevant divergence between the $n$-sample laws reduces to the corresponding single-sample divergence between $P_{\theta_0,\beta}$ and $P_{\theta_1,\beta}$. In the weak-signal regime considered here, this reduction is particularly useful because the single-sample divergence admits a sharp expansion in $\beta$, whose leading order is determined by the first likelihood moment at which the two parameters differ.

\begin{proposition}[Two-point estimation lower bound]
\label{prop:lecam_two_point_general}
Let $d: \Theta \times \Theta \to \mathbb{R}^{+} $ be a pseudo-metric on $\Theta$, let $\varepsilon>0$, and suppose that $\theta_0,\theta_1\in\Theta$ satisfy $d(\theta_0,\theta_1)>2\varepsilon$.
Then every estimator $\widehat{\theta}$ based on $n$ i.i.d. observations satisfies
\begin{align}
    \max_{j\in\{0,1\}}
    P_{\theta_j}^{\otimes n}
    \left(
        d(\widehat{\theta},\theta_j)>\varepsilon
    \right)
    \geq
    \frac{1}{4}
    \exp\left\{
        -n D_{\mathrm{KL}}\left(P_{\theta_0}\|P_{\theta_1}\right)
    \right\}.
    \label{eq:info_two_point_lecam}
\end{align}
Consequently, if $\max_{j\in\{0,1\}} P_{\theta_j}^{\otimes n} \left(       d(\widehat{\theta},\theta_j)>\varepsilon \right) \leq \delta$, then
\begin{align}
    nD_{\mathrm{KL}}\left(P_{\theta_0}\|P_{\theta_1}\right) \geq \log\left(\frac{1}{4\delta}\right).    \label{eq:info_two_point_estimation}
\end{align}
In particular, for $0<\delta<1/4$, $nD_{\mathrm{KL}}\left(P_{\theta_0}\|P_{\theta_1}\right) \geq c\log(1/\delta)$ for a universal constant $c>0$.
\end{proposition}

\subsection{Hellinger separation implied by successful recovery}
\label{subsec:hellinger_separation_recovery}

The two-point argument can also be used in the reverse direction. Whereas Le Cam's method starts from two statistically close parameter values and concludes that accurate estimation is impossible, here we start from the assumption that accurate estimation is possible and infer that any two parameter values separated by more than twice the target accuracy must induce statistically separated distributions.

The following elementary consequence of the standard two-point testing argument will be useful in deriving necessary conditions for achievability; see, e.g., \cite[Chapter~2]{tsybakov2009introduction}.

\begin{proposition}[Successful recovery implies Hellinger separation]
\label{prop:recovery_implies_hellinger_separation_general}
Let $d: \Theta \times \Theta \to \mathbb{R}^{+} $ be a pseudo-metric on $\Theta$, fix $\varepsilon>0$, and let
$0<\delta_0<1/2$.
Suppose that an estimator $\widehat\theta$, based on $n$ i.i.d.\ observations, satisfies $ \max_{j\in\{0,1\}} P_{\theta_j}^{\otimes n} (d(\widehat\theta,\theta_j)>\varepsilon ) \leq \delta_0$ for two parameters $\theta_0,\theta_1\in\Theta$ satisfying $d(\theta_0,\theta_1)>2\varepsilon$.
Then
\begin{align}
    h(P_{\theta_0},P_{\theta_1}) \geq \frac{1-2\delta_0}{\sqrt n}.    \label{eq:recovery_hellinger_general}
\end{align}
\end{proposition}

\begin{proof}
The estimator induces a binary test between $\theta_0$ and $\theta_1$.
Define $\psi \triangleq \mathbf{1} \{ d(\widehat\theta,\theta_0)>\varepsilon\}$, where $\psi=0$ corresponds to choosing $\theta_0$ and $\psi=1$ to
choosing $\theta_1$.
Under $\theta_0$,
\begin{align}
    P_{\theta_0}^{\otimes n}(\psi=1) = P_{\theta_0}^{\otimes n} \left(d(\widehat\theta,\theta_0)>\varepsilon\right) \leq \delta_0.
\end{align}
Under $\theta_1$, the event $\{\psi=0\}$ implies $d(\widehat\theta,\theta_0)\leq\varepsilon$. Since $d(\theta_0,\theta_1)>2\varepsilon$, the triangle inequality gives $d(\widehat\theta,\theta_1) \geq d(\theta_0,\theta_1) - d(\widehat\theta,\theta_0) > \varepsilon$. Hence
\begin{align}
    P_{\theta_1}^{\otimes n}(\psi=0) \leq P_{\theta_1}^{\otimes n} \left( d(\widehat\theta,\theta_1)>\varepsilon \right) \leq \delta_0.
\end{align}
Thus the sum of the two testing errors is at most $2\delta_0$.

On the other hand, the binary testing bound~\eqref{eq:info_testing_tv} gives
\begin{align}
    P_{\theta_0}^{\otimes n}(\psi=1) + P_{\theta_1}^{\otimes n}(\psi=0) \geq 1- \operatorname{TV} \left( P_{\theta_0}^{\otimes n}, P_{\theta_1}^{\otimes n} \right).
\end{align}
Therefore, $\operatorname{TV} ( P_{\theta_0}^{\otimes n}, P_{\theta_1}^{\otimes n}) \geq 1-2\delta_0$. Finally, using $\operatorname{TV}(P,Q)\leq h(P,Q)$ and
Lemma~\ref{lem:info_tensorization},
\begin{align}
    1-2\delta_0 \leq h\left( P_{\theta_0}^{\otimes n}, P_{\theta_1}^{\otimes n} \right) \leq \sqrt n\,    h(P_{\theta_0},P_{\theta_1}).
\end{align}
Rearranging yields~\eqref{eq:recovery_hellinger_general}.
\end{proof}

\section{Information-theoretic sample complexity}
\label{sec:information_theoretic_sample_complexity}

\subsection{Proof of Proposition~\ref{prop:divergence_first_likelihood_moment}}
\label{sec:proof_prop_divergence_first_likelihood_moment}

For $i = 0,1$, we define 
\begin{align}
    L_i(Y) \triangleq L_{\theta_i,\beta}(Y) = \frac{dP_{\theta_i,\beta}}{dP_0}(Y).
\end{align}
By Assumption~\ref{ass:regular_observed_likelihood_ratio_expansion}, for any fixed expansion order and any finite $q\geq 1$, the likelihood ratios admit $L^q(P_0)$-controlled Taylor expansions. Since $r(\theta_0,\theta_1)=r$, by the definition in~\eqref{eq:first_mismatching_likelihood_moment_order}, the likelihood moments agree up to order $r-1$, while $\Phi_r(\theta_0)-\Phi_r(\theta_1)$ is nonzero in $L^2(P_0)$. Therefore, for every finite $q\geq 1$,
\begin{align}
    L_0-L_1 = \beta^r \big[ \Phi_r(\theta_0)-\Phi_r(\theta_1) \big] + O_{L^q(P_0)}(\beta^{r+1}). \label{eq:likelihood_ratio_difference_first_mismatch}
\end{align}
Let $p_0>2$ denote the fixed exponent from
Assumption~\ref{ass:reciprocal_likelihood_regularity}.
Then
\begin{align}
    \norm{L_1^{-1}-1}_{L^{p_0}(P_0)} = O(\beta),
    \qquad
    \norm{L_1^{-1}}_{L^{p_0}(P_0)} = O(1). \label{eq:proof_reciprocal_moment_bound}
\end{align}

We first prove the $\chi^2$ expansion.
Using the likelihood-ratio representation
in~\eqref{eq:info_chisq_likelihood_ratio},
\begin{align}
    \chi^2(P_{\theta_0,\beta},P_{\theta_1,\beta})  = \mathbb{E}_{Y \sim P_0} \left[ \frac{(L_0-L_1)^2}{L_1} \right], \label{eq:chisq_likelihood_ratio_identity}
\end{align}
and write
\begin{align}
    \frac{(L_0-L_1)^2}{L_1}  = (L_0-L_1)^2 + (L_0-L_1)^2(L_1^{-1}-1).
\end{align}
H\"older's inequality with exponents $p_0$ and $p_0/(p_0-1)$ gives
\begin{align}
    \mathbb{E}_{Y \sim P_0} \left[ (L_0-L_1)^2 |L_1^{-1}-1| \right]
    &\leq
    \norm{L_0-L_1}_{L^{2p_0/(p_0-1)}(P_0)}^2 \norm{L_1^{-1}-1}_{L^{p_0}(P_0)}
    \notag \\
    &=
    O(\beta^{2r+1}), \label{eq:chisq_reciprocal_remainder_bound}
\end{align}
where the required likelihood-difference moment is finite by Assumption~\ref{ass:regular_observed_likelihood_ratio_expansion}. Substituting~\eqref{eq:likelihood_ratio_difference_first_mismatch}--\eqref{eq:proof_reciprocal_moment_bound} into~\eqref{eq:chisq_reciprocal_remainder_bound} yields
\begin{align}
    \chi^2(P_{\theta_0,\beta},P_{\theta_1,\beta}) = \beta^{2r} \left\| \Phi_r(\theta_0)-\Phi_r(\theta_1) \right\|_{L^2(P_0)}^2 + O(\beta^{2r+1}).
\end{align}

We next prove the KL expansion. Let $\Delta_\beta \triangleq L_0-L_1$. By the likelihood-ratio representation in~\eqref{eq:info_kl_likelihood_ratio},
\begin{align}
    D_{\mathrm{KL}} (P_{\theta_0,\beta} \| P_{\theta_1,\beta}) = \mathbb{E}_{Y \sim P_0} \left[ (L_1+\Delta_\beta) \log \left( 1+\frac{\Delta_\beta}{L_1} \right) \right].
\end{align}
For $x>-1$, we use the Taylor expansion $(1+x)\log(1+x) = x+\frac{x^2}{2}+\mathcal{R}(x)$, with $|\mathcal{R}(x)| \leq C|x|^3$. Since $L_0=L_1+\Delta_\beta\geq 0$, $x=\Delta_\beta/L_1\geq-1$, and hence
\begin{align}
    D_{\mathrm{KL}} (P_{\theta_0,\beta} \| P_{\theta_1,\beta})
    &=
    \mathbb{E}_{Y \sim P_0}[\Delta_\beta] + \frac{1}{2} \mathbb{E}_{Y \sim P_0} \left[\frac{\Delta_\beta^2}{L_1} \right] + R_\beta,
    \label{eqn:KL_div_expansion2}
\end{align}
where $|R_\beta| \leq C \mathbb{E}_{Y \sim P_0} \left[ \frac{|\Delta_\beta|^3}{L_1^2} \right]$. Since $p_0>2$, H\"older's inequality with exponents $p_0/2$ and $p_0/(p_0-2)$ gives
\begin{align}
    \mathbb{E}_{Y \sim P_0} \left[ \frac{|\Delta_\beta|^3}{L_1^2} \right] \leq  \norm{\Delta_\beta}_{L^{3p_0/(p_0-2)}(P_0)}^3 \norm{L_1^{-1}}_{L^{p_0}(P_0)}^2 = O(\beta^{3r}),    \label{eq:kl_cubic_remainder_bound}
\end{align}
where again the required finite moment of $\Delta_\beta$ follows from Assumption~\ref{ass:regular_observed_likelihood_ratio_expansion}.

By~\eqref{eq:observed_likelihood_ratio_centering}, since $L_0$ and $L_1$ are likelihood ratios with respect to $P_0$, $\mathbb{E}_{Y \sim P_0}[\Delta_\beta] = \mathbb{E}_{Y \sim P_0}[L_0-L_1] = 0$. Substituting~\eqref{eq:chisq_likelihood_ratio_identity} into~\eqref{eqn:KL_div_expansion2} therefore gives
\begin{align}
    D_{\mathrm{KL}} (P_{\theta_0,\beta} \| P_{\theta_1,\beta}) = \frac{1}{2}    \chi^2(P_{\theta_0,\beta},P_{\theta_1,\beta}) + O(\beta^{3r}).
\end{align}
Since $r\geq 1$, $O(\beta^{3r})=O(\beta^{2r+1})$, and substituting the $\chi^2$ expansion from~\eqref{eq:chi_square_expansion_general} yields
\begin{align}
    D_{\mathrm{KL}} (P_{\theta_0,\beta} \| P_{\theta_1,\beta}) = \frac{\beta^{2r}}{2}  \left\| \Phi_r(\theta_0)-\Phi_r(\theta_1) \right\|_{L^2(P_0)}^2 + O(\beta^{2r+1}).
\end{align}
This completes the proof.

\subsection{Proof of Theorem~\ref{thm:lecam_lower_bound}}
\label{app:proof_lecam_lower_bound}

Let $\theta_0,\theta_1$ be the hard pair from the theorem and write $r=r(\theta_0,\theta_1)$. By Proposition~\ref{prop:divergence_first_likelihood_moment},
\begin{align}
    D_{\mathrm{KL}} \left( P_{\theta_0,\beta} \,\middle\|\, P_{\theta_1,\beta} \right) = \frac{\beta^{2r}}{2} \norm{ \Phi_r(\theta_0)-\Phi_r(\theta_1) }_{L^2(P_0)}^2 + O(\beta^{2r+1}).
\end{align}
Hence, for all sufficiently small $\beta$,
\begin{align}
    D_{\mathrm{KL}} \left( P_{\theta_0,\beta} \,\middle\|\, P_{\theta_1,\beta} \right) \leq C_{\theta_0,\theta_1}\beta^{2r}.
    \label{eq:hard_pair_kl_upper}
\end{align}

Since
$d_{\mathrm{eq}}(\theta_0,\theta_1)>2\varepsilon$, Proposition~\ref{prop:lecam_two_point_general}, applied with $d=d_{\mathrm{eq}}$, implies that any estimator satisfying
$\max_{j=0,1} P_{\theta_j,\beta}^{\otimes n} \left( d_{\mathrm{eq}}(\widehat\theta,\theta_j)>\varepsilon \right) \leq \delta$ must satisfy for $0<\delta<1/4$
\begin{align}
    n D_{\mathrm{KL}} \left(P_{\theta_0,\beta} \,\middle\|\, P_{\theta_1,\beta} \right) \geq c\log(1/\delta).
\end{align}
Combining this with~\eqref{eq:hard_pair_kl_upper} gives $n \geq c_{\theta_0,\theta_1} \beta^{-2r} \log(1/\delta)$.
Since the minimax risk over $K$ is at least the maximum risk over the pair $\{\theta_0,\theta_1\}$, the same lower bound holds for $n^\star(\varepsilon,\delta;\beta)$.

\subsection{Proof of Theorem~\ref{thm:cross_order_stability_necessity}}
\label{subsec:proof_thm_cross_order_stability_necessity}

We first isolate two auxiliary consequences used in the proof. The first converts a minimax recovery guarantee into pairwise Hellinger separation, while the second relates Hellinger distance to $L^2(P_0)$ separation of the corresponding likelihood ratios.

\begin{lem}[Minimax recovery implies Hellinger separation]
\label{lem:minimax_implies_hellinger_separation}
Fix $\varepsilon>0$ and $0<\delta_0<1/2$. Suppose that, for all sufficiently small $\beta$, $n^\star(\varepsilon,\delta_0;\beta) \leq C_\varepsilon\beta^{-2r}$. Then, for every $\theta,\theta'\in K$ satisfying $d_{\mathrm{eq}}(\theta,\theta')>2\varepsilon$,
\begin{align}
    h(P_{\theta,\beta},P_{\theta',\beta}) \geq c_0\beta^r,    \label{eq:necessity_hellinger_lower}
\end{align}
for $c_0 \triangleq \frac{1-2\delta_0}{\sqrt{C_\varepsilon}}$ and for all sufficiently small $\beta$.
\end{lem}

\begin{lem}[Hellinger comparison near the null]
\label{lem:hellinger_likelihood_l2_comparison}
Suppose Assumption~\ref{ass:regular_observed_likelihood_ratio_expansion} holds.
Then, for every $M<\infty$, there exist $C_M<\infty$ and $\beta_M>0$ such that, for every $0<\beta\leq\beta_M$ and every $\theta,\theta'\in\Theta$,
\begin{align}
    h^2(P_{\theta,\beta},P_{\theta',\beta}) \leq  2  \norm{L_{\theta,\beta}-L_{\theta',\beta} }_{L^2(P_0)}^2 +  C_M\beta^M.    \label{eq:hellinger_likelihood_l2_upper}
\end{align}
\end{lem}

Combining~\eqref{eq:necessity_hellinger_lower} and~\eqref{eq:hellinger_likelihood_l2_upper} yields for some constant $c > 0$,
\begin{align}
    \norm{L_{\theta,\beta}-L_{\theta',\beta} }_{L^2(P_0)} \geq c\beta^r
\end{align}
uniformly over $\theta,\theta'\in K$ satisfying $d_{\mathrm{eq}}(\theta,\theta')>2\varepsilon$. Taking the infimum over such pairs proves \eqref{eq:necessary_likelihood_separation}.
It remains to prove the two lemmas.

\begin{proof}[Proof of Lemma~\ref{lem:minimax_implies_hellinger_separation}]
Apply Proposition~\ref{prop:recovery_implies_hellinger_separation_general} with $d=d_{\mathrm{eq}}$ and $n\leq C_\varepsilon\beta^{-2r}$.
\end{proof}

\begin{proof}[Proof of Lemma~\ref{lem:hellinger_likelihood_l2_comparison}]
Let $L_0=L_{\theta,\beta}, L_1=L_{\theta',\beta}$. By Lemma~\ref{lem:info_likelihood_ratio_representations},
\begin{align}
    h^2(P_{\theta,\beta},P_{\theta',\beta}) = \mathbb{E}_{Y\sim P_0} \left[(\sqrt{L_0}-\sqrt{L_1})^2 \right].
\end{align}
For all $L_0,L_1\geq 0$,
\begin{align}
    (\sqrt{L_0}-\sqrt{L_1})^2 = L_0+L_1-2\sqrt{L_0L_1} \leq L_0+L_1.
\end{align}
Let $ A\triangleq\{L_0\geq 1/2\}$. On $A$,
\begin{align}
    (\sqrt{L_0}-\sqrt{L_1})^2 = \frac{(L_0-L_1)^2} {(\sqrt{L_0}+\sqrt{L_1})^2} \leq 2(L_0-L_1)^2.
\end{align}
Therefore, splitting the expectation over $A$ and $A^c$,
\begin{align}
    h^2(P_{\theta,\beta},P_{\theta',\beta})
    &\leq 2 \, \mathbb{E}_{Y\sim P_0} \left[(L_0-L_1)^2\mathbf{1}_A \right] + \mathbb{E}_{Y\sim P_0} \left[ (L_0+L_1)\mathbf{1}_{A^c} \right]
    \\
    &\leq 2\norm{L_0-L_1}_{L^2(P_0)}^2 + \mathbb{E}_{Y\sim P_0} \left[(L_0+L_1)\mathbf{1}_{A^c} \right].    \label{eq:hellinger_l2_reduction}
\end{align}

By Assumption~\ref{ass:regular_observed_likelihood_ratio_expansion},
for every finite $p$, $\sup_{\vartheta\in\Theta} \norm{ L_{\vartheta,\beta}-1}_{L^p(P_0)} = O(\beta)$. Therefore, $P_0(A^c) = O(\beta^p)$. Since $L_0\leq 1/2$ on $A^c$ and $\sup_{\vartheta\in\Theta} \norm{L_{\vartheta,\beta}}_{L^2(P_0)}=O(1)$, Cauchy--Schwarz gives $\mathbb{E}_{Y\sim P_0} \left[ (L_0+L_1)\mathbf{1}_{A^c} \right] = O(\beta^{p/2})$. As $p$ may be chosen arbitrarily large, for every $M<\infty$,
$\mathbb{E}_{Y\sim P_0} \left[ (L_0+L_1)\mathbf{1}_{A^c} \right] \leq C_M\beta^M$ for all sufficiently small $\beta$, uniformly over $\theta,\theta'\in\Theta$. Substituting into~\eqref{eq:hellinger_l2_reduction} proves \eqref{eq:hellinger_likelihood_l2_upper}.
\end{proof}

\section{Rate-optimal recovery via maximum likelihood}
\label{app:general_mle}

The goal of this appendix is to establish an exponential-tail bound for the maximum-likelihood estimator. The basic idea is simple. If the maximum-likelihood estimator lies at quotient distance greater than $\varepsilon$ from the true parameter $\theta^\star$, then there must exist an alternative parameter whose likelihood is at least as large as that of $\theta^\star$. If, in addition, every such alternative is separated from the true distribution by a Hellinger distance of at least $\tau_\varepsilon$, then the error event is contained in
\begin{align}
    \left\{ \sup_{\substack{\theta\in\Theta:\ h(P_{\theta,\beta},P_{\theta^\star,\beta})\geq\tau_\varepsilon}} \prod_{i=1}^{n} \frac{p_{\theta,\beta}(Y_i)} {p_{\theta^\star,\beta}(Y_i)} \geq 1 \right\}.    \label{eq:mle_likelihood_ratio_event}
\end{align}
A likelihood-ratio maximal inequality of Wong and Shen~\cite{wong1995probability} shows that, under suitable control of the local Hellinger complexity of the likelihood family,
\begin{align}
    P_{\theta^\star,\beta}^{\otimes n} \left( \sup_{\substack{\theta\in\Theta:\ h(P_{\theta,\beta},P_{\theta^\star,\beta})\geq\tau_\varepsilon}} \prod_{i=1}^{n} \frac{p_{\theta,\beta}(Y_i)} {p_{\theta^\star,\beta}(Y_i)} \geq 1 \right) \leq C_\varepsilon \exp \{-c_\varepsilon n\tau_\varepsilon^2\}.
    \label{eq:wong_shen_step}
\end{align}
Thus, once parameters at quotient distance greater than $\varepsilon$ are separated in Hellinger distance by $\tau_\varepsilon$, the probability that any such parameter can outperform the truth in likelihood decays exponentially at rate $n\tau_\varepsilon^2$.

In the weak-signal regime, we show that, for every fixed $\varepsilon>0$, $d_{\mathrm{eq}}(\theta,\theta^\star)>\varepsilon$ implies $h(P_{\theta,\beta},P_{\theta^\star,\beta}) \geq a_\varepsilon\beta^{r_{\mathrm{glob}}}$ uniformly for sufficiently small $\beta$. Hence, we may take $\tau_\varepsilon \triangleq a_\varepsilon\beta^{r_{\mathrm{glob}}}$, and \eqref{eq:wong_shen_step} yields
\begin{align}
    P_{\theta^\star,\beta}^{\otimes n} \left( d_{\mathrm{eq}}(\widehat{\theta}_{\mathrm{MLE}},\theta^\star)>\varepsilon \right) \leq C_\varepsilon \exp \{-c_\varepsilon' n\beta^{2r_{\mathrm{glob}}}\}.   \label{eq:mle_exponential_tail_motivation}
\end{align}
Consequently, estimation with error probability at most $\delta$ requires $n \gtrsim \beta^{-2r_{\mathrm{glob}}} \log (1/\delta)$, which matches the information-theoretic weak-signal scale.

\subsection{Likelihood-ratio maximal inequality}

For maximum-likelihood upper bounds, pairwise separation is not enough: one must control the likelihood uniformly over an entire class of alternatives. Wong and Shen~\cite[Theorem~1]{wong1995probability} provide such a bound in terms of local Hellinger bracketing entropy. We use the following consequence of their result, obtained from a logarithmic entropy bound and a dyadic peeling argument.

\paragraph{Hellinger bracketing entropy.}
Bracketing entropy is a standard measure of the complexity of a function class in empirical-process theory; see, e.g., \cite{vanderVaartWellner1996,gao2007entropy}. For a class $\mathcal{F}$ of measurable functions, a bracket $[\underline f,\overline f]$ consists of all $f\in\mathcal{F}$ satisfying
\begin{align}
    [\underline f,\overline f] \triangleq \left\{ f\in\mathcal{F}: \underline f\leq f\leq\overline f \quad P_0\text{-a.s.} \right\}.
\end{align}
Its $L^2(P_0)$ width is $\norm{\overline f-\underline f}_{L^2(P_0)}$. The bracketing number $N_{[]}(u,\mathcal{F},L^2(P_0))$ is the smallest number of brackets of width at most $u$ needed to cover $\mathcal{F}$, and the corresponding bracketing entropy is
\begin{align}
    H_{[]}(u,\mathcal{F},L^2(P_0)) \triangleq \log N_{[]}(u,\mathcal{F},L^2(P_0)). \label{eq:bracketing_entropy_definition}
\end{align}

In likelihood problems, this complexity is naturally applied to
square-root densities, i.e., $h(P,Q) =\norm{\sqrt p-\sqrt q}_{L^2(P_0)}$. The bracketing entropy therefore quantifies the local size of the likelihood class at the statistical scale relevant for controlling its supremum.

Let $\mathcal{P}$ be a class of probability measures dominated by a common probability measure $P_0$, let $P^\star\in\mathcal{P}$, and for $t>0$, define $\mathcal{P}^\star(t) \triangleq \{ P\in\mathcal{P}: h(P,P^\star)\leq t \}$.

\begin{lem}[Likelihood-ratio maximal inequality]
\label{lem:likelihood_ratio_maximal_inequality}
Suppose $\sup_{P\in\mathcal{P}}h(P,P^\star)\leq 2t_0$ and, for every $0<t\leq t_0$,
\begin{align}
    H_{[]} \left( u, \left\{ \sqrt p: P\in\mathcal{P}^\star(2t) \right\}, L^2(P_0) \right) \leq D\log\left(1+\frac{At}{u}\right), \label{eq:lr_maximal_entropy_condition}
\end{align}
for $c_0t^2\leq u\leq C_0t$, where $c_0>0$ is sufficiently small and $C_0>0$ sufficiently large. Then there exist constants $c,C,B>0$ such that, for every $0<\tau\leq t_0/2$ satisfying $n\tau^2\geq B$,
\begin{align}
    P^{\star\otimes n} \left( \sup_{\substack{ P\in\mathcal{P}:\ h(P,P^\star)\geq\tau }} \prod_{i=1}^n \frac{p(Y_i)}{p^\star(Y_i)} \geq1 \right) \leq C\exp \{-cn\tau^2\}.
    \label{eq:lr_maximal_bound}
\end{align}
\end{lem}

\begin{proof}
For a Hellinger scale $t$, the entropy bound~\eqref{eq:lr_maximal_entropy_condition} implies
\begin{align}
    \int_{t^2/2^8}^{\sqrt{2}t} H_{[]}^{1/2} \left( \frac{u}{c_3}, \left\{\sqrt p:P\in\mathcal{P}^\star(2t)\right\}, L^2(P_0) \right) du \leq Ct,
\end{align}
where $c_3$ is the constant in Wong and Shen~\cite[Theorem~1]{wong1995probability}. Hence their entropy-integral condition is satisfied whenever $nt^2\geq B$, for a sufficiently large constant $B$.

Set $\mathcal{S}_j \triangleq \left\{ P\in\mathcal{P}: 2^j\tau\leq h(P,P^\star)<2^{j+1}\tau \right\}$. For every $j$ such that $2^j\tau\leq t_0$, applying~\cite[Theorem~1]{wong1995probability} at scale $2^j\tau$ gives
\begin{align}
    P^{\star\otimes n} \left( \sup_{P\in\mathcal{S}_j} \prod_{i=1}^n \frac{p(Y_i)}{p^\star(Y_i)} \geq \exp \{-c'n4^j\tau^2\} \right) \leq C\exp \{-cn4^j\tau^2\}.
\end{align}
Since $\exp\{-c'n4^j\tau^2\}\leq1$,
\begin{align}
    \left\{ \sup_{P\in\mathcal{S}_j} \prod_{i=1}^n \frac{p(Y_i)}{p^\star(Y_i)} \geq1 \right\} \subseteq \left\{ \sup_{P\in\mathcal{S}_j} \prod_{i=1}^n \frac{p(Y_i)}{p^\star(Y_i)} \geq \exp\{-c'n4^j\tau^2\} \right\},
\end{align}
and therefore
\begin{align}
    P^{\star\otimes n} \left( \sup_{P\in\mathcal{S}_j} \prod_{i=1}^n \frac{p(Y_i)}{p^\star(Y_i)} \geq1 \right) \leq C\exp \{-cn4^j\tau^2\}.
\end{align}
The remaining alternatives with $h(P,P^\star)\geq t_0$ are handled by the same argument at scale $t_0$. Summing the resulting bounds over the shells yields $\sum_{j\geq0} C\exp \{-cn4^j\tau^2\} \leq C\exp \{-cn\tau^2\}$, which proves~\eqref{eq:lr_maximal_bound}.
\end{proof}

Motivated by Lemma~\ref{lem:likelihood_ratio_maximal_inequality}, we impose the local Hellinger-complexity condition~\eqref{eq:lr_maximal_entropy_condition} as an explicit assumption. We will later verify that it holds for the finite-dimensional likelihood-moment representation considered in this work.

Let $r=r_{\mathrm{glob}}<\infty$ and let $K\subseteq\Theta_{\mathrm{gen}}$ be compact.
For $\tau>0$ and $\theta^\star\in K$, define the local square-root likelihood class
\begin{align}
    \mathcal{H}_{\beta,\theta^\star}(\tau) \triangleq \left\{ \sqrt{L_{\theta,\beta}}: h(P_{\theta,\beta},P_{\theta^\star,\beta})\leq\tau, \ \theta\in\Theta \right\}. \label{eq:local_hellinger_class}
\end{align}
The next condition controls the local size of the likelihood family, similar to~\eqref{eq:lr_maximal_entropy_condition}.

\begin{assum}[Controlled local Hellinger complexity]
\label{ass:local_complexity}
For every $b>0$, there exist constants $C,D,c_0,C_0,\tau_0,\beta_0>0$, such that uniformly over $\theta^\star\in K$, $0<\beta\leq\beta_0$, and $b\beta^r\leq\tau\leq\tau_0$,
\begin{align}
    H_{[]} \left( u, \mathcal{H}_{\beta,\theta^\star}(2\tau), L^2(P_0) \right) \leq D\log\left(1+\frac{C\tau}{u}\right),
    \qquad
    c_0\tau^2\leq u\leq C_0\tau.
    \label{eq:local_complexity}
\end{align}
\end{assum}

\subsection{General maximum-likelihood upper bound}

The two-point argument of Proposition~\ref{prop:recovery_implies_hellinger_separation_general} and Lemma~\ref{lem:minimax_implies_hellinger_separation} shows that Hellinger separation at scale $\beta^r$ is necessary for recovery with $n\asymp\beta^{-2r}$ samples. The next result shows that, together with controlled local likelihood complexity, this separation is also sufficient for maximum likelihood.

\begin{thm}[Rate-optimality of maximum likelihood]
\label{thm:mle_general}
Suppose Assumption~\ref{ass:local_complexity} holds. Assume moreover that, for every fixed $\varepsilon>0$, there exist $a_\varepsilon>0$ and $\beta_\varepsilon>0$ such that
\begin{align}
    d_{\mathrm{eq}}(\theta,\theta^\star)>\varepsilon
    \quad\Longrightarrow\quad h(P_{\theta,\beta},P_{\theta^\star,\beta}) \geq a_\varepsilon\beta^r, \label{eq:stable_hellinger_separation}
\end{align}
uniformly over $\theta^\star\in K$, $\theta\in\Theta$, and $0<\beta\leq\beta_\varepsilon$. Then, for every fixed $\varepsilon>0$, there exist constants $c_\varepsilon,C_\varepsilon,\beta_0(\varepsilon)>0$ such that, for every $0<\beta\leq\beta_0(\varepsilon)$ and $n\beta^{2r}\geq C_\varepsilon$,
\begin{align}
    \sup_{\theta^\star\in K} P_{\theta^\star,\beta}^{\otimes n} \left( d_{\mathrm{eq}} ( \widehat\theta_{\mathrm{MLE}}, \theta^\star ) > \varepsilon \right) \leq C_\varepsilon \exp\{-c_\varepsilon n\beta^{2r}\}.
    \label{eq:mle_exponential_bound}
\end{align}
Consequently, for fixed $\varepsilon>0$, $n\gtrsim \beta^{-2r}\log(1/\delta)$ observations suffice for quotient error at most $\varepsilon$ with probability at least $1-\delta$.
\end{thm}

\begin{proof}
Fix $\varepsilon>0$ and $\theta^\star\in K$, and set $\tau_\varepsilon \triangleq a_\varepsilon\beta^r$. By Assumption~\ref{ass:local_complexity}, Lemma~\ref{lem:likelihood_ratio_maximal_inequality} gives, provided $n\tau_\varepsilon^2\geq C_\varepsilon$,
\begin{align}
    P_{\theta^\star,\beta}^{\otimes n} \left(\sup_{\substack{\theta\in\Theta:\\ h(P_{\theta,\beta},P_{\theta^\star,\beta})   \geq\tau_\varepsilon}} \prod_{i=1}^{n} \frac{ p_{\theta,\beta}(Y_i) }{   p_{\theta^\star,\beta}(Y_i) } \geq 1 \right) \leq C_\varepsilon \exp\{-c_\varepsilon n\tau_\varepsilon^2\}.    \label{eq:wong_shen_step2}
\end{align}

By the defining property of the MLE, $\prod_{i=1}^{n}  \left(p_{\widehat\theta_{\mathrm{MLE}},\beta}(Y_i) / p_{\theta^\star,\beta}(Y_i) \right)  \geq 1$. Together with~\eqref{eq:stable_hellinger_separation}, this implies
\begin{align}
    \left\{ d_{\mathrm{eq}} ( \widehat\theta_{\mathrm{MLE}}, \theta^\star ) > \varepsilon \right\} \subseteq \left\{ \sup_{\substack{\theta\in\Theta:\\ h(P_{\theta,\beta},P_{\theta^\star,\beta}) \geq\tau_\varepsilon}} \prod_{i=1}^{n} \frac{ p_{\theta,\beta}(Y_i) }{   p_{\theta^\star,\beta}(Y_i) } \geq 1 \right\}.
\end{align}
Since $\tau_\varepsilon=a_\varepsilon\beta^r$, \eqref{eq:wong_shen_step2} yields~\eqref{eq:mle_exponential_bound}. The sample-complexity statement follows immediately.
\end{proof}

\subsection{Finite-dimensional verification}

We now show that the Hellinger complexity of Assumption~\ref{ass:local_complexity} holds under the Assumptions~\ref{ass:finite_dimensional_likelihood_geometry}, \ref{ass:nondegenerate_cross_order_likelihood_geometry}, and~\ref{ass:uniform_likelihood_remainder_envelope}.

\begin{lem}[Finite-dimensional local Hellinger complexity]
\label{lem:finite_dimensional_local_hellinger_complexity}
Suppose Assumptions~\ref{ass:finite_dimensional_likelihood_geometry}, \ref{ass:nondegenerate_cross_order_likelihood_geometry}, and~\ref{ass:uniform_likelihood_remainder_envelope} hold.
Then, for every $b>0$, there exist constants $D,C,c_0,C_0,\tau_0,\beta_0>0$ such that, uniformly over $\theta^\star\in K$, $0<\beta\leq\beta_0$, and $b\beta^r\leq\tau\leq\tau_0$,
\begin{align}
    H_{[]} \left( u, \mathcal{H}_{\beta,\theta^\star}(2\tau), L^2(P_0) \right) \leq D\log\left(1+\frac{C\tau}{u}\right),
    \qquad
    c_0\tau^2\leq u\leq C_0\tau.    \label{eq:finite_dimensional_local_hellinger_entropy}
\end{align}
\end{lem}

\begin{proof}
Write $\rho_{\theta,\beta}\triangleq \sqrt{L_{\theta,\beta}}$, and $U_{\theta,\beta}\triangleq L_{\theta,\beta}-1$. By the likelihood-ratio expansion,
\begin{align}
    U_{\theta,\beta} = \sum_{j=1}^{2r}\beta^j\phi_j(\cdot;\theta) + R_{2r+1}(\cdot;\theta,\beta),
\end{align}
where the coefficients through order $2r$ lie in fixed finite-dimensional subspaces. Moreover, Assumption~\ref{ass:uniform_likelihood_remainder_envelope} and compactness give
\begin{align}
    \left\| \sup_{\theta\in\Theta}|U_{\theta,\beta}| \right\|_{L^{4r+2}(P_0)} = O(\beta).
\label{eq:uniform_likelihood_increment_envelope}
\end{align}

Let $T_{2r}(x) \triangleq \sum_{k=0}^{2r} \binom{1/2}{k}x^k$ be the degree-$2r$ Taylor polynomial of $\sqrt{1+x}$ at zero. For $x\geq-1$,
\begin{align}
    \left| \sqrt{1+x}-T_{2r}(x) \right| \leq C_r|x|^{2r+1}.
    \label{eq:sqrt_scalar_remainder}
\end{align}
Applying this bound to $U_{\theta,\beta}$ and collecting powers of $\beta$ gives
\begin{align}
    \rho_{\theta,\beta} = 1+ \sum_{j=1}^{2r} \beta^j\widetilde{\phi}_j(\cdot;\theta) + \widetilde{R}_{2r+1}(\cdot;\theta,\beta),
    \label{eq:sqrt_expansion_review}
\end{align}
where the $\widetilde{\phi}_j$ lie in a fixed finite-dimensional subspace of $L^2(P_0)$, and
\begin{align}
    \left\| \sup_{\theta\in\Theta} |\widetilde{R}_{2r+1}(\cdot;\theta,\beta)| \right\|_{L^2(P_0)} = O(\beta^{2r+1}). \label{eq:sqrt_uniform_remainder_review}
\end{align}

If $\rho_{\theta,\beta}\in\mathcal{H}_{\beta,\theta^\star}(2\tau)$, then $\|\rho_{\theta,\beta}-\rho_{\theta^\star,\beta}\|_{L^2(P_0)} \leq 2\tau$.
Hence the finite-dimensional component of $\rho_{\theta,\beta}-\rho_{\theta^\star,\beta}$ lies in an $L^2(P_0)$ ball of radius $C\tau$. Since $\tau\geq b\beta^r$, $\beta^{2r+1}=o(\tau^2)$, so the remainder class is contained, for sufficiently small $\beta$, in a common bracket of $L^2(P_0)$ width $o(\tau^2)$.

A $D$-dimensional $L^2(P_0)$ ball of radius $C\tau$ admits a bracketing cover of cardinality at most $(1+C\tau/u)^D$. For $u\geq c_0\tau^2$, the remainder bracket can be absorbed into these brackets after adjusting constants. Therefore,
\begin{align}
    H_{[]} \left( u, \mathcal{H}_{\beta,\theta^\star}(2\tau), L^2(P_0) \right) \leq D\log\left(1+\frac{C\tau}{u}\right),
\end{align}
uniformly over the stated range, proving~\eqref{eq:finite_dimensional_local_hellinger_entropy}.
\end{proof}

\begin{lem}[Finite-dimensional likelihood-to-Hellinger comparison]
\label{lem:finite_dimensional_likelihood_hellinger_comparison}
Suppose Assumptions~\ref{ass:finite_dimensional_likelihood_geometry} and \ref{ass:nondegenerate_cross_order_likelihood_geometry} hold. For every fixed $r\geq 1$, there exist constants $c,C,\beta_0>0$ such that, for every $0<\beta\leq\beta_0$ and every $\theta,\theta'\in\Theta$,
\begin{align}
    h(P_{\theta,\beta},P_{\theta',\beta}) \geq c \norm{ L_{\theta,\beta}-L_{\theta',\beta} }_{L^2(P_0)} - C\beta^{r+1}. \label{eq:finite_dimensional_likelihood_hellinger_comparison}
\end{align}
\end{lem}

\begin{proof}
By the uniform order-$r$ likelihood-ratio expansion,
\begin{align}
    L_{\theta,\beta}-L_{\theta',\beta}
    &=
    \sum_{\ell=1}^{r} \beta^\ell \left[ \Phi_\ell(\theta)-\Phi_\ell(\theta') \right] + R_{r+1}(\theta,\theta',\beta),
\end{align}
where $\sup_{\theta,\theta'\in\Theta} \norm{R_{r+1}(\theta,\theta',\beta)}_{L^2(P_0)} = O(\beta^{r+1})$.
The truncated sum belongs to the fixed finite-dimensional space $\mathcal{H}_r \triangleq \operatorname{span} \left\{ \Phi_\ell(\vartheta): \vartheta\in\Theta,\; 1\leq\ell\leq r \right\}$. Hence, by equivalence of norms on $\mathcal{H}_r$, there exists $c_0>0$ such that $\norm{v}_{L^1(P_0)} \geq c_0\norm{v}_{L^2(P_0)}$, for every $v\in\mathcal{H}_r$. Using the triangle inequality and
$\norm{R}_{L^1(P_0)}\leq\norm{R}_{L^2(P_0)}$,
\begin{align}
    \norm{ L_{\theta,\beta}-L_{\theta',\beta} }_{L^1(P_0)} \geq c_0 \norm{ L_{\theta,\beta}-L_{\theta',\beta} }_{L^2(P_0)} - C\beta^{r+1}.
\end{align}
Finally, by~\eqref{eq:info_tv_hellinger},
\begin{align}
    h(P_{\theta,\beta},P_{\theta',\beta}) \geq \mathrm{TV}(P_{\theta,\beta},P_{\theta',\beta}) = \frac{1}{2} \norm{ L_{\theta,\beta}-L_{\theta',\beta} }_{L^1(P_0)},
\end{align}
and the result follows after adjusting constants.
\end{proof}

\subsection{Proof of Theorem~\ref{thm:mle_upper_bound}}
\label{subsec:thm_mle_upper_bound}

Write $r=r_{\mathrm{glob}}$ and fix $\varepsilon>0$.
Lemma~\ref{lem:fixed_accuracy_likelihood_separation} gives the stable
likelihood separation
\begin{align}
    d_{\mathrm{eq}}(\theta,\theta^\star)>\varepsilon
    \quad\Longrightarrow\quad
    \norm{
        L_{\theta,\beta}
        -
        L_{\theta^\star,\beta}
    }_{L^2(P_0)}
    \geq
    c_\varepsilon\beta^r
    \label{eq:finite_dimensional_fixed_accuracy_l2_separation}
\end{align}
for all sufficiently small $\beta$, uniformly over $\theta^\star\in K$ and $\theta\in\Theta$. By Lemma~\ref{lem:finite_dimensional_likelihood_hellinger_comparison} and~\eqref{eq:finite_dimensional_fixed_accuracy_l2_separation}, for every fixed $\varepsilon>0$ and all sufficiently small $\beta$,
\begin{align}
    d_{\mathrm{eq}}(\theta,\theta^\star)>\varepsilon
    \quad\Longrightarrow\quad
    h(P_{\theta,\beta},P_{\theta^\star,\beta})
    \geq
    c'_\varepsilon\beta^r,
\end{align}
uniformly over $\theta^\star\in K$ and $\theta\in\Theta$.
Thus the Hellinger-separation condition in
Theorem~\ref{thm:mle_general} holds.
Lemma~\ref{lem:finite_dimensional_local_hellinger_complexity} verifies Assumption~\ref{ass:local_complexity}; thus, the result follows directly from Theorem~\ref{thm:mle_general}.

\section{Rate-optimal recovery via method-of-moments}
\label{sec:rate_optimal_recovery_MoM}

\subsection{Finite-dimensional likelihood geometry}

The following elementary finite-dimensional result  will be used repeatedly to transfer between Euclidean coordinate norms and $L^2(P_0)$ norms of the corresponding likelihood functions.

\begin{lem}[Finite-dimensional Gram geometry]
\label{lem:finite_dimensional_gram_geometry}
Let $q=(q_1,\ldots,q_D)^\top\in L^2(P_0)^D$, and define $\mathbf{G} \triangleq \mathbb{E}_{Y\sim P_0} \left[ q(Y)q(Y)^\top \right]$. Then:
\begin{enumerate}[label=(\roman*)]

    \item For every $a\in\mathbb{R}^D$,
    \begin{align}
        \norm{\langle a,q\rangle}_{L^2(P_0)}^2 = a^\top \mathbf{G}a. \label{eq:prelim_gram_identity}
    \end{align}
    Consequently, the functions $q_1,\ldots,q_D$ are linearly independent
    in $L^2(P_0)$ if and only if $\mathbf{G}\succ 0$.

    \item If $\mathbf{G}\succ 0$, then, for every $a\in\mathbb{R}^D$,
    \begin{align}
        \sqrt{\lambda_{\min}(\mathbf{G})}\,\norm{a}_2 \leq \norm{\langle a,q\rangle}_{L^2(P_0)} \leq \sqrt{\lambda_{\max}(\mathbf{G})}\,\norm{a}_2.   \label{eq:prelim_gram_norm_equivalence}
    \end{align}

    \item Let $\mathcal{H} \triangleq \operatorname{span}\{q_1,\ldots,q_D\}$. If $\mathbf{G}\succ 0$, then for every $f\in\mathcal{H}$,
    \begin{align}
        \norm{\mathbb{E}_{Y\sim P_0} \left[ q(Y)f(Y) \right]}_2 \geq \sqrt{\lambda_{\min}(\mathbf{G})} \norm{f}_{L^2(P_0)}.        \label{eq:prelim_coordinate_operator_lower_bound}
    \end{align}

\end{enumerate}
\end{lem}

\begin{proof}
For $f=\langle a,q\rangle$,
\begin{align}
    \norm{f}_{L^2(P_0)}^2  = \mathbb{E}_{Y\sim P_0} \left[a^\top q(Y)q(Y)^\top a\right] = a^\top \mathbf{G}a,
\end{align}
which proves~\eqref{eq:prelim_gram_identity}. The equivalence between linear independence of $q_1,\ldots,q_D$ and $\mathbf{G}\succ0$ follows immediately, while~\eqref{eq:prelim_gram_norm_equivalence} follows from the extremal eigenvalue bounds for the quadratic form $a^\top\mathbf{G} a$. Finally, if $f\in\mathcal{H}$, write $f=\langle a,q\rangle$. Then $\mathbb{E}_{Y\sim P_0}[q(Y)f(Y)] = \mathbf{G} a, $ and, since $\mathbf{G}\succ0$,
\begin{align}
    \norm{\mathbf{G} a}_2^2 = a^\top\mathbf{G}^2a \geq \lambda_{\min}(\mathbf{G})\, a^\top\mathbf{G} a = \lambda_{\min}(\mathbf{G})\norm{f}_{L^2(P_0)}^2.
\end{align}
Taking square roots proves~\eqref{eq:prelim_coordinate_operator_lower_bound}.
\end{proof}

Under Assumption~\ref{ass:finite_dimensional_likelihood_geometry}, the likelihood moments admit the coordinate representation $\Phi_\ell(\theta) = \langle \alpha_\ell(\theta), q_\ell \rangle$. For every $k\geq 1$ and $\theta,\theta'\in\Theta$, define
\begin{align}
    \Delta\alpha_{\leq k}(\theta,\theta')
    &\triangleq \alpha_{\leq k}(\theta) - \alpha_{\leq k}(\theta'),
    \\
    D_{\beta,k}
    &\triangleq \operatorname{diag} \left(\beta I_{D_1}, \beta^2 I_{D_2}, \ldots, \beta^k I_{D_k} \right).
\end{align}
Then the likelihood-ratio difference truncated through order $k$ can be
written as
\begin{align}
    \sum_{\ell=1}^{k} \beta^\ell \left[ \Phi_\ell(\theta) - \Phi_\ell(\theta') \right] = \left\langle D_{\beta,k} \Delta\alpha_{\leq k}(\theta,\theta'), q_{\leq k} \right\rangle. \label{eq:truncated_likelihood_difference_coordinates}
\end{align}
Thus, $\Delta\alpha_{\leq k}$ describes the parameter-side separation, while $D_{\beta,k}$ accounts for the different weak-signal scales of the
successive likelihood orders.

The following elementary consequence of the finite-dimensional Gram geometry will be used below.

\begin{lem}[Gram bound for truncated likelihood differences]
\label{lem:truncated_likelihood_gram_bound}
For every $k\geq1$, let
\begin{align}
    \mathbf{G}_k \triangleq \mathbb{E}_{Y\sim P_0} \left[ q_{\leq k}(Y)q_{\leq k}(Y)^\top \right].
    \end{align}
Then
\begin{align}
    \left\| \sum_{\ell=1}^{k} \beta^\ell \left[ \Phi_\ell(\theta)-\Phi_\ell(\theta') \right] \right\|_{L^2(P_0)}^2 = \Delta\alpha_{\leq k}(\theta,\theta')^\top  D_{\beta,k} \mathbf{G}_k  D_{\beta,k} \Delta\alpha_{\leq k}(\theta,\theta').   \label{eq:truncated_likelihood_gram_identity}
\end{align}
Consequently, if $\mathbf{G}_k\succ0$, then, for every $0<\beta\leq1$,
\begin{align}
    \left\|
    \sum_{\ell=1}^{k} \beta^\ell \left[ \Phi_\ell(\theta)-\Phi_\ell(\theta') \right] \right\|_{L^2(P_0)}  \geq \sqrt{\lambda_{\min}(\mathbf{G}_k)} \beta^k \left\| \Delta\alpha_{\leq k}(\theta,\theta') \right\|_2.    \label{eq:truncated_likelihood_gram_lower}
\end{align}
\end{lem}

\begin{proof}
By~\eqref{eq:truncated_likelihood_difference_coordinates} and Lemma~\ref{lem:finite_dimensional_gram_geometry}, \eqref{eq:truncated_likelihood_gram_identity} follows immediately. Applying~\eqref{eq:prelim_gram_norm_equivalence} with $a=D_{\beta,k}\Delta\alpha_{\leq k}(\theta,\theta')$ and using $0<\beta\leq1$ gives~\eqref{eq:truncated_likelihood_gram_lower}.
\end{proof}

\subsection{Compactness and deterministic estimation}

The following compactness argument upgrades pointwise identifiability of a continuous finite-dimensional representation to uniform separation of parameters that remain a fixed positive distance apart.

\begin{lem}[Uniform separation on compact sets]
\label{lem:compact_uniform_separation}
Let $(\Theta,d)$ be a compact metric space, let $K\subseteq\Theta$ be compact, and let $F:\Theta\to\mathbb{R}^D$ be continuous.
Suppose that, for every $\theta^\star\in K$,
\begin{align}
    F(\theta)=F(\theta^\star)    \quad\Longrightarrow\quad    d(\theta,\theta^\star)=0.    \label{eq:prelim_compact_identifiability}
\end{align}
Then, for every $\varepsilon>0$,
\begin{align}
    \inf_{\substack{\theta^\star\in K,\;\theta\in\Theta:\\ d(\theta,\theta^\star)\geq\varepsilon}} \norm{F(\theta)-F(\theta^\star)}_2 >0.   \label{eq:prelim_compact_uniform_separation}
\end{align}
\end{lem}

\begin{proof}
The constraint set in \eqref{eq:prelim_compact_uniform_separation} is compact, and the objective is continuous and strictly positive on that set by \eqref{eq:prelim_compact_identifiability}. Hence it attains a strictly positive minimum.
\end{proof}

The following elementary reduction isolates the deterministic step underlying the method-of-moments analysis: once the population moment map separates parameters by a fixed amount, an estimation error can occur only if the empirical moment vector deviates sufficiently far from its population value.

\begin{lem}[Minimum-distance reduction]
\label{lem:minimum_distance_reduction}
Let $\widehat{\Gamma}$ be a random vector in $\mathbb{R}^D$, let $\Gamma:\Theta\to\mathbb{R}^D$, and let $\theta^\star\in\Theta$. Define
\begin{align}
    \widehat\theta \in \argmin_{\theta\in\Theta} \norm{\widehat\Gamma-\Gamma(\theta)}_2. \label{eq:prelim_minimum_distance_estimator}
\end{align}
Suppose that, for some $\varepsilon,\Delta>0$,
\begin{align}
    \inf_{\substack{ \theta\in\Theta:\\ d(\theta,\theta^\star)>\varepsilon }} \norm{\Gamma(\theta)-\Gamma(\theta^\star)}_2 \geq \Delta. \label{eq:prelim_minimum_distance_population_separation}
\end{align}
Then an estimation error larger than $\varepsilon$ can occur only if the empirical moment vector deviates from its population value by at least $\Delta/2$. In particular,
\begin{align}
    \mathbb{P}\left( d(\widehat\theta,\theta^\star)>\varepsilon \right) \leq \mathbb{P}\left( \norm{\widehat\Gamma-\Gamma(\theta^\star)}_2 \geq \frac{\Delta}{2} \right). \label{eq:prelim_minimum_distance_probability_reduction}
\end{align}
\end{lem}

\begin{proof}
By optimality of $\widehat\theta$, $\|{\widehat\Gamma-\Gamma(\widehat\theta)\|}_2 \leq  \|{\widehat\Gamma-\Gamma(\theta^\star)\|}_2$. Hence, by the triangle inequality,
\begin{align}
    \|{\Gamma(\widehat\theta)-\Gamma(\theta^\star)\|}_2\leq \|{\Gamma(\widehat\theta)-\widehat\Gamma\|}_2 + \|{\widehat\Gamma-\Gamma(\theta^\star)\|}_2 \leq 2\|{\widehat\Gamma-\Gamma(\theta^\star)\|}_2.    \label{eq:prelim_minimum_distance_basic_bound}
\end{align}
Therefore, if
$d(\widehat\theta,\theta^\star)>\varepsilon$, then \eqref{eq:prelim_minimum_distance_population_separation} implies
\begin{align}
    \Delta \leq \|{\Gamma(\widehat\theta)-\Gamma(\theta^\star)\|}_2 \leq 2\|{\widehat\Gamma-\Gamma(\theta^\star)\|}_2.
\end{align}
Thus, $d(\widehat\theta,\theta^\star)>\varepsilon$ implies $\|{\widehat\Gamma-\Gamma(\theta^\star)\|}_2 \geq \Delta/2$. Taking probabilities yields~\eqref{eq:prelim_minimum_distance_probability_reduction}.
\end{proof}

\subsection{Proof of Lemma~\ref{lem:fixed_accuracy_likelihood_separation}}
\label{subsec:lem_fixed_accuracy_likelihood_separation}

Write $r=r_{\mathrm{glob}}$ and fix $\varepsilon>0$. For every $\theta^\star\in K\subseteq\Theta_{\mathrm{gen}}$, generic identifiability and Lemma~\ref{lem:coordinate_faithfulness} imply that if  $\alpha_{\leq r}(\theta) = \alpha_{\leq r}(\theta^\star)$, then $d_{\mathrm{eq}}(\theta,\theta^\star)=0$. Since $\alpha_{\leq r}$ is continuous, Lemma~\ref{lem:compact_uniform_separation} gives $\gamma_{K,\varepsilon}>0$ such that
\begin{align}
    d_{\mathrm{eq}}(\theta,\theta^\star)\geq\varepsilon
    \quad\Longrightarrow\quad
    \norm{\alpha_{\leq r}(\theta) - \alpha_{\leq r}(\theta^\star)}_2 \geq \gamma_{K,\varepsilon},    \label{eq:uniform_fixed_accuracy_coordinate_separation}
\end{align}
uniformly over $\theta^\star\in K$ and $\theta\in\Theta$.

By Assumption~\ref{ass:nondegenerate_cross_order_likelihood_geometry}, $\mathbf{G}_r\succ0$.
Hence Lemma~\ref{lem:truncated_likelihood_gram_bound} and \eqref{eq:uniform_fixed_accuracy_coordinate_separation} give for $0<\beta\leq 1$,
\begin{align}
    d_{\mathrm{eq}}(\theta,\theta^\star)\geq\varepsilon
    \quad\Longrightarrow\quad
    \left\|\sum_{\ell=1}^{r} \beta^\ell \left[ \Phi_\ell(\theta) - \Phi_\ell(\theta^\star) \right] \right\|_{L^2(P_0)} \geq c_{K,\varepsilon}\beta^r   \label{eq:uniform_generic_truncated_separation}
\end{align}

Finally, the uniform order-$r$ likelihood-ratio expansion gives
\begin{align}
    L_{\theta,\beta} - L_{\theta^\star,\beta} = \sum_{\ell=1}^{r} \beta^\ell \left[ \Phi_\ell(\theta) - \Phi_\ell(\theta^\star) \right] + O_{L^2(P_0)}(\beta^{r+1}),
\end{align}
uniformly over $\theta^\star\in K$ and $\theta\in\Theta$. Since the remainder is lower order than \eqref{eq:uniform_generic_truncated_separation}, for all sufficiently small $\beta$,
\begin{align}
    d_{\mathrm{eq}}(\theta,\theta^\star)\geq\varepsilon
    \quad\Longrightarrow\quad
    \norm{L_{\theta,\beta} - L_{\theta^\star,\beta}}_{L^2(P_0)} \geq c_\varepsilon\beta^{r_{\mathrm{glob}}},
\end{align}
uniformly over $\theta^\star\in K$ and $\theta\in\Theta$.

\subsection{Proof of Lemma~\ref{lem:bounded_second_moment_likelihood_coordinates}}
\label{app:proof_of_bounded_second_moment_likelihood_coordinates}

For every fixed $k$, the coordinates of $q_{\leq k}$ belong to $L^4(P_0)$.
Indeed, each basis function $q_{\ell,j}$ belongs to $\operatorname{span}\{\Phi_\ell(\theta):\theta\in\Theta\}$, while Assumption~\ref{ass:regular_observed_likelihood_ratio_expansion} implies $\Phi_\ell(\theta)\in L^p(P_0)$ for every finite $p$. Hence $\|q_{\leq k}\|_{L^4(P_0)}<\infty$.

Moreover, the same assumption implies $\sup_{\theta\in\Theta} \sup_{0<\beta\leq\beta_0} \|L_{\theta,\beta}\|_{L^2(P_0)} < \infty$ for sufficiently small $\beta_0$. Therefore, by Hölder's inequality,
\begin{align}
    \mathbb{E}_{Y\sim P_{\theta,\beta}} \left[ \|q_{\leq k}(Y)\|_2^2 \right] &=
    \mathbb{E}_{Y\sim P_0} \left[ L_{\theta,\beta}(Y) \|q_{\leq k}(Y)\|_2^2 \right]
    \notag\\
    &\leq \|L_{\theta,\beta}\|_{L^2(P_0)} \|q_{\leq k}\|_{L^4(P_0)}^2 \leq C_k .
\end{align}
Finally,
\begin{align}
    \mathbb{E}_{Y\sim P_{\theta,\beta}} \left[ \left\| q_{\leq k}(Y) - \Gamma_{k,\beta}(\theta) \right\|_2^2 \right] \leq \mathbb{E}_{Y\sim P_{\theta,\beta}} \left[ \|q_{\leq k}(Y)\|_2^2 \right],
\end{align}
which proves the claim.

\subsection{Proof of Theorem~\ref{thm:mom_upper_bound}}
\label{subsec:proof_thm_mom_upper_bound}

Write $r=r_{\mathrm{glob}}$ and fix $\varepsilon>0$. Recall that $\Gamma_{r,\beta}(\theta) = \mathbb{E}_{Y \sim P_0} \left[ q_{\leq r}(Y)L_{\theta,\beta}(Y) \right]$. Using the order-$r$ likelihood expansion, $L_{\theta,\beta} = 1+ \sum_{\ell=1}^r \beta^\ell \langle \alpha_\ell(\theta), q_\ell \rangle + R_{r+1}(\theta,\beta)$,
we obtain
\begin{align}
    \Gamma_{r,\beta}(\theta) -
    \Gamma_{r,\beta}(\theta^\star)
    &=
    \mathbf{G}_r D_{\beta,r} \left[ \alpha_{\leq r}(\theta) - \alpha_{\leq r}(\theta^\star) \right] + \rho_{r+1}(\theta,\theta^\star,\beta),
    \label{eq:population_moment_expansion}
\end{align}
where, by the uniform likelihood-remainder bound and Cauchy--Schwarz,
\begin{align}
    \sup_{\theta,\theta^\star\in\Theta} \|\rho_{r+1}(\theta,\theta^\star,\beta) \|_2
    \leq C_r\beta^{r+1}.
    \label{eq:population_moment_remainder}
\end{align}
Since Assumption~\ref{ass:nondegenerate_cross_order_likelihood_geometry} gives $\mathbf{G}_r\succ0$, for $0<\beta\leq1$,
\begin{align}
    \left\| \mathbf{G}_r D_{\beta,r} \left[ \alpha_{\leq r}(\theta) - \alpha_{\leq r}(\theta^\star) \right] \right\|_2
    &\geq
    \lambda_{\min}(\mathbf{G}_r) \left\| D_{\beta,r} \left[\alpha_{\leq r}(\theta) - \alpha_{\leq r}(\theta^\star) \right] \right\|_2
    \\
    &\geq
    \lambda_{\min}(\mathbf{G}_r) \beta^r \| \alpha_{\leq r}(\theta) - \alpha_{\leq r}(\theta^\star) \|_2. \label{eq:population_moment_leading_lower_bound}
\end{align}
By~\eqref{eq:uniform_fixed_accuracy_coordinate_separation}, there exists $\gamma_\varepsilon>0$ such that
\begin{align}
    d_{\mathrm{eq}}(\theta,\theta^\star)>\varepsilon
    \quad\Longrightarrow\quad
    \| \alpha_{\leq r}(\theta) - \alpha_{\leq r}(\theta^\star) \|_2 \geq \gamma_\varepsilon,
\end{align}
uniformly over $\theta^\star\in K$ and $\theta\in\Theta$.
Combining this bound with
\eqref{eq:population_moment_expansion}--\eqref{eq:population_moment_remainder}, we obtain
\begin{align}
    \| \Gamma_{r,\beta}(\theta) - \Gamma_{r,\beta}(\theta^\star) \|_2 \geq \lambda_{\min}(\mathbf{G}_r) \gamma_\varepsilon\beta^r - C_r\beta^{r+1}.
\end{align}
Hence, for all sufficiently small $\beta$,
\begin{align}
    d_{\mathrm{eq}}(\theta,\theta^\star)>\varepsilon
    \quad\Longrightarrow\quad
    \| \Gamma_{r,\beta}(\theta) - \Gamma_{r,\beta}(\theta^\star) \|_2 \geq b_\varepsilon\beta^r, \label{eq:population_moment_epsilon_separation}
\end{align}
for some $b_\varepsilon>0$, uniformly over
$\theta^\star\in K$ and $\theta\in\Theta$.

Set $X_i=q_{\leq r}(Y_i)$, and let $\widehat{\Gamma}_{r,\delta}$ denote the robust mean estimator of $\Gamma_{r,\beta}(\theta^\star)=\mathbb{E}[X_i]$ defined above. Since~\eqref{eq:population_moment_epsilon_separation} holds, Lemma~\ref{lem:minimum_distance_reduction} applies with $\Delta=b_\varepsilon\beta^r$, yielding
\begin{align}
    P_{\theta^\star,\beta}^{\otimes n} \left( d_{\mathrm{eq}} \bigl( \widehat{\theta}_{\mathrm{MoM}}, \theta^\star \bigr) > \varepsilon \right) \leq P_{\theta^\star,\beta}^{\otimes n} \left( \left\| \widehat{\Gamma}_{r,\delta}  - \Gamma_{r,\beta}(\theta^\star) \right\|_2  \geq \frac{b_\varepsilon\beta^r}{2}
    \right).
    \label{eqn:app_E30}
\end{align}
By Lemma~\ref{lem:bounded_second_moment_likelihood_coordinates} and the robust concentration bound~\eqref{eq:robust_likelihood_coordinate_concentration},
\begin{align}
    \sup_{\theta^\star\in K}
    P_{\theta^\star,\beta}^{\otimes n} \left( \left\| \widehat{\Gamma}_{r,\delta} - \Gamma_{r,\beta}(\theta^\star) \right\|_2 > C_r \sqrt{\frac{V_r\log(1/\delta)}{n}} \right) \leq  \delta .
    \label{eq:robust_mom_concentration_proof}
\end{align}
Therefore, if
\begin{align}
    C_r \sqrt{\frac{V_r\log(1/\delta)}{n}} \leq \frac{b_\varepsilon\beta^r}{2},
\end{align}
then the minimum-distance reduction above implies
\begin{align}
    \sup_{\theta^\star\in K} P_{\theta^\star,\beta}^{\otimes n}  \left( d_{\mathrm{eq}} \bigl( \widehat{\theta}_{\mathrm{MoM}}, \theta^\star \bigr) > \varepsilon \right) \leq \delta .
\end{align}
Equivalently, after absorbing $V_r$, $C_r$, and $b_\varepsilon$ into a constant $C_\varepsilon$, it is sufficient that $n \geq C_\varepsilon \beta^{-2r} \log (1/{\delta})$. Since $r=r_{\mathrm{glob}}$, this proves the fixed-accuracy upper bound $n \lesssim \beta^{-2r_{\mathrm{glob}}} \log(1/\delta)$.

\section{Algebraic preliminaries}
\label{sec:smooth_quotient_models}

This appendix collects the geometric preliminaries underlying the local and global-to-local analysis in the main text. We first formalize the local smooth structure of the observational quotient, which allows us to separate tangent directions along an equivalence class from transverse, locally identifiable directions. We then show that this local quotient structure holds generically under the compact semi-algebraic parameter spaces and compact group actions assumed in the main text. These results provide the geometric justification for the normal-space formulation of local identifiability and the local rank conditions used throughout the shrinking-accuracy analysis.

\subsection{Local quotient geometry}
\label{subsec:local_quotient_geometry}

Here, we isolate the geometric assumption used in the local analysis. It requires that, near the target equivalence class, the observational quotient has the structure of a smooth finite-dimensional manifold, and that the quotient metric is compatible with this smooth structure. For background on smooth and quotient manifolds, see, e.g.,~\cite{absil2008optimization,lee2003smooth}.

\begin{assum}[Local quotient manifold structure]
\label{ass:equiv_class_submanifold}
Fix $\theta^\star\in\Theta$.
Assume that in a neighborhood of the target $\theta^\star$, the parameter space $\Theta$ is a $C^2$ embedded submanifold of $\mathbb{R}^m$.
Assume further that, in a neighborhood of the target class $[\theta^\star]$, the quotient $\Theta/{\sim}$ is a $C^2$ embedded submanifold, and the canonical projection
\begin{align}
    \pi:\Theta\to\Theta/{\sim}, \qquad
    \pi(\theta)=[\theta],
\end{align}
is a $C^2$ submersion near $\theta^\star$.
Finally, assume that the quotient loss $d_{\mathrm{eq}}$ is locally compatible with this smooth structure, in the sense that in any local chart around $[\theta^\star]$, it is locally equivalent to Euclidean distance.
\end{assum}

Under Assumption~\ref{ass:equiv_class_submanifold}, the equivalence class $[\theta^\star]$ is locally a $C^2$ embedded submanifold of $\Theta$. Its tangent directions are precisely the infinitesimal changes of representative that do not change the observed law, while transverse directions provide local coordinates for the identifiable quotient.

\begin{definition}[Tangent and normal spaces to the equivalence class]
\label{def:tangent_normal_quotient}
Under Assumption~\ref{ass:equiv_class_submanifold}, let $D\pi(\theta^\star)$ denote the differential (Jacobian) of the canonical projection $\pi:\Theta\to\Theta/{\sim}$ at $\theta^\star$. The tangent space to the local equivalence class at $\theta^\star$ is
\begin{align}
    T_{\theta^\star}[\theta^\star] \triangleq \ker D\pi(\theta^\star) \subseteq T_{\theta^\star}\Theta .
\end{align}
We define the normal space to the equivalence class inside the parameter space by
\begin{align}
    W_{\theta^\star} \triangleq \bigl(T_{\theta^\star}[\theta^\star]\bigr)^\perp \cap T_{\theta^\star}\Theta ,
\end{align}
where the orthogonal complement is taken with respect to the Euclidean inner product on $\mathbb{R}^m$.
\end{definition}

The space $T_{\theta^\star}[\theta^\star]$ contains the locally unidentifiable directions. By contrast, $W_{\theta^\star}$ is a local linear model for the identifiable quotient near $[\theta^\star]$. Thus, a local rank condition for the truncated likelihood-moment map on $W_{\theta^\star}$ is the differential form of local identifiability on the quotient.

\subsection{Generic validity under semi-algebraic group actions}
\label{subsec:generic_quotient_structure}

Assumption~\ref{ass:equiv_class_submanifold} is the local condition needed for the analysis.
The stronger structural assumptions used in the main text---namely, a compact semi-algebraic parameter space and a compact Lie group inducing the observational equivalence relation---provide a convenient and broadly applicable sufficient condition for this local quotient structure to hold generically.

\begin{lem}[Generic local quotient structure]
\label{lem:generic_local_quotient_structure}
Suppose that $\Theta\subseteq\mathbb{R}^m$ is compact and semi-algebraic, and that the observational equivalence relation is induced by a compact Lie group $G$ whose action on $\Theta$ extends smoothly to a neighborhood of $\Theta$ (Assumption~\ref{ass:group_induced_equivalence}). Then Assumption~\ref{ass:equiv_class_submanifold} holds for every $\theta^\star$ in a generic regular subset of $\Theta$.
\end{lem}

\begin{proof}
By semi-algebraic stratification, $\Theta$ admits a finite smooth stratification whose top-dimensional strata form a generic subset of $\Theta$; the complement has lower dimension and hence measure zero \cite[Chapter~9]{bochnak1998real}. On each such smooth stratum, the action of the Lie group $G$ is proper. The slice theorem and the principal orbit theorem therefore imply that, outside a further lower-dimensional exceptional set, the orbits have principal type, the local quotient is a smooth manifold, and the quotient map is a smooth submersion \cite[Sections~2.7--2.8, Theorem~2.8.5]{duistermaat2000lie}. The local slice representation, together with the $G$-invariance and local Euclidean equivalence of $d_\Theta$, gives the corresponding local compatibility of $d_{\mathrm{eq}}$ with the quotient manifold structure. Hence Assumption~\ref{ass:equiv_class_submanifold} holds generically.
\end{proof}

Thus, the smooth quotient assumption used in the local analysis is not an additional global modeling restriction. It is the generic local structure induced by the compact semi-algebraic and compact group-action assumptions imposed in the main text.

\subsection{Locally Lipschitz analytic tail}

The local analysis considers the joint regime $\beta\to 0$ and $\varepsilon\to 0$, and therefore requires remainder control that remains valid as the parameter separation also vanishes. The leading truncated likelihood difference between nearby parameters is of order $\beta^s d_{\mathrm{eq}}(\theta,\theta^\star)$, whereas a uniform $O(\beta^{s+1})$ remainder would be negligible only under an additional relation between $\beta$ and $d_{\mathrm{eq}}(\theta,\theta^\star)$. The next lemma shows that local likelihood regularity yields the sharper bound $O(\beta^{s+1}d_{\mathrm{eq}}(\theta,\theta^\star))$. Hence the analytic tail is smaller than the leading local term by a factor $O(\beta)$, uniformly in the joint limit $\beta\to 0$ and $d_{\mathrm{eq}}(\theta,\theta^\star)\to 0$.

\begin{lem}[Locally Lipschitz analytic tail]
\label{lem:locally_lipschitz_analytic_tail}
Let $K\subseteq\Theta_{\mathrm{reg}}$ be the compact regular parameter class, and suppose Assumption~\ref{ass:local_likelihood_regularity} holds. Fix $s\geq 1$. Then, for every finite $p\geq 1$, there exist constants $C_{K,p}<\infty$, $\eta_K>0$, and $\beta_K>0$ such that, uniformly over $\theta^\star\in K$,
\begin{align}
    \Bigg\| \left[L_{\theta,\beta}-L_{\theta^\star,\beta}\right] -\sum_{\ell=1}^{s} \beta^\ell \left[ \Phi_\ell(\theta)-\Phi_\ell(\theta^\star) \right] \Bigg\|_{L^p(P_0)} \leq C_{K,p}\, \beta^{s+1} d_{\mathrm{eq}}(\theta,\theta^\star)    \label{eq:locally_lipschitz_analytic_tail}
\end{align}
whenever $d_{\mathrm{eq}}(\theta,\theta^\star)\leq\eta_K$, and $0<\beta\leq\beta_K$. 
\end{lem}

\begin{proof}
Define the order-$s$ likelihood remainder
\begin{align}
    R_{s+1}(\theta,\beta) \triangleq L_{\theta,\beta} - 1 - \sum_{\ell=1}^{s}   \beta^\ell\Phi_\ell(\theta).    \label{eq:local_tail_remainder_definition}
\end{align}
By Assumption~\ref{ass:local_likelihood_regularity}, the likelihood expansion may be differentiated with respect to $\theta$, and hence
\begin{align}
    \sup_{\vartheta\in U} \norm{D_\theta R_{s+1}(\vartheta,\beta)  }_{\mathrm{op},\,L^p(P_0)} \leq C_p\beta^{s+1}    \label{eq:local_tail_derivative_bound}
\end{align}
for all sufficiently small $\beta$, where $U$ is a fixed neighborhood of $K$ and the norm is the operator norm from the parameter tangent space into $L^p(P_0)$.

Since $K\subseteq\Theta_{\mathrm{reg}}$, the quotient admits uniform local slices around $K$. After replacing $\theta$ by an observationally equivalent representative, which does not change $L_{\theta,\beta}$ or any of the likelihood moments $\Phi_\ell(\theta)$, we may assume that $\theta$ and $\theta^\star$ lie in the same local slice. For $\eta_K$ sufficiently small, they can then be joined in that slice by a $C^1$ path
$\gamma:[0,1]\to U$, such that $\gamma(0)=\theta^\star$, and $\gamma(1)=\theta$, whose length satisfies
\begin{align}
    \int_0^1\norm{\dot\gamma(t)}_2\,dt \leq C_K d_{\mathrm{eq}}(\theta,\theta^\star),    \label{eq:local_slice_path_length}
\end{align}
uniformly over $\theta^\star\in K$.

Since $R_{s+1}(\cdot,\beta)$ is continuously differentiable as an $L^p(P_0)$-valued map, the fundamental theorem of calculus gives
\begin{align}
    R_{s+1}(\theta,\beta) - R_{s+1}(\theta^\star,\beta) = \int_0^1 D_\theta R_{s+1}(\gamma(t),\beta) [\dot\gamma(t)] \,dt .
\end{align}
Therefore, by~\eqref{eq:local_tail_derivative_bound} and \eqref{eq:local_slice_path_length},
\begin{align}
    \norm{R_{s+1}(\theta,\beta) - R_{s+1}(\theta^\star,\beta) }_{L^p(P_0)}
    &\leq C_p\beta^{s+1} \int_0^1\norm{\dot\gamma(t)}_2\,dt
    \notag\\
    &\leq
    C_{K,p} \beta^{s+1} d_{\mathrm{eq}}(\theta,\theta^\star).
\end{align}
Using~\eqref{eq:local_tail_remainder_definition}, the left-hand side is exactly the quantity in \eqref{eq:locally_lipschitz_analytic_tail}, which proves the claim.
\end{proof}

\section{Sample complexity under algebraic likelihood geometry}
\label{sec:sample_complexity_algebric_likelihood_geometry}

\subsection{Proof of Proposition~\ref{prop:generic_rloc_rglob}}
\label{subsec:proof_prop_generic_rloc_rglob}

Let $r=r_{\mathrm{glob}}$ and consider the polynomial coordinate map $\alpha_{\leq r}:\Theta\to\mathbb{R}^{D_{\leq r}}$. Since $\alpha_{\leq r}$ is constant on observational equivalence classes, it induces a map on the quotient. On the generic regular stratum of $\Theta/\!\sim$, let $t=\dim(\Theta/\!\sim)$ denote the quotient dimension. By the definition of $r_{\mathrm{glob}}$, the induced map is generically injective, and hence its image has dimension $t$. By semi-algebraic stratification, applied to the polynomial map $\alpha_{\leq r}$, there exists a finite stratification such that its restriction to each stratum is smooth and has constant rank \cite[Proposition~8]{bolte2011generic}. Since the image has dimension $t$, this rank equals $t$ on a generic subset of the regular quotient stratum. Equivalently, for generic $\theta^\star$, $D\alpha_{\leq r}(\theta^\star) \big|_{W_{\theta^\star}}$ is injective. By the definition of the local identifiability order in Definition~\ref{def:rloc_likelihood}, this implies $r_{\mathrm{loc}}(\theta^\star)\leq r=r_{\mathrm{glob}}$. The exceptional strata are lower dimensional and therefore have measure zero, completing the proof.

\subsection{Proof of Proposition~\ref{prop:population_stability_local_global}}
\label{subsec:proof_prop_population_stability_local_global}

Let $s=r_{\mathrm{loc}}$, and $r=r_{\mathrm{glob}}$. We first prove the local statement. By the definition of $s=r_{\mathrm{loc}}$ in~\eqref{eq:rloc_equivalent_differential_injectivity}, for every $\theta^\star\in K$, $\ker D\alpha_{\leq s}(\theta^\star) \cap W_{\theta^\star} = \{0\}$. Thus, applying the constant-rank theorem on a local quotient slice transverse to the equivalence class, the map $\alpha_{\leq s}$ admits a locally Lipschitz inverse on the identifiable directions. By compactness of $K$, the corresponding constants can be chosen uniformly, so there exist $\eta>0$ and $C_{\mathrm{inv}}<\infty$ such that
\begin{align}
    d_{\mathrm{eq}}(\theta,\theta^\star)  \leq C_{\mathrm{inv}} \norm{\alpha_{\leq s}(\theta) - \alpha_{\leq s}(\theta^\star)}_2
\end{align}
whenever
$d_{\mathrm{eq}}(\theta,\theta^\star)\leq\eta$.
This proves~\eqref{eq:population_local_stability_alpha}.

For the global statement, since $K\subseteq\Theta_{\mathrm{gen}}$, generic identifiability and Lemma~\ref{lem:coordinate_faithfulness} imply
\begin{align}
    \alpha_{\leq r}(\theta)  = \alpha_{\leq r}(\theta^\star)    \quad\Longrightarrow\quad  d_{\mathrm{eq}}(\theta,\theta^\star)=0.
\end{align}
Since $\alpha_{\leq r}$ is continuous, Lemma~\ref{lem:compact_uniform_separation}, applied with the fixed radius $\eta>0$ above, gives $\gamma_\eta>0$ such that
\begin{align}
    d_{\mathrm{eq}}(\theta,\theta^\star)\geq\eta \quad\Longrightarrow\quad
    \norm{ \alpha_{\leq r}(\theta) -  \alpha_{\leq r}(\theta^\star) }_2 \geq  \gamma_\eta,
\end{align}
uniformly over $\theta^\star\in K$ and $\theta\in\Theta$.
This proves~\eqref{eq:population_global_separation_alpha}.

\subsection{Proof of Proposition~\ref{prop:likelihood_separation_local_global}}
\label{subsec:proof_prop_likelihood_separation_local_global}

Let $s=r_{\mathrm{loc}}$ and $r=r_{\mathrm{glob}}$, and let $\eta>0$ be given by Proposition~\ref{prop:population_stability_local_global}.
By Assumption~\ref{ass:nondegenerate_cross_order_likelihood_geometry}, $\mathbf{G}_r\succ0$. Since $\mathbf{G}_k$ is a principal submatrix of $\mathbf{G}_r$ for $k\leq r$, $\mathbf{G}_r\succ 0$ implies $\mathbf{G}_k\succ0$ for every $k\leq r$. By Lemma~\ref{lem:truncated_likelihood_gram_bound}, for every $k\leq r$ and $0<\beta\leq 1$,
\begin{align}
    \left\|\sum_{\ell=1}^k \beta^\ell \left[ \Phi_\ell(\theta)-\Phi_\ell(\theta^\star) \right] \right\|_{L^2(P_0)} \geq c_k\beta^k \left\| \alpha_{\leq k}(\theta)-\alpha_{\leq k}(\theta^\star) \right\|_2. \label{eq:gram_weighted_likelihood_lower}
\end{align}

For the local bound, Proposition~\ref{prop:population_stability_local_global} gives $\left\| \alpha_{\leq s}(\theta)-\alpha_{\leq s}(\theta^\star) \right\|_2 \geq C_{\mathrm{inv}}^{-1} d_{\mathrm{eq}}(\theta,\theta^\star)$ whenever $d_{\mathrm{eq}}(\theta,\theta^\star)\leq\eta$.
Hence~\eqref{eq:gram_weighted_likelihood_lower}, with $k=s$, implies
\begin{align}
    \left\| \sum_{\ell=1}^s \beta^\ell \left[ \Phi_\ell(\theta)-\Phi_\ell(\theta^\star) \right] \right\|_{L^2(P_0)} \geq c\beta^s d_{\mathrm{eq}}(\theta,\theta^\star).
    \label{eqn:app_G5}
\end{align}
By Lemma~\ref{lem:locally_lipschitz_analytic_tail}, after decreasing $\eta$ if necessary,
\begin{align}
    L_{\theta,\beta}-L_{\theta^\star,\beta} = \sum_{\ell=1}^s \beta^\ell \left[\Phi_\ell(\theta)-\Phi_\ell(\theta^\star) \right] + O_{L^2(P_0)}\!\left(\beta^{s+1}d_{\mathrm{eq}}(\theta,\theta^\star) \right).
    \label{eqn:app_G6}
\end{align}
The remainder is lower order, so substituting~\eqref{eqn:app_G5} into~\eqref{eqn:app_G6}, and for all sufficiently small $\beta$,
\begin{align}
    \left\| L_{\theta,\beta}-L_{\theta^\star,\beta} \right\|_{L^2(P_0)} \geq c_{\mathrm{loc}} \beta^{r_{\mathrm{loc}}} d_{\mathrm{eq}}(\theta,\theta^\star),
\end{align}
which proves~\eqref{eq:local_likelihood_separation_r_loc}.

For the global bound, Proposition~\ref{prop:population_stability_local_global} gives that if $d_{\mathrm{eq}}(\theta,\theta^\star)\geq\eta$ then $\left\| \alpha_{\leq r}(\theta)-\alpha_{\leq r}(\theta^\star) \right\|_2 \geq \gamma_\eta$. Thus~\eqref{eq:gram_weighted_likelihood_lower}, with $k=r$, yields
\begin{align}
    \left\| \sum_{\ell=1}^r \beta^\ell \left[ \Phi_\ell(\theta)-\Phi_\ell(\theta^\star) \right] \right\|_{L^2(P_0)} \geq c\beta^r.
\end{align}
The uniform order-$r$ likelihood-ratio expansion then gives, for all sufficiently small $\beta$,
\begin{align}
    \left\| L_{\theta,\beta}-L_{\theta^\star,\beta} \right\|_{L^2(P_0)} \geq c_{\mathrm{glob}}\beta^{r_{\mathrm{glob}}}
\end{align}
whenever $d_{\mathrm{eq}}(\theta,\theta^\star)\geq\eta$. This proves~\eqref{eq:global_likelihood_separation_r_glob}.

\subsection{Proof of Theorem~\ref{thm:global_to_local_upper_bound}}
\label{subsec:two_stage_upper_bound}

Let $r=r_{\mathrm{glob}}$ and $s=r_{\mathrm{loc}}$, and let $\eta>0$ be the uniform local radius from Proposition~\ref{prop:population_stability_local_global}, decreased if necessary so that Lemma~\ref{lem:locally_lipschitz_analytic_tail} holds on the same neighborhood.

We split the sample into two independent subsamples of sizes $n_1$ and $n_2$. By Theorem~\ref{thm:mom_upper_bound}, the first-stage estimator $\widetilde\theta$ satisfies $P_{\theta^\star,\beta}^{\otimes n_1} \left( d_{\mathrm{eq}}(\widetilde\theta,\theta^\star) > \eta/4 \right) \leq \delta/2$ provided $n_1 \gtrsim \beta^{-2r}\log(1/\delta)$.
Let
\begin{align}
    \mathcal{E}_1 \triangleq \left\{ d_{\mathrm{eq}}(\widetilde\theta,\theta^\star) \leq \eta/4 \right\}.
\end{align}
On $\mathcal{E}_1$, the local search set contains $\theta^\star$, and every $\theta\in\Theta_{\mathrm{loc}}(\widetilde\theta)$ satisfies $d_{\mathrm{eq}}(\theta,\theta^\star) \leq 3\eta/4$.

It remains to establish the population separation for the second-stage moment map.
For every such $\theta$,
\begin{align}
    \Gamma_{s,\beta}(\theta) - \Gamma_{s,\beta}(\theta^\star)
    &=
    \mathbb{E}_{Y \sim P_0}\!\left[ q_{\leq s}(Y) \bigl( L_{\theta,\beta}(Y)-L_{\theta^\star,\beta}(Y) \bigr) \right]
    \notag\\
    &=
    \mathbf{G}_s
    \begin{pmatrix}
        \beta\,[\alpha_1(\theta)-\alpha_1(\theta^\star)]\\
        \vdots\\
        \beta^s[\alpha_s(\theta)-\alpha_s(\theta^\star)]
    \end{pmatrix}
    +
    O\!\left( \beta^{s+1}d_{\mathrm{eq}}(\theta,\theta^\star) \right),    \label{eq:local_population_moment_expansion}
\end{align}
where the remainder bound follows from Lemma~\ref{lem:locally_lipschitz_analytic_tail} and Cauchy--Schwarz.
Since $\mathbf{G}_s\succ0$, Proposition~\ref{prop:population_stability_local_global} and $0<\beta\leq 1$ imply
\begin{align}
    \left\| \Gamma_{s,\beta}(\theta) - \Gamma_{s,\beta}(\theta^\star) \right\|_2 \geq b_K\beta^s d_{\mathrm{eq}}(\theta,\theta^\star)    \label{eq:local_population_moment_separation_s}
\end{align}
for all sufficiently small $\beta$, uniformly over the local search set and $\theta^\star\in K$.

Conditional on the first subsample and on $\mathcal{E}_1$, the second-stage search set is fixed and contains $\theta^\star$. By~\eqref{eq:local_population_moment_separation_s}, every $\theta\in\Theta_{\mathrm{loc}}(\widetilde\theta)$ satisfying $d_{\mathrm{eq}}(\theta,\theta^\star)>\varepsilon$ has population moment separation at least $\Delta = b_K\beta^s\varepsilon$. Hence, by the same minimum-distance argument as in Lemma~\ref{lem:minimum_distance_reduction},
\begin{align}
    P_{\theta^\star,\beta}^{\otimes n_2} \left( d_{\mathrm{eq}}(\widehat\theta_{\mathrm{loc}},\theta^\star)>\varepsilon \,\middle| \, \mathcal{E}_1 \right) \leq P_{\theta^\star,\beta}^{\otimes n_2} \left( \left\| \widehat{\Gamma}_{s,\delta/2} - \Gamma_{s,\beta}(\theta^\star) \right\|_2 \geq \frac{b_K\beta^s\varepsilon}{2} \right). \label{eq:local_mom_minimum_distance_reduction}
\end{align}
By independence of the two subsamples and the robust concentration bound~\eqref{eq:robust_likelihood_coordinate_concentration},
\begin{align}
    P_{\theta^\star,\beta}^{\otimes n_2} \left(\left\| \widehat{\Gamma}_{s,\delta/2} - \Gamma_{s,\beta}(\theta^\star) \right\|_2 > C_s \sqrt{\frac{V_s\log(2/\delta)}{n_2} } \right) \leq \frac{\delta}{2}.
\end{align}
Therefore, the conditional second-stage failure probability is at most $\delta/2$ whenever
\begin{align}
    C_s \sqrt{\frac{V_s\log(2/\delta)}{n_2} } \leq \frac{b_K\beta^s\varepsilon}{2},
\end{align}
or equivalently, $n_2 \gtrsim \beta^{-2s}\varepsilon^{-2}\log(1/\delta)$. Since the first-stage failure probability is also at most $\delta/2$, a union bound gives
$\sup_{\theta^\star\in K} P_{\theta^\star,\beta}^{\otimes n} ( d_{\mathrm{eq}}(\widehat\theta_{\mathrm{loc}},\theta^\star) > \varepsilon) \leq \delta$. Choosing $n_1$ and $n_2$ of the corresponding orders therefore gives
\begin{align}
    n  \gtrsim \left( \beta^{-2r_{\mathrm{glob}}} + \beta^{-2r_{\mathrm{loc}}}\varepsilon^{-2} \right) \log(1/\delta) \asymp \max\left\{ \beta^{-2r_{\mathrm{glob}}}, \beta^{-2r_{\mathrm{loc}}}\varepsilon^{-2}  \right\} \log(1/\delta),
\end{align}
which proves~\eqref{eq:global_to_local_upper_sample_complexity}.

\subsection{Proof of Lemma~\ref{lem:local_hard_alternatives}}
\label{subsec:proof_lem_local_hard_alternatives}

Let $s=r_{\mathrm{loc}}$ and fix $\theta_0\in\Theta_{\mathrm{reg}}$ as in the statement. Work in a local slice through $\theta_0$ representing the identifiable quotient near $[\theta_0]$. By the definition of $s$, there exists a nonzero tangent direction $h\in W_{\theta_0}$ such that $D\Phi_{\leq s-1}(\theta_0)[h]=0$, and $D\Phi_s(\theta_0)[h]\neq 0$. Since the rank of $\Phi_{\leq s-1}$ is locally constant on $\Theta_{\mathrm{reg}}$, the constant-rank theorem implies that its local level set through $\theta_0$ is a smooth submanifold whose tangent space contains $h$. Hence there exists a $C^1$ curve $t\mapsto\theta(t)$ in this level set such that $\theta(0)=\theta_0$, and $\dot\theta(0)=h$. After decreasing $t_0$ if necessary, the curve lies in $K$ for $|t|\leq t_0$.

By construction, $\Phi_{\leq s-1}(\theta(t)) = \Phi_{\leq s-1}(\theta_0)$. Moreover,
\begin{align}
    \Phi_s(\theta(t))-\Phi_s(\theta_0) = tD\Phi_s(\theta_0)[h]+o(|t|),
\end{align}
and therefore, since $D\Phi_s(\theta_0)[h]\neq0$, $ \norm{ \Phi_s(\theta(t))-\Phi_s(\theta_0) }_{L^2(P_0)} \asymp |t|$. Finally, because the curve lies in a local quotient slice and $\dot\theta(0)=h\neq0$, $d_{\mathrm{eq}}(\theta(t),\theta_0) \asymp |t|$. Taking common constants after decreasing $t_0$ proves \eqref{eq:local_hard_curve_distance}--\eqref{eq:local_hard_curve_informative_order}.

\subsection{Proof of Theorem~\ref{thm:global_to_local_lower_bound}}
\label{subsec:proof_thm_global_to_local_lower_bound}

Let $r=r_{\mathrm{glob}}$, and $s=r_{\mathrm{loc}}$. The global term is exactly the fixed-accuracy lower bound of Theorem~\ref{thm:lecam_lower_bound}, applied to the pair in the statement. Thus, for all sufficiently small fixed $\varepsilon>0$,
\begin{align}
    n^\star(\varepsilon,\delta;\beta) \geq c_{\mathrm{glob}} \beta^{-2r} \log(1/\delta).
    \label{eq:global_component_lower}
\end{align}

For the local term, let $\theta_0\in K$ and $t\mapsto\theta(t)$ be the curve from Lemma~\ref{lem:local_hard_alternatives}. Choosing $|t|\asymp\varepsilon$ gives, for all sufficiently small $\varepsilon$, $d_{\mathrm{eq}}(\theta(t),\theta_0)>2\varepsilon$, and $\Phi_{\leq s-1}(\theta(t)) = \Phi_{\leq s-1}(\theta_0)$, and $\norm{\Phi_s(\theta(t))-\Phi_s(\theta_0) }_{L^2(P_0)} \lesssim \varepsilon$. 
Finite-dimensional norm equivalence and
Lemma~\ref{lem:locally_lipschitz_analytic_tail} therefore give, for
every fixed finite $p$,
\begin{align}
    \norm{L_{\theta(t),\beta} - L_{\theta_0,\beta}}_{L^p(P_0)} \leq C_p\beta^s\varepsilon.
    \label{eq:local_pair_lr_upper}
\end{align}

Using the likelihood-ratio representation of $\chi^2$, H\"older's
inequality, and
Assumption~\ref{ass:reciprocal_likelihood_regularity},
\begin{align}
    D_{\mathrm{KL}}    \left(P_{\theta(t),\beta}  \,\middle\|\, P_{\theta_0,\beta} \right) \leq C\beta^{2s}\varepsilon^2.
\end{align}
Proposition~\ref{prop:lecam_two_point_general} applied to this pair yields $n^\star(\varepsilon,\delta;\beta) \geq c_{\mathrm{loc}} \beta^{-2s} \varepsilon^{-2} \log(1/\delta)$. Combining the global and local bounds proves $n^\star(\varepsilon,\delta;\beta) \geq c_0 \max \left\{ \beta^{-2r_{\mathrm{glob}}}, \beta^{-2r_{\mathrm{loc}}}\varepsilon^{-2} \right\} \log(1/\delta)$, which proves~\eqref{eq:global_to_local_lower_rate}.

\section{Applications to algebraic weak-signal perturbations}
\label{sec:applications_algebraic_perturbations}

\subsection{Proof of Proposition~\ref{prop:linear_perturbation_verification}}
\label{sec:proof_linear_perturbation_verification}

By Lemma~\ref{lem:linear_weak_input_likelihood_moments}, $\phi_k(y;\theta)=\langle Q_k(y),M_k(\theta)\rangle$. Hence each likelihood space $\operatorname{span}\{\Phi_k(\theta):\theta\in\Theta\}$ is contained in the finite-dimensional span of the components of $Q_k$, so Assumption~\ref{ass:finite_dimensional_likelihood_geometry} holds. If $q_{\leq r}$ is any basis of the corresponding likelihood-coordinate space, then $q_{\leq r}=B_rQ_{\leq r}$ for a fixed linear map $B_r$. Therefore, if $\widetilde{\Gamma}_{r,\beta}$ denotes the population map associated with $q_{\leq r}$,
\begin{align}
    \left\| \Gamma_{r,\beta}(\theta) - \Gamma_{r,\beta}(\theta^\star) \right\|_2 \geq \frac{ \| \widetilde{\Gamma}_{r,\beta}(\theta) - \widetilde{\Gamma}_{r,\beta}(\theta^\star) \|_2
    }{\|B_r\|_{\mathrm{op}}}.
\label{eq:channel_coordinate_population_separation_transfer}
\end{align}
Thus the population separation used in Theorem~\ref{thm:mom_upper_bound} transfers to the channel coordinates $Q_{\leq r}$.

It remains to verify the concentration condition.
By Assumption~\ref{ass:regular_linear_perturbation_model}, the fixed coordinates $Q_{\leq k}$ belong to $L^4(P_0)$ for every finite $k$, while the uniform bound on $s(z;\theta)$ implies that $\|L_{\theta,\beta}\|_{L^2(P_0)}$ is uniformly bounded over $\theta\in\Theta$ for all sufficiently small $\beta$. Hence, by H\"older's inequality, $\sup_{\theta\in\Theta} \sup_{0<\beta\leq\beta_0} \mathbb{E}_{P_{\theta,\beta}} \left[ \|Q_{\leq k}(Y)\|_2^2  \right] < \infty$. The corresponding centered second moments are therefore uniformly bounded, so the geometric median-of-means concentration bound applies. Together with~\eqref{eq:channel_coordinate_population_separation_transfer}, this proves the fixed-accuracy claim.

For the global-to-local result, since $s(z;\theta)$ is polynomial in $\theta$, each
$M_k(\theta)=\mathbb{E}[s(Z;\theta)^{\otimes k}]$ (and hence each induced likelihood-coordinate map $\alpha_k(\theta)$) is polynomial. Thus, Assumption~\ref{ass:algebraic_parameter_structure} holds. Moreover, Assumption~\ref{ass:regular_linear_perturbation_model} and the chain rule for $m_{\theta,\beta}(z)=m_0+\beta s(z;\theta)$ imply Assumption~\ref{ass:local_likelihood_regularity}. The same second-moment argument applies to $Q_{\leq r_{\mathrm{loc}}}$. Therefore all assumptions of Theorem~\ref{thm:global_to_local_upper_bound} hold, proving the second claim.

\subsection{Proof of Lemma~\ref{lem:dual_coordinates_signal_moments}}
\label{app:proof_dual_coordinates_signal_moments}

If~\eqref{eq:dual_coordinate_centering}-\eqref{eq:dual_coordinate_biorthogonality} hold, then changing measure to $P_0$ and using the likelihood expansion gives
\begin{align}
    \mathbb{E}_{Y \sim P_{\theta,\beta}}[\psi_k(Y)] = \mathbb{E}_{Y \sim P_0}[\psi_k(Y)] + \sum_{j\geq1}\beta^j \mathbb{E}_{Y \sim P_0} \left[ \psi_k(Y)\left\langle Q_j(Y),M_j(\theta)\right\rangle \right] = \beta^k M_k(\theta),
\end{align}
proving~\eqref{eq:direct_signal_moment_extraction}.

Conversely, suppose~\eqref{eq:direct_signal_moment_extraction} holds. Changing measure to $P_0$ gives $\mathbb{E}_{Y \sim P_0} \left[ \psi_k(Y)L_{\theta,\beta}(Y) \right] = \beta^k M_k(\theta)$. Setting $\beta=0$ proves~\eqref{eq:dual_coordinate_centering}, while comparing coefficients of $\beta^j$ yields
\begin{align}
    \mathbb{E}_{Y \sim P_0} \left[ \psi_k(Y)\left\langle Q_j(Y),M_j(\theta)\right\rangle \right]
    =
    \begin{cases}
        M_k(\theta), & j=k,\\
        0, & j\neq k.
    \end{cases}
\end{align}
By linearity, this extends to every $M\in \operatorname{span} \{M_j(\theta):\theta\in\Theta\}$, proving~\eqref{eq:dual_coordinate_biorthogonality}.

Finally, if
$\sum_{j=1}^r\langle Q_j(Y),M_j\rangle=0$,
multiplying by $\psi_k(Y)$ and applying~\eqref{eq:dual_coordinate_biorthogonality} gives $M_k=0$ for every $k$.
Hence the coordinates are jointly linearly independent, so the corresponding finite-dimensional cross-order Gram matrix is positive definite.

\subsection{Proof of Lemma~\ref{lem:mean_scaling_regular_likelihood_ratio_expansion}}
\label{app:proof_mean_scaling_regular_likelihood_ratio_expansion}

For fixed $z$ and $\theta$, the conditional likelihood ratio is
\begin{align}
    \frac{p_{\theta,\beta}(y\vert z)}{p_0(y)} = \exp \left\{ \left\langle \eta_{\theta,\beta}(z)-\eta_0, T(y) \right\rangle - A\bigl(\eta_{\theta,\beta}(z)\bigr) + A(\eta_0) \right\}.
\end{align}
By Assumption~\ref{assum:regular_exponential_family}, the natural parameters $\eta_{\theta,\beta}(z)$ remain in a fixed compact subset of $\operatorname{int}(\Omega)$ for all sufficiently small $\beta$. Hence the exponential-family likelihood ratio and its derivatives in $\beta$ have finite $L^p(P_0)$ norms, uniformly over $\theta\in\Theta$ and $z\in\operatorname{supp}(\mu)$, for every finite $p\ge 1$. Averaging over $Z$ preserves analyticity, so $\beta\mapsto L_{\theta,\beta}$ is real analytic as an $L^p(P_0)$-valued map, uniformly over $\theta\in\Theta$.

It remains to verify the reciprocal bound. Since $\eta_{\theta,\beta}(z)$ stays uniformly inside $\operatorname{int}(\Omega)$ and $\eta_{\theta,0}(z)=\eta_0$, there is $C<\infty$ such that $\left\| \eta_{\theta,\beta}(z)-\eta_0 \right\| \le C\beta$ uniformly over $\theta$ and $z$, for all sufficiently small $\beta$. Therefore, for every finite $p\ge 1$, the conditional likelihood ratios and their reciprocals are uniformly bounded in $L^p(P_0)$. Since
\begin{align}
    L_{\theta,\beta}(y) = \mathbb{E}_{Z \sim \mu} \left[ \frac{p_{\theta,\beta}(y\vert Z)}{p_0(y)} \right],
\end{align}
Jensen's inequality for the convex map $x\mapsto x^{-1}$ gives
\begin{align}
    L_{\theta,\beta}(y)^{-1} \le \mathbb{E}_{Z \sim \mu} \left[ \left( \frac{p_{\theta,\beta}(y\vert Z)}{p_0(y)} \right)^{-1} \right].
\end{align}
Thus, $\sup_{\theta\in\Theta} \| L_{\theta,\beta}^{-1} \|_{L^{2p}(P_0)} = O(1)$.
On the other hand, analyticity and $L_{\theta,0}=1$ imply
\begin{align}
    \sup_{\theta\in\Theta} \left\| L_{\theta,\beta}-1 \right\|_{L^{2p}(P_0)} = O(\beta).
\end{align}
Combining these two bounds with Hölder's inequality yields
\begin{align}
    \sup_{\theta\in\Theta} \left\| L_{\theta,\beta}^{-1}-1 \right\|_{L^p(P_0)}
    &=
    \sup_{\theta\in\Theta} \left\| L_{\theta,\beta}^{-1} \left( 1-L_{\theta,\beta} \right) \right\|_{L^p(P_0)}
    \notag \\
    &\le
    \sup_{\theta\in\Theta} \left\| L_{\theta,\beta}^{-1} \right\|_{L^{2p}(P_0)} \sup_{\theta\in\Theta} \left\| L_{\theta,\beta}-1 \right\|_{L^{2p}(P_0)} =  O(\beta).
\end{align}
This proves Assumption~\ref{ass:reciprocal_likelihood_regularity}.

\subsection{Linear mean-scaling latent-variable models}
\label{subsec:linear_mean_scaling_latent_variable_models}

An important specialization of Model~\ref{model:mean_scaling_exponential_family} is obtained when the latent signal contribution is linear in the unknown parameter. More precisely, suppose that $s(z;\theta) = {A}(z)\theta$, where ${A}(z):\mathbb{R}^m\to\mathbb{R}^d$ is a known linear operator. Then, the conditional mean parameter is $m_{\theta,\beta}(z) = m_0 + \beta {A}(z)\theta$. The corresponding order-$k$ signal-only moments take the form
\begin{align}
    M_k(\theta) = \mathbb{E}_{Z\sim\mu} \left[\bigl({A}(Z)\theta \bigr)^{\otimes k}\right].    \label{eq:linear_mean_scaling_signal_only_moments}
\end{align}

This specialization captures many of the principal latent-variable models studied in the weak-signal regime. The operator ${A}(Z)$ describes the latent transformation applied to the signal, whereas the exponential-family observation law determines the null-dependent statistical coordinates ${Q}_k(y)$. Thus, the latent identifiability problem is encoded by the tensors in~\eqref{eq:linear_mean_scaling_signal_only_moments}, independently of the specific observation family. Table~\ref{tab:linear_mean_scaling_models} lists representative examples.

\begin{table*}[t]
\centering
\begingroup
\setlength{\tabcolsep}{3.5pt}
\renewcommand{\arraystretch}{1.05}
\scriptsize

\caption{\textbf{Representative models covered by the linear mean-scaling latent-variable framework.}
Each row specifies the latent space $\mathcal{Z}$, the parameter, the linear map $A(\cdot)$, and the intrinsic observational equivalence relation. We use $\mathcal{T}_\ell$ for circular shifts, $J$ for reflection, $\Pi$ for projection, $P_\pi$ for permutation matrices, and $\rho(g)$ for a linear representation of a group element $g$. In particular, if $R\in\mathrm{SO}(n)$ or $R\in\mathrm{SO}(3)$, then $R$ is a rotation and $\rho(R)$ is its induced action on the parameter space. For cryo-EM with translations, $S_t$ denotes the in-plane shift operator. For patch-induced MRA, $E$ denotes zero-padding, $P_0$ and $P_1$ denote coordinate projections.}
\label{tab:linear_mean_scaling_models}

\begin{tabular}{|L{2.7cm}||L{1.80cm}|L{2.5cm}|L{2.65cm}|L{3.25cm}|L{2.35cm}|}
\hline
\textbf{Application} &
\textbf{Latent $Z$} &
\textbf{Parameter $\theta^\star$} &
\textbf{Linear map $A(Z)$} &
\textbf{Equivalence / symmetry ($[\theta^\star]$)} &
\textbf{Min.\ moments $r$ (typical)} \\
\hline\hline

\textbf{Spherical GMM} &
$\{1,\ldots,K\}$ &
$(\theta^\star_1,\ldots,\theta^\star_K)\in(\mathbb{R}^d)^K$ &
$A(z)\theta=\theta_z$ &
Label permutation $\theta\sim(\theta_{\pi(1)},\ldots,\theta_{\pi(K)})$ &
$K \geq 3:\; 3$ \cite{hsu2013learning,anandkumar2014tensor} $\qquad K = 2:\; 2$ \\
\hline

\textbf{Mixture of linear regressions $^\dagger$} &
$\{1,\ldots,K\}$ &
$(\theta^\star_1,\ldots,\theta^\star_K)\in(\mathbb{R}^d)^K$ &
$A(z;x)\theta=\langle x,\theta_z\rangle$  &
Label permutation $\theta\sim(\theta_{\pi(1)},\ldots,\theta_{\pi(K)})$ &
$K \geq 3:\; 3$ \cite{sedghi2016provable,chaganty2013spectral} $\qquad K = 2:\; 2$ \\
\hline

\textbf{Cyclic MRA} &
$\mathbb{Z}_d$ &
$\theta^\star\in\mathbb{R}^d$ &
$A(\ell)=\mathcal{T}_\ell$ &
$\mathbb{Z}_d\cdot\theta$ &
$3$ \cite{bendory2017bispectrum,bandeira2023estimation} \\
\hline

\textbf{Dihedral MRA} &
$D_{2d}$ &
$\theta^\star\in\mathbb{R}^d$ &
$A(g)\in\{\mathcal{T}_\ell,\mathcal{T}_\ell J\}$ &
$D_{2d}\cdot\theta$ &
$3$ \cite{bendory2022dihedral,edidin2025orbit} \\
\hline

\textbf{Projected MRA} &
$\mathbb{Z}_d$ &
$\theta^\star\in\mathbb{R}^d$, $d$ odd &
$A(\ell)=\Pi\mathcal{T}_\ell$ &
$D_{2d}\cdot\theta$ &
$3$ \cite{bandeira2023estimation,balanov2026projected}  
\\
\hline

\textbf{Patch-induced MRA} &
$\mathbb{Z}_{2L}\times\mathbb{Z}_{2L}$ &
$\theta^\star\in\mathbb{R}^L$ &
$A(g)=P_0\mathcal{T}_{g_1}E+P_1\mathcal{T}_{g_2}E$ &
Generically trivial &
Unknown \cite{abraham2026sample} \\
\hline

\textbf{Coordinate sign-flip model} &
$\{\pm1\}^d$ &
$\theta^\star\in\mathbb{R}^d$ &
$A(s)=\mathrm{diag}(s)$ &
$\theta\sim s\odot\theta$ &
$2$ \\
\hline

\textbf{Permutation group} &
$S_d$ &
$\theta^\star\in\mathbb{R}^d$ &
$A(\pi)=P_\pi$ &
$\theta\sim P_\pi\theta$ &
$d$ \cite{bandeira2023estimation} \\
\hline

\textbf{Orbit recovery under $\mathrm{SO}(n)$} &
$\mathrm{SO}(n)$ &
bandlimited $n$-dimensional volume &
$A(R)=\rho(R)$ &
$\theta\sim\rho(R)\theta$ &
$3$ \cite{bandeira2023estimation,fan2024maximum,bendory2025orbit} \\
\hline

\textbf{$\mathbb{S}^2$ registration} &
$\mathrm{SO}(3)$ &
bandlimited $f^\star:\mathbb{S}^2\to\mathbb{R}$ &
$A(R)=\rho(R)$ &
$f\sim\rho(R)f$ &
$3$ \cite{bandeira2023estimation,fan2024maximum} \\
\hline

\textbf{Rigid motion} &
$\mathrm{SE}(n)$ &
bandlimited $n$-dimensional volume  &
$A(g) = g \cdot f$,  $g \in \mathrm{SE}(n)$ &
Global rigid motion &
$3$--$5$ \cite{balanov2025orbit} \\
\hline

\textbf{Cryo-EM without translations} &
$\mathrm{SO}(3)$ &
bandlimited 3D volume &
$A(R)=\Pi\rho(R)$ &
Global $\mathrm{O}(3)$ ambiguity &
$3$ \cite{bandeira2023estimation,fan2024maximum,balanov2026provable} \\
\hline

\textbf{Cryo-EM with translations} &
$\mathrm{SO}(3)\times\mathbb{R}^2$ &
bandlimited 3D volume &
$A(R,t)=S_t\Pi\rho(R)$ &
Global $\mathrm{E}(3)$ ambiguity &
Unknown \\
\hline

\textbf{Heterogeneous orbit recovery} &
$\{1,\ldots,K\}\times G$ &
$(\theta^\star_1,\ldots,\theta^\star_K)$ &
$A(c,g)\theta=\rho(g)\theta_c$ &
$\theta\sim(\rho(g)\theta_{\pi(1)},\ldots,\rho(g)\theta_{\pi(K)})$ &
Unknown \\
\hline
\end{tabular}

\vspace{2pt}
\parbox{\textwidth}{\scriptsize
$^\dagger$ The framework in~\eqref{eq:conditional_observation_density_intro} is written without observed covariates to simplify notation. Mixtures of linear regressions fit a straightforward covariate-dependent extension in which $A(z)$ is replaced by $A(z;x)$, under standard moment and nonsingularity assumptions on the covariate distribution.}

\endgroup
\end{table*}

\end{appendices}

\end{document}